\documentclass[12pt]{article}
\usepackage{graphicx}
\usepackage{float}
\usepackage{amsmath}
\usepackage{amsthm}
\usepackage{amssymb}
\usepackage{amscd}
\usepackage{amsfonts}
\usepackage{makeidx}
\usepackage{enumerate}          %% PARA PODER HACER "LOS ENUMERATE"
    \newtheorem{teo}{Theorem}[section]
    \newtheorem{lem}[teo]{Lemma}
    \newtheorem{prop}[teo]{Proposition}
    
    \newtheorem{defn}[teo]{Definition}
    \newtheorem{obs}[teo]{Remark}

\begin{document}

%\title{The explicit formula for the diffracted wave in the Neumann Scattering Problem}

\title{An explicit formula for the nonstationary diffracted wave  scattered on a NN-wedge}
\author{Anel Esquivel Navarrete, Anatoli Merzon}

\date{\today}

%\date{Monday, 30 de Enero de 2013}                      %%%%%%%%%%%%     PARA PONER FECHA MANUAL
%\date{Tuesday, 31 de Enero de 2013(2)}                     %%%%%%%%%%%%     PARA PONER FECHA MANUAL
%\date{Wednesday, 09 de Enero de 2013}                %%%%%%%%%%%%     PARA PONER FECHA MANUAL
%\date{Thursday, 06 de Enero de 2013}                     %%%%%%%%%%%%     PARA PONER FECHA MANUAL
%\date{Friday, 9 de Enero de 2013}                 %%%%%%%%%%%%     PARA PONER FECHA MANUAL
%\date{Saturday, 01 de Enero de 2013}                     %%%%%%%%%%%%     PARA PONER FECHA MANUAL
%\date{Sunday, 16 de Enero de 2013}                    %%%%%%%%%%%%     PARA PONER FECHA MANUAL

\maketitle

{\abstract We  consider two dimensional nonstationary scattering of plane waves by a NN-wedge. We prove the existence and uniqueness of a solution to the corresponding mixed problem and we give an explicit formula for the solution. Also the Limiting Amplitude Principle is proved and a rate of convergence to the limiting amplitude is obtained.}
%%                                    \abstract{We continue with the investigation of the nonstationary scattering by wedges [1]-[4].
%%In this paper we consider a nonstationary scattering of plane waves by a wedge with both boundary conditions of the Neumann type. We give he exact formula for the diffracted wave in the Neumann Scattering Problem.

%%%%%%%%%%%%%%%%%%%%%%%%%%%%%%%%%%%%%%%%%%%%%%%%%%%%%%%%%%%%%%%%%%%%%%%%%%%%%%%%%%%%%%%%%%%%%%%%%%%%
%%%%%%%%%%%%%%%%%%%%%%%%%   SECCIÓN
%%%%%%%%%%%%%%%%%%%%%%%%%%%%%%%%%%%%%%%%%%%%%%%%%%%%%%%%%%%%%%%%%%%%%%%%%%%%%%%%%%%%%%%%%%%%%%%%%%%%

\section{Introduction}
In this paper we justify an exact formula for the cylindrical wave diffracted by the edge of a wedge when a nonstationary plane wave (see (\ref{uin})) impinges on an arbitrary two dimensional wedge. This cylindrical wave (more exactly its amplitude) arised for the first time in the paper by Sommerfeld \cite{som} where stationary diffraction of a plane wave by a half plane was studied. While many papers are devoted to stationary scattering by a wedge  (see \cite{BLG}-\cite{ob}), the nonstationary diffraction is studied considerably less. The diffraction of  nonstationary incident Heaviside type waves was considered in
\cite{sov}-
%, \cite{kell},  \cite{petra}, \cite{bor},
\cite{rot}. A systematic mathematical analysis of the scattering of nonstationary plane harmonic waves in the cases of DD and DN problems was undertaken in \cite{kmm}-\cite{mm}. In this paper we analyze the NN-scattering, we prove the Limiting Amplitude Principle and we find the rate of convergence to the Limiting Amplitude. Also we compare our results with the corresponding results of \cite{hew} where the case of a wedge with the magnitude of $2\pi$ (half plane) is considered.

Let
\begin{eqnarray}\label{uin}
u_{in}(y,t) := e^{i(\mathbf{k_0}\cdot y - \omega_0 t)}f(t-\mathbf{n_0}\cdot y),\qquad y\in\mathbb{R}^2,\ t\in\mathbb{R},
\end{eqnarray}
where
%$\omega_0$ is the frequency of the impinging  wave $u_{in}$ and $n_0$ is the unit vector corresponding wave vector
%\begin{equation}\label{n0}
$\mathbf{n_0}=(\cos\alpha,\sin\alpha)\in\mathbb{R}^2$, %\qquad
$\mathbf{k_0}=\omega_0 \mathbf{n_0}\in\mathbb{R}^2$, %\qquad
$\omega_0>0$. Here $f$ is a profile function of the wave that is assumed belonging to $C^{\infty}(\mathbb{R})$ and for some $0<s_0<1$ it has the form
\begin{equation}\label{f}
f(s) = \left \{ \begin{array}{ll}
								 0, & \text{if }s\leq 0,\\
								 1, & \text{if }s\geq s_0.
								\end{array} \right.
\end{equation}
We consider two-dimensional plane wave scattering scattering by the wedge
\begin{equation*}\label{W}
W:=\left\{y=(y_1,y_2)\in\mathbb{R}^2 : y_1=\rho\cos\theta, y_2=\rho\sin\theta,\rho\geq 0, 0\leq\theta\leq\phi\right\}
\end{equation*}
of the magnitude $\phi\in(0,\pi)$ (see Figure 1). In Section \ref{secLimAm} we also consider the Heaviside function ($s_0=0$) by taking the limit $s_0\rightarrow 0$.
 %in the obtained formulas.

We suppo\-se that the wedge is of the NN-type, that means that the corresponding boun\-dary conditions are of the Neumann type  on %the
 both
% of its
 sides. The exact solution of the scattering problems for the wedges of DD-type (both  boundary conditions are of the Dirichlet type)  and DN-type (one boundary condition is of the Dirichlet type and the other is of the Neumann type) have been obtained in \cite{kmm}-\cite{mm}; also the  Limiting Amplitude Principle for  these  problems was  proved in both cases. The exact  solutions  are

%\vspace{1cm}
%%%%%%%%%%%%%%%%%%%%%%%%%%%%%%%%%%%%%%%%%%%%%%%%%%%%%%%%%%
%%%%%%%%%%%%%%%%%%%%%%%%%%%%%%%%%%%%%%%%%%%%%%%%%%%%%%%%%%
%%%%%%%%%%%%%%%%%%%%%%%%%%%%%%%%%%%%%%%%%%%%%%%%%%%%%%%%%%
%\\\\\\\\\\\\\\\\\\
\setlength{\unitlength}{0.9mm}
\hspace{1cm} \begin{picture} (110,90)
 %os' y_1
 \put(100,35) {\vector (1,0){39}} \put(142,34){$y_1$}

  \put (55,35){\line(1,0){50}}

 %os' y_2
 \put
 (55,35){\vector(0,1){50}}  \put(57,85){$y_2$}

 %storona Q_2
  \put(55,35){\line(2,1){65}}

 %nadpis' W
  \put (125,45){W}

%nadpis' Q
  \put (20,55){Q}

%fronti voln
  \put(4,47){\line(1,-4){10}}
  \put(-3,45){\line(1,-4){10}}
  \put(11,49){\line(1,-4){10}}
  \put(19,51){\line(1,-4){10}}
  \put(26,53){\line(1,-4){10}}
  \put(34,55){\line(1,-4){10}}
  \put(42,57){\line(1,-4){10}}

%Oboznachhenie nachhala koordinat.
  \put(55,31){0}

%pryamaya liniya , na kotoroy ldezhit  vektor n_o
  \put (3,22){\line(4,1){60}}

%vector \mathbf{n_0}
  \put(55,35) {\vector (4,1){25}}

%duga \alpha
 \qbezier(65,35)(65.5,36)(64.5,37)

%oboznachdenie \alpha

 \put(67,35.5){$\alpha$}

 %duga \vp
 \qbezier(95,35)(98,39)(96,44) \qbezier(96,44)(95,46)(94,48)
 \qbezier(94,48)(95,46)(94,48) \qbezier(94,48)(93,49)(92,50)
 \qbezier(92,50)(91,51)(90,52)
 \qbezier(90,52)(89.5,52.4)(89.5,52.4)
 %oboznachenie vektora \mathbf{n_0}
  \put(81,41){$\mathbf{n_0}$}
 %nadpis' \phi

 \put (97,44){$\phi$}

 %nadpis' Q_1
 \put (127,29){$Q_1$}

 %nadpis' Q_2

 \put (115,70){$Q_2$}

 % el vector n_2
%  \put(100,57.5) {\vector (1,-2){5}} \put(108,51){$\mathbf{n_2}$}

 %oboznachenie -\pi/2

 %volna sinusoida

  \qbezier(3,22)(6, 26)(10,24)
 \qbezier(10,24)(13.5,21.5)(17,26)
 \qbezier(17,26)(20.5, 29.5)(24.,28)
 \qbezier(24.,28)(28.5, 25) (31.5, 29.5)
 \qbezier(31.5,29.5)(34.5,33.5) (39,32)
 \qbezier(39,32) (43.5,30,5)(47.5,33.5)
 \qbezier(47.5,33.5) (50.5,35)(54.5,35)

 %final'noe zakruglenie

 %\qbezier(40,40)(42.44,9)(44.45,39.2)
 %\qbezier(44.45,39.2)(46.4,40)(48,42)
 %\qbezier(48,42)(49.6,44)(51.55,44.8)
 %\qbezier(51.55,44.8)(53.6,45)(56,44)
 %\qbezier(56,44)(58.4,43)(60.45,43.2)
 %\qbezier(60.45,43.2)(62.4,44)(64,45)

%%%%%%%%%%%%%%%%%%%%%%

%critical direcciones $\theta_1$ y $\theta_2$

%direccion $\theta_1$ y oboznachenie \theta_1
\multiput(55,35)(5,4){12}{\line(5,4){4}}
\put(116,83){$\theta=\theta_1$}

%direccion $\theta_2$ y oboznachenie \theta_2
\multiput(55,35)(5,-1.25){13}{\line(4,-1){4}}
\put(120,17){$\theta=\theta_2$}

 \end{picture}

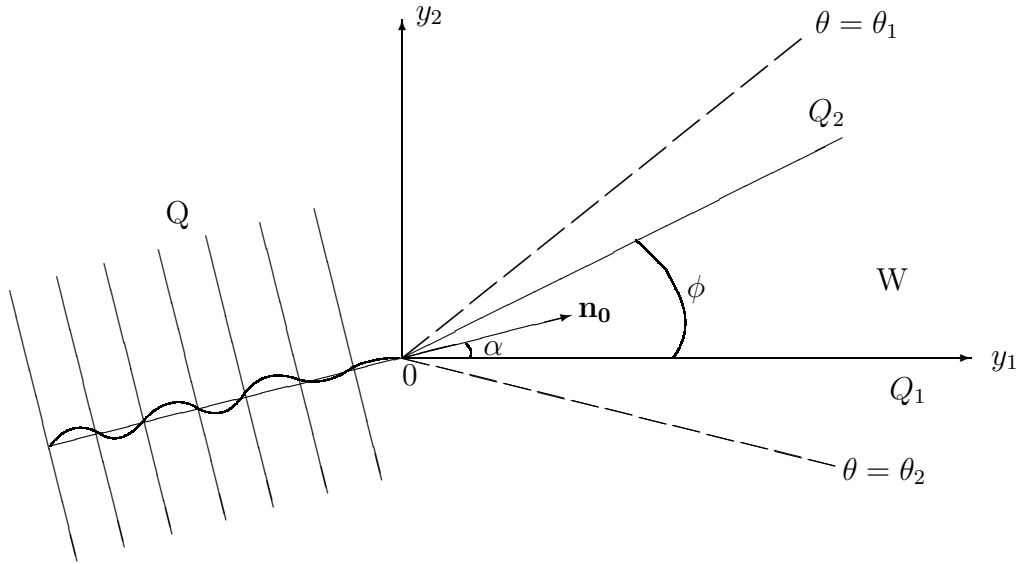
 \centerline{Figure 1. The impinging plane wave}
%%%%%%%%%%%%%%%%%%%%%%%%%%%%%%%%%%%%%%%%%%%%%%%%%%%%%%%%%%
%%%%%%%%%%%%%%%%%%%%%%%%%%%%%%%%%%%%%%%%%%%%%%%%%%%%%%%%%%
%%%%%%%%%%%%%%%%%%%%%%%%%%%%%%%%%%%%%%%%%%%%%%%%%%%%%%%%%%
\vspace{0.4cm}

\noindent decomposed into the sum of the impinging wave $u_{in}$, a reflected wave $u_r$ and a wave diffracted by the edge of the wedge $u_d$ (see (77) in \cite{la} and Lemma 15.1 in \cite{mm}), (see Figure 2).

%%%%%%%%%%%%%% COPIAR %%%%%%%%%%%%%%%%%%

%\begin{figure}[H]
%  \centering \hspace{2cm}
%\includegraphics[width=11cm, height=9cm]{Figure2}
% \centerline{Figure 2. A total field. }
  %\label{fig:ejemplo}
%\end{figure}

%\newpage
%%%%%%%%%%%%%%%%%%%%%%%%%%%%%%%%%%%%%%%%%
In Figure 2 we present schematically the diffracted, reflected and incident waves. The straight line to the right, which intersects the angle, is the front of the incident wave at time $t>0$. The segments of the straight lines connecting the sides of the angle with the directions $\theta=\theta_{1,2}$ are fronts of the  waves reflected by the sides according to the geometrical optic. The circles are fronts of the wave diffracted by the vertex. This latter wave has discontinuities on the lines $\theta=\theta_{1,2}$ which are compensated precisely by the reflected waves so that the total field is continuous.

The reflected wave $u_r$ is the result of the optical reflections and it is expressed explicitly from the boundary conditions (see formulas (9) and (1.9) in \cite{la} and \cite{mm}, res\-pecti\-vely). Thus, the greatest interest and difficulty presents the wave $u_d$  difracted by the edge of the wedge. The exact formulas for $u_d$ for the  DD and DN scattering  problems were obtained in  \cite{la} and \cite{mm}. The derivation was based on the representation of ${u}_d(y,t)$ as the inverse Fourier transform $F_{\omega\rightarrow t}^{-1}$ of the ``stationary" wave $\widehat{u}_d(y,\omega)$, $\omega\in\mathbb{C}^+$. This ``stationary" wave is a ``diffracted" part of the solution to the corresponding sta\-tiona\-ry pro\-blem with a parameter $\omega\in\mathbb{C}^+$ which appears after the Fourier-Laplace transform of the non-stationary problem, that is
\begin{eqnarray}\label{ud}
  u_d(y,t) &=& F_{\omega\rightarrow t}^{-1}\big[\widehat{u}_d(y,\omega)\big], \qquad y\in\mathbb{R}^2\setminus W, \ t\in\mathbb{R}.
\end{eqnarray}
In \cite{la}, formulas (76), (41) and in \cite{mm} formulas (12.13) and (12.5) a representation for (\ref{ud}) has been obtained
%and in the polar coordinates $y_1=\rho\cos\theta$, $y_2=\rho\sin\theta$, $\widehat{u}_{d}(\rho,\theta,\omega)$ takes the form
%(see (76), (41) in \cite{la}  and (12.13), (12.5) in \cite{mm})
\begin{eqnarray}\label{udc}
\widehat{u}_{d}(\rho,\theta,\omega)
&=& \dfrac{i}{4\Phi}\ \widehat{g}(\omega) \int\limits_{\mathcal{C}_0}
    e^{-\rho\omega\sinh\beta} H(\beta+i\theta) \ \mathrm{d}\beta, \qquad \omega\in\mathbb{C}^+,
\end{eqnarray}
(we use the Fourier transform in the form
\begin{eqnarray}\label{TF}
\left.
F_{t\rightarrow \omega}[u(t)] = \widehat{u}(\omega) :=
  \int\limits_{-\infty}^{+\infty}e^{i \omega t} u(t)\ \mathrm{d}t\ \right).
\end{eqnarray}
Here
%\begin{eqnarray}\label{g}
  $\widehat{g}(\omega) := \widehat{f}(\omega-\omega_0)$, $\omega\in\overline{\mathbb{C}^+}$,
%\end{eqnarray}
is the Fourier-Laplace transform or the Complex Fourier transform of the function $e^{-i\omega_0 t} f(t)$ (see Lemma \ref{propg}, i)) and
%In fact, in \cite{la} and \cite{mm}, the function $H$ is defined for $\beta\in\mathbb{C},\ \pi <\Phi\leq 2\pi$, as follows
\begin{equation}\label{H}
H(\beta,\Phi)
:=    \coth\left[q \left(\beta + i\ \dfrac{\pi}{2} - i\alpha\right)   \right]
    - \coth\left[q \left(\beta- i\ \displaystyle\frac{3\pi}{2} + i\alpha \right) \right],
\end{equation}
for the DD-wedge and
\begin{eqnarray}\label{H1V}
H(\beta,\Phi) &:=&
               \frac{1}{\sinh\left[q \left(\beta + i\ \frac{ \pi}{2} - i\alpha \right) \right]}
          \ +\ \frac{1}{\sinh\left[q \left(\beta - i\ \frac{3\pi}{2} + i\alpha \right) \right]}, \quad \beta\in\mathbb{C},
\end{eqnarray}
for the DN-wedge;  %(see (20) in \cite{la} and (2.16) in \cite{mm}), with
%\begin{eqnarray}\label{q}
%q &:=& \dfrac{\pi}{2\Phi}, \qquad \Phi = 2\pi-\phi,  \qquad \pi <\Phi\leq 2\pi,
%\end{eqnarray}
%\begin{eqnarray}\label{q}
$q := \dfrac{\pi}{2\Phi}$, % \qquad
$\Phi = 2\pi-\phi$, % \qquad
$\pi <\Phi\leq 2\pi$
%\end{eqnarray}
and $\mathcal{C}_0$ is the counter-clockwise directed  contour
\begin{eqnarray}\label{C0}
  \mathcal{C}_0 &:=& \gamma_1 \cup \gamma_2
\end{eqnarray}
with
%\begin{eqnarray}\label{ga12}
     $\gamma_1 := \mathbb{R} - i\frac{\pi}{2}$, % \qquad \qquad %\left\{a-i\frac{\pi}{2}:a\in\mathbb{R}\right\},
and  $\gamma_2 := \mathbb{R} - i\frac{5\pi}{2}$,
%\end{eqnarray}
 (see Figure 3  corresponding to the case $\mathrm{Re\ }\omega >0$). Also the regions where the  function $e^{-\rho\omega\sinh\beta}$  decays  are hatched in Figure 3. Note that in (\ref{udc}) $\theta\neq\theta_1,\theta_2$ where
\begin{equation}\label{T1T2}
\theta_1:= 2\phi-\alpha,\qquad \theta_2:=2\pi-\alpha,
\end{equation}
since the function $H$
%given by (\ref{H}) and (\ref{H1V})
has poles in $\beta=0$ for these critical values and the integral in (\ref{udc}) does not converge. On the critical rays
 $l_k:=\{(\rho\cos\theta_k,\rho\sin\theta_k)\in\mathbb{R}^2:\rho>0\}$, $k=1,2$, functions $u_d$ and $u_r$ are discontinuous but their sum,  $u$ is continuous. In \cite{la}, \cite{mm}, the calculation of the inverse Fourier-Laplace transform (\ref{ud}) has been done and the following representation for the diffracted waves in the cases of DD and DN problems were obtained in \cite{la}, (90) and in \cite{mm}, (15.1):
\begin{eqnarray}\label{ud2}
u_{d}(\rho,\theta,t)
&=& \dfrac{ie^{-i\omega_0 t}}{4\Phi}  \int\limits_{\mathbb{R}}
    e^{i\rho\omega_0\cosh\beta} Z(\beta,\theta) f(t-\rho\cosh\beta) \ \mathrm{d}\beta.
\end{eqnarray}
Here the function $Z$ is a combination of $H$, (see formulas (88)  and (15.10) in \cite{la} and \cite{mm}, respectively).

%%%%%%%%%%%%%%%%%%%%%%%%%%%%%%%%%%%%%%%%%%%%%%%%%%%%%%%%%%%%%%%%%%%%%%%%%%%%%%%%%%%%%%%%%%%%%
%%%%%%%%%%%%%%%%%%%%%%%%%%%%%%%%%%%%%%%%%%%%%%%%%%%%%%%%%%%%%%%%%%%%%%%%%%%%%%%%%%%%%%%%%%%%%
%%%%%%%%%%%%%%%%%%%%%%%%%%%%%%%%%%%%%%%%%%%%%%%%%%%%%%%%%%%%%%%%%%%%%%%%%%%%%%%%%%%%%%%%%%%%%
%%%%%%%%%%%%%%        FIGURA 3

%\newpage
{\scriptsize
\setlength{\unitlength}{1mm}
\begin {picture} (200,130)

%tochki koordinatnie na vert. osi

\put(80,95) {\circle*{0.7}}
\put(80,80) {\circle*{0.7}}
\put(80,87.5) {\circle*{0.7}}
\put(80,95) {\circle*{0.7}}
\put(80,72.5) {\circle*{0.7}}
\put(80,27.5) {\circle*{0.7}}
%eto -9\pi/4
\put(80,20) {\circle*{0.7}}

 \put(80,12.5) {\circle*{0.7}}
 \put(80,110) {\circle*{1}}
%Oboznacheniya tochek na vert. osi

\put(80.5,95){$0$}
\put(81,110){$\frac{\pi}{2}$}
\put(81,87.5){$-\frac{\pi}{4}$}
\put(81,72.5){$-\frac{3\pi}{4}$}
 \put(81,27.5){$-\frac{9\pi}{4}$}
\put(81,12.5){$-\frac{11\pi}{4}$}
 \put(81,65){$-\pi$}
\put(81,50){$-\frac{3\pi}{2}$}\put(80,5) {\circle*{0.7}}
\put(81,5){$-3\pi$} \put(80,-10){\circle*{0.7}}
\put(81,-10){$-\frac{7\pi}{2}$}

%Oboznacheiya tochek na vert. osi -5\pi/2 y -\pi/2
 \put(81,18.5){$-\frac{5\pi}{2}$}
 \put(81,79){$-\frac{\pi}{2}$}

%asimptoti krivij
%\put (0,110){\line(1,0){160}}
%\put (0,50){\line(1,0){80}}
 %\put(0,-10){\line(1,0){80}}

%pervaya krivaya \gamma(0)
% \qbezier(80,95)(85,102)(90,104)
%Razriv krivoy dlia nulia
\qbezier(82,98)(85,102)(90,104)
\qbezier(90,104)(100,105.5)(120,106.5)
\qbezier(120,106.5)(140,107)(150,107.3)
\qbezier(80,95)(75,88)(70,86)
 \qbezier(70,86)(60,84.5)(40,83.5)
\qbezier(40,83.5)(20,83)(10,82.7)
%shtrijovka pervoy polosi v pravoy chasti
%razriv pervoy shtrijovki.
\put(82,68){\line(0,1){3}}
\put(82,76){\line(0,1){3}}
\put(82,83){\line(0,1){3}}
%\put(82,91){\line(0,1){7}}
\put(82,91){\line(0,1){3}}
%razriv vtoroy shtrijovki
% \put(84,69){\line(0,1){30}}
 \put(84,70){\line(0,1){2}}
 \put(84,76.5){\line(0,1){2}}

 \put(84,83){\line(0,1){3}}
 \put(84,91){\line(0,1){9}}

%\put(86,71.5){\line(0,1){30}}
%razriv tret'ey shtrijovki.
%\put(86,71.5){\line(0,1){30}}

%\put(86,71.5){\line(0,1){1}}
%\put(86,77){\line(0,1){1}}

\put(86,92){\line(0,1){10}}

 \put(88,73){\line(0,1){30}}
\put(90,74){\line(0,1){30}}
 \put(92,74){\line(0,1){30}}
\put(94,74.5){\line(0,1){30}}
 \put(96,75){\line(0,1){30}}
\put(98,75){\line(0,1){30}}
 \put(100,75){\line(0,1){30}}
\put(102,75.6){\line(0,1){30}}
 \put(104,75.5){\line(0,1){30}}
 \put(106,76){\line(0,1){30}}
 \put(108,76.2){\line(0,1){30}}
 \put(110,76.2){\line(0,1){30}}
 \put(110,76.3){\line(0,1){30}}
 \put(112,76.4){\line(0,1){30}}
  \put(114,76.4){\line(0,1){30}}
   \put(116,76.5){\line(0,1){30}}
    \put(118,76.5){\line(0,1){30}}
     \put(120,76.6){\line(0,1){30}}

 \put(122,76.6){\line(0,1){30}}
  \put(124,76.7){\line(0,1){30}}
   \put(126,76.7){\line(0,1){30}}
    \put(128,76.8){\line(0,1){30}}
     \put(130,76.8){\line(0,1){30}}
 \put(132,76.9){\line(0,1){30}}
  \put(134,76.9){\line(0,1){30}}
   \put(136,77){\line(0,1){30}}
    \put(138,77){\line(0,1){30}}
     \put(140,77.1){\line(0,1){30}}
 \put(142,77.1){\line(0,1){30}}
  \put(144,77.2){\line(0,1){30}}
   \put(146,77.2){\line(0,1){30}}
    \put(148,77.3){\line(0,1){30}}
     \put(150,77.3){\line(0,1){30}}
%shtrijovka pervoy polosi v levoy chasti
\put(78,62){\line(0,1){30}} \put(76,60.5){\line(0,1){30}}
\put(74,58.5){\line(0,1){30}}
 \put(72,57){\line(0,1){30}}
\put(70,56){\line(0,1){30}}
 \put(68,56){\line(0,1){30}}
\put(66,55.5){\line(0,1){30}}
 \put(64,55){\line(0,1){30}}
\put(62,55){\line(0,1){30}}
%razriv 10 shtrijovki dlia kontura y bukvi \gamma_1
\put(60,55){\line(0,1){30}}
% \put(60,80){\line(0,1){5}}
\put(58,54.4){\line(0,1){30}}
 \put(56,54.5){\line(0,1){30}}
 \put(54,54){\line(0,1){30}}
 \put(52,53.8){\line(0,1){30}}
 \put(50,53.8){\line(0,1){30}}
 \put(48,53.8){\line(0,1){30}}
 \put(46,53.6){\line(0,1){30}}
  \put(44,53.6){\line(0,1){30}}
   \put(42,53.5){\line(0,1){30}}
    \put(40,53.5){\line(0,1){30}}
     \put(38,53.4){\line(0,1){30}}
 \put(36,53.4){\line(0,1){30}}
  \put(34,53.3){\line(0,1){30}}
   \put(32,53.3){\line(0,1){30}}
    \put(30,53.2){\line(0,1){30}}
     \put(28,53.2){\line(0,1){30}}
 \put(26,53.1){\line(0,1){30}}
  \put(24,53.1){\line(0,1){30}}
   \put(22,53){\line(0,1){30}}
    \put(20,53){\line(0,1){30}}
     \put(18,52.9){\line(0,1){30}}
 \put(16,52.9){\line(0,1){30}}
  \put(14,52.8){\line(0,1){30}}
   \put(12,52.8){\line(0,1){30}}
    \put(10,52.7){\line(0,1){30}}
%vtoraya liniya

%Razriv okolo -3\pi74
\qbezier(80,65)(85,72)(85,72)
\qbezier(87.5,73)(88,73.5)(90,74)
 \qbezier(90,74)(100,75.5)(120,76.5)
\qbezier(120,76.5)(140,77)(150,77.3)

%%VTORAYA LINIYA, LEVAYA POLOVINA

\qbezier(80,65)(75,58)(70,56) \qbezier(70,56)(60,54.5)(40,53.5)
\qbezier(40,53.5)(20,53)(10,52.7)
%3 liniya
 \qbezier(80,35)(85,42)(90,44)
\qbezier(90,44)(100,45.5)(120,46.5)
\qbezier(120,46.5)(140,47)(150,47.3)
\qbezier(80,35)(75,28)(70,26) \qbezier(70,26)(60,24.5)(40,23.5)
\qbezier(40,23.5)(20,23)(10,22.7)
%shtrijjovka vtoroy polosi, pravaya chast'
%Razriv pervoy shtrijovki.

% \put(82,8){\line(0,1){30}}
 \put(82,8){\line(0,1){4.5}}
 \put(82,15){\line(0,1){2.3}}
 \put(82,24){\line(0,1){3}}
 \put(82,30){\line(0,1){7}}
%Razriv vtoroy shtrijovki

 \put(84,11){\line(0,1){1.5}}
\put(84,14.5){\line(0,1){3}}
 \put(84,24){\line(0,1){3}}
 \put(84,32){\line(0,1){7}}

%Razriv tret'ey shtrijovki

%\put(86,17){\line(0,1){1}}
%\put(86,24){\line(0,1){1}}

\put(86,33){\line(0,1){8}}

%%%%%%%%%%%%%%%%%%%%%%%%%%%%%%%%
%Razriv chervertoy shtrijovki

 \put(88,17){\line(0,1){2}}
 \put(88,24){\line(0,1){3}}
 \put(88,32){\line(0,1){11}}

% \put(88,20){\line(0,1){10}}
% \put(88,32){\line(0,1){5}}

\put(90,14){\line(0,1){30}}
 \put(92,14){\line(0,1){30}}
\put(94,14.5){\line(0,1){30}}
 \put(96,15){\line(0,1){30}}
\put(98,15){\line(0,1){30}}
 \put(100,15){\line(0,1){30}}
\put(102,15.6){\line(0,1){30}}
 \put(104,15.5){\line(0,1){30}}
 \put(106,16){\line(0,1){30}}
 \put(108,16.2){\line(0,1){30}}
 \put(110,16.2){\line(0,1){30}}
 \put(110,16.3){\line(0,1){30}}
 \put(112,16.4){\line(0,1){30}}
  \put(114,16.4){\line(0,1){30}}
   \put(116,16.5){\line(0,1){30}}
    \put(118,16.5){\line(0,1){30}}
     \put(120,16.6){\line(0,1){30}}

 \put(122,16.6){\line(0,1){30}}
  \put(124,16.7){\line(0,1){30}}
   \put(126,16.7){\line(0,1){30}}
    \put(128,16.8){\line(0,1){30}}
     \put(130,16.8){\line(0,1){30}}
 \put(132,16.9){\line(0,1){30}}
  \put(134,16.9){\line(0,1){30}}
   \put(136,17){\line(0,1){30}}
    \put(138,17){\line(0,1){30}}
     \put(140,17.1){\line(0,1){30}}
 \put(142,17.1){\line(0,1){30}}
  \put(144,17.2){\line(0,1){30}}
   \put(146,17.2){\line(0,1){30}}
    \put(148,17.3){\line(0,1){30}}
     \put(150,17.3){\line(0,1){30}}

%shtrijovka vtoroy polosi v levoy chasti
\put(78,2){\line(0,1){30}}
 \put(76,0.5){\line(0,1){30}}
\put(74,-2.1){\line(0,1){30}}
 \put(72,-3){\line(0,1){30}}
\put(70,-4){\line(0,1){30}}
 \put(68,-4){\line(0,1){30}}
\put(66,-4.5){\line(0,1){30}}
 \put(64,-5){\line(0,1){30}}
\put(62,-5){\line(0,1){30}}

\put(60,-5.2){\line(0,1){30}}

\put(58,-5.6){\line(0,1){30}}
 \put(56,-5.5){\line(0,1){30}}
 \put(54,-6){\line(0,1){30}}
 \put(52,-6.2){\line(0,1){30}}
 \put(50,-6.2){\line(0,1){30}}
 \put(48,-6.2){\line(0,1){30}}
 \put(46,-6.4){\line(0,1){30}}
  \put(44,-6.4){\line(0,1){30}}
   \put(42,-6.5){\line(0,1){30}}
    \put(40,-6.5){\line(0,1){30}}
     \put(38,-6.6){\line(0,1){30}}
 \put(36,-6.6){\line(0,1){30}}
  \put(34,-6.7){\line(0,1){30}}
   \put(32,-6.7){\line(0,1){30}}
    \put(30,-6.8){\line(0,1){30}}
     \put(28,-6.8){\line(0,1){30}}
 \put(26,-6.9){\line(0,1){30}}
  \put(24,-6.9){\line(0,1){30}}
   \put(22,-7){\line(0,1){30}}
    \put(20,-7){\line(0,1){30}}
     \put(18,-7.1){\line(0,1){30}}
 \put(16,-7.1){\line(0,1){30}}
  \put(14,-7.2){\line(0,1){30}}
   \put(12,-7.2){\line(0,1){30}}
    \put(10,-7.3){\line(0,1){30}}
%4 liniya

\qbezier(80,65)(85,72)(85,72)
\qbezier(87.5,73)(88,73.5)(90,74)

 \qbezier(90,74)(100,75.5)(120,76.5)
\qbezier(120,76.5)(140,77)(150,77.3)

%razriv okolo -10\pi/4

\qbezier(80,5)(85,12)(85,12)
% \qbezier(89.5,13.7)(90,14)(90,14)

%konec razriva

 \qbezier(90,14)(100,15.5)(120,16.5)
\qbezier(120,16.5)(140,17)(150,17.3)

%levaya polovina

\qbezier(80,5)(75,-2)(70,-4) \qbezier(70,-4)(60,-5.5)(40,-6.5)
\qbezier(40,-6.5)(20,-7)(10,-7.3)

%oboznacheniya konturov {\cal C}, {\cal C}_1+i\pi/4, -{\cal C}_1-13\pi/4
%\put(90,50){${\cal C}_+$}
% \put(120,50){${\cal C}_1+i\pi/4$}
% \put(30,40){$-{\cal C}_{1}-13\pi/4$}

%vektori, ukazivayuschie konturi
%\put(116,50){\vector(-2,-1){15}}
% \put(116,50){\vector(1,-2){11.5}}

% \put(116,50){\vector(1,3){12.5}}
%\put(90,50){\vector(-2,-1){30}}
% \put(90,50){\vector(2,-1){10}}
 %\put(28,40){\vector(1,3){11}}

%\put(28,40){\vector(-1,-2){14}}
 %\put(28,40){\vector(2,-1){32.5}}
%napravleniya kontura

%{\linethickness{0.3mm}\put(60,60){\vector(0,1){10}}}
%\put(120,87.5){\vector(-1,0){10}} \put(60,72.5){\vector(-1,0){10}}
%\put(40,12.5){\vector(1,0){10}}
 %\put(100,27.5){\vector(1,0){10}}
%\put(100,40){\vector(0,-1){10}}

% contur \cal C_0
% kontur \gamma_1 s razrivom dlia -\pi/2
%\linethickness{0.3mm} \put (10,80){\line(1,0){140}}

\linethickness{0.3mm} \put (10,80){\line(1,0){70}}
\linethickness{0.3mm} \put (88,80){\line(1,0){62}}

%\linethickness{0.3mm} \put (100,27.5){\line(1,0){60}}
%\linethickness{0.3mm}\put (60,72.5){\line(-1,0){60}}

%%%%%%%%%%%%%%%%%%%%%%%%%%%%%
%Estoy BORRANDO LAS LINEAS MÁS GRUESAS
%\linethickness{1mm}\put(10,16){\line(1,0){62}}
%\qbezier(72,16)(83,22)(89,25)
%\linethickness{1mm}\put(89,25){\line(1,0){61}}
%%%%%%%%%%%%%%%%%%%%%%%%%%%%

%contur \gamma_2 s razrivom dlia -5\pi/2
\linethickness{0.3mm}\put(10,20){\line(1,0){70}}
\linethickness{0.3mm}\put(89,20){\line(1,0){61}}
%napravleniya \gamma_1 y \gamma_2
{\linethickness{0.3mm}\put(140,80){\vector(-1,0){20}}}
{\linethickness{0.3mm}\put(120,20){\vector(1,0){20}}}

%%%%%%%%%%%%%%%%%%%%%%%%%%%%%
%Estoy BORRANDO LAS LINEAS MÁS GRUESAS
%\linethickness{1mm}\put(89,85){\vector(1,0){61}}
%\qbezier(72,76)(83,82)(89,85)
%\linethickness{1mm}\put(10,76){\vector(1,0){62}}
%%%%%%%%%%%%%%%%%%%%%%%%%%%%%

%Oboznacheniya \gamma_1 y \gamma_2

\put(104.4,81){$\gamma_1$}
\put(104.4,21){$\gamma_2$}

%linii vertikal'nie
%\thicklines \linethickness{0.3mm}\put(100,87.5){\line(0,-1){60}}

%\linethickness{0.3mm}\put (60,72.5){\line(0,-1){60}}
% Os' vertikal'naya
%\thinlines \put(80,-12){\vector(0,1){150}}
\thinlines \put(80,-12){\vector(0,1){130}}

%oboznachenie vert. osi
% \put(75,120){$\mathrm{Re}\ \omega > 0$}

 \end{picture}
}

\vspace{1cm}
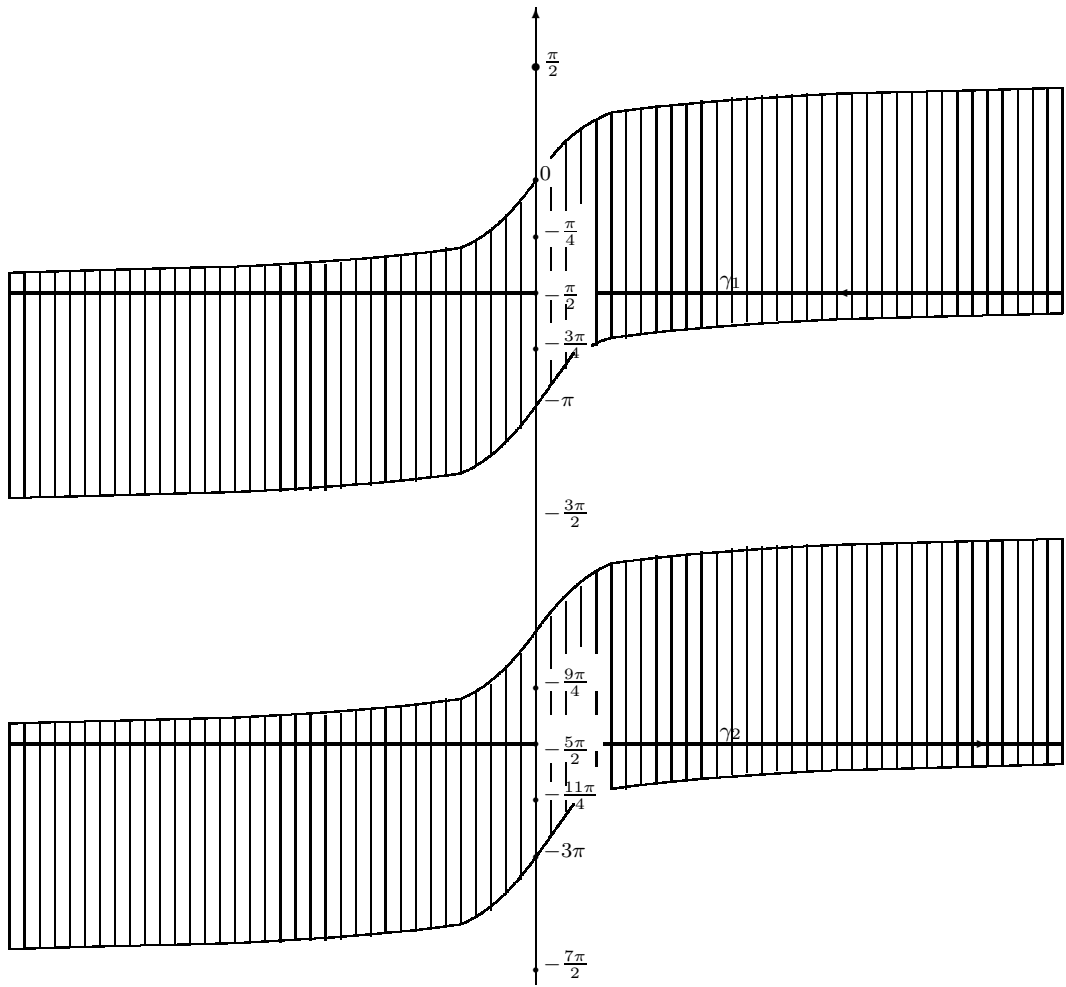
 \centerline{Figure 3. Contour $\mathcal{C}_{0}$.}

%%%%%%%%%%%%%%%%%%%%%%%%%%%%%%%%%%%%%%%%%%%%%%%%%%%%%%%%%%%%%%%%%%%%%%%%%%%%%%%%%%%%%%%%%%%%%
%%%%%%%%%%%%%%%%%%%%%%%%%%%%%%%%%%%%%%%%%%%%%%%%%%%%%%%%%%%%%%%%%%%%%%%%%%%%%%%%%%%%%%%%%%%%%
%%%%%%%%%%%%%%%%%%%%%%%%%%%%%%%%%%%%%%%%%%%%%%%%%%%%%%%%%%%%%%%%%%%%%%%%%%%%%%%%%%%%%%%%%%%%%

\vspace{0.5cm}
% and $\widehat{f}$ is the Fourier transform of the profile function $f$.
A crucial role in these representations plays the fact that the integrals in (\ref{udc}) converge absolutely for $\omega\in\mathbb{R}$ since the integrand $H$ in (\ref{H}) and (\ref{H1V}) decreases exponentially. In the case of a NN-wedge (see Remark 2.5 in \cite{la}) the corresponding integrand takes the form
\begin{eqnarray}\label{HN}
H_N(\beta,\Phi):= \coth\left[q \left(\beta + i\ \dfrac{\pi}{2} - i\alpha\right)   \right]
            +\coth\left[q \left(\beta- i\ \displaystyle\frac{3\pi}{2} + i\alpha \right) \right], \qquad \beta\in\mathbb{C},
\end{eqnarray}
thus the integral (\ref{udc}) does not converge absolutely for $\omega\in\mathbb{R}$ (it converges conditionally). We find an asymptotic of the convergence to the limiting amplitude and compare our results with \cite{hew} for scattering on  a half plane.
 So the proof of (\ref{ud2}) is not obtained by an application of the method \cite{la}-\cite{mm} directly. We apply some trick in this case to obtain (\ref{ud2}), (see Lemma \ref{propJ}). We will use this representation to prove that the solution of the nonstationary scattering problem belongs to a functional space securing the uniqueness and that the Limiting Amplitude Principle holds.
  %Finally we obtain the formula for the wave diffracted by the edge of half-plane taking in the final formula (\ref{rud}) $\Phi=2\pi$ (see (\ref{limZ})), we compare the obtained formula with \cite{hew} and analyze the rate of the convergence of the diffracted wave  amplitude to the limiting amplitude. We compare the rate of convergence to the limiting amplitude with corresponding rate obtained in \cite{hew}. The plan of the paper is the following. In section 2...
The paper is organized as follows.  In section 2 we formulate the problem for nonstationary plane wave scattering  on a NN-wedge. In section 3 we give a solution to the stationary NN-problem with a complex parameter. In sections 4 and 5 we prove an explicit representation for the diffracted wave with a smooth profile function and in section 6 we give a complete solution to the scattering problem and prove the existence and uniqueness of the solution. In sections 7 and 8 we prove  the Limiting Amplitude Principle and  we obtain a rate of convergence to the limiting amplitude.  Finally, in section 9 we consider the case of half plane and  compare our results with the results of \cite{hew} when the impinging wave profile is the Heaviside function.
%\vspace{1cm}
%%%%%%%%%%%%%%%%%%%%%%%%%%%%%%%%%%%%%%%%%%%%%%%%%%%%%%%%%%%%%%%%%%%%%%%%%%%%%%%%%%%%%%%%%%%%%%%%%%%%
%%%%%%%%%%%%%%%%%%%%%%%%%   SECCIÓN
%%%%%%%%%%%%%%%%%%%%%%%%%%%%%%%%%%%%%%%%%%%%%%%%%%%%%%%%%%%%%%%%%%%%%%%%%%%%%%%%%%%%%%%%%%%%%%%%%%%%
\section{The statement of the problem}
Let us proceed  to the  exact formulation. We denote by $Q:=\mathbb{R}^2\setminus W$ the complement angle of magnitude
\begin{eqnarray}\label{defFi}
\Phi:=2\pi - \phi, \qquad \Phi\in(\pi,2\pi]
\end{eqnarray}
and $\partial Q = Q_1\cup Q_2$, where $Q_1:=\left\{(y_1,0):y_1>0\right\}$ and $Q_2:=\left\{(\rho\cos\phi,\rho\sin\phi):\rho>0\right\}$. The front of the wave  $u_{in}(y,t)$  is the line
$\{y\in\mathbb{R}^2 : t - \mathbf{n_0} \cdot y = 0\}$ in $\mathbb{R}^2$ at the moment $t$. To the right of this line namely, for $\mathbf{n_0}\cdot y>t$, $u_{in}(y,t)=0$ by (\ref{uin}) and (\ref{f}). We impose the following conditions on the vector $\mathbf{n_0}$. Suppose that for $t \leq 0$, the front of $u_{in}(y,t)$ is disjoint from $W\setminus\{0\}$. This is equivalent to the condition $\phi-\frac{\pi}{2}<\alpha<\frac{\pi}{2}$. Moreover, suppose that the incident wave is reflected by both sides of the wedge. This is equivalent to the condition $0<\alpha<\phi$. Therefore, these two conditions on the vector $\mathbf{n_0}$ are expressed by the inequalities
\begin{equation} \label{alfa}
\text{max} \left\{\phi - \frac{\pi}{2}, 0 \right\} < \alpha < \text{min}  \left\{\frac{\pi}{2}, \phi\right\} .
\end{equation}
If some of these conditions are not satisfied, then the statement of the problem is slightly more complicated technically, but not fundamentally. All the final formulas for the solution will be valid for arbitrary $\alpha$, see Remark \ref{obsalpha}. To avoid these technical complications, we assume below that (\ref{alfa}) holds (see Figure 1). The scattering of the incident wave $u_{in}$ by the NN-wedge $W$ is described by means of the following initial boundary value problem
\begin{equation}\label{NP}
\left\{     \begin{array}{rcl}
            \square u(y,t,\Phi)=0, &\quad  & y\in Q   \\
            \dfrac{\partial\ }{\partial \mathbf{n}} u(y,t,\Phi)=0,         &  & y\in \partial Q
            \end{array}
\right| t\in\mathbb{R}
\end{equation}
where $\square:=\partial_t^2-\Delta$. Here $\dfrac{\partial\ }{\partial \mathbf{n}}$ means the normal exterior derivative. We include the incident wave $u_{in}$ in the statement of the problem through the initial condition
\begin{equation}\label{ic}
u(y,t)=u_{in}(y,t), \qquad\qquad y\in Q, \qquad t<0.
\end{equation}
It is possible since $u_{in}(y,t)$ is a solution to problem (\ref{NP}) for $t<0$. The Neumann conditions in (\ref{NP}) hold for $t<0$ since $u_{in}$ is identically zero in a neighborhood of $\partial Q$. Let us denote the scattered  wave by
%\begin{eqnarray}
$u_s(y,t):=u(y,t)-u_{in}(y,t)$.
%\end{eqnarray}
Obviously $u_s$ satisfies the following mixed problem
\begin{equation}\label{Pus}
\left\{
        \begin{array}{rcll}
        \square  u_s(y,t) &=& 0,                   &   y\in Q,\ t>0    \\
        \dfrac{\partial\ }{\partial\mathbf{n}} u_s(y,t) &=&
      - \dfrac{\partial\ }{\partial\mathbf{n}} u_{in}(y,t),  \qquad   &   y\in \partial Q,\ t>0    \\
        u_s(y,0) &=& \dot{u}_s(y,0)\ =\ 0, & y\in Q.
        \end{array}
\right.%| t>0,
\end{equation}
%The diffracted wave:
%\begin{eqnarray}
%u_d(y,t):=u_s(y,t)-u_r(y,t).
%\end{eqnarray}
Define the reflected wave $u_r(y,t)$ as
\begin{eqnarray} \label{ur}
u_r(\rho,\theta,t):=\left\{
									\begin{array}{cl}
									 u_{r,1}(\rho,\theta,t),&\qquad \phi \leq \theta \leq \theta_1\\
									0,& \qquad\theta_1< \theta<\theta_2\\
									 u_{r,2}(\rho,\theta,t),&\qquad\theta_2 \leq \theta \leq 2\pi\\
									\end{array}
					\right.
\end{eqnarray}
where $y=\rho e^{i\theta}$, $\theta_1,\theta_2$ are given by (\ref{T1T2}) and $u_{r,1}$, $u_{r,2}$ are the plane waves reflected by $Q_1$ and $Q_2$, respectively and they are defined as
%\begin{equation}\label{ur1ur2}
$u_{r,1}(\rho,\theta,t) = e^{i(\mathbf{k_1}\cdot y-\omega_0 t)}f(t-\mathbf{v_1}\cdot y)$, %         \qquad
$u_{r,2}(\rho,\theta,t) = e^{i(\mathbf{k_2}\cdot y-\omega_0 t)}f(t-\mathbf{v_2}\cdot y)$.
%\end{equation}
Here
%\begin{equation}\label{kini}
$\mathbf{k_1} = \omega_0 \mathbf{v_1}$, %\quad
$\mathbf{v_1} = (\cos \theta_1,\sin \theta_1)$ and % \qquad
$\mathbf{k_2} = \omega_0 \mathbf{v_2}$, %\quad
$\mathbf{v_2} = (\cos \theta_2,\sin \theta_2)$.
%\end{equation}
In the present paper we prove the existence of a solution to the nonstationary problem (\ref{Pus}) in a functional space $\mathcal{E}_{\varepsilon,N}$. We will prove that this solution (the scattered wave $u_s$) is represented as a sum of the reflected wave $u_r$ (\ref{ur}) and the diffracted wave $u_d$ (see Section \ref{secSolSta}). We find the Sommerfeld-Malyuzhinets type representation for the diffracted wave and we prove that it has the form (\ref{ud2}) with
\begin{eqnarray}\label{ZN}
    Z(\beta,\Phi) = Z_N(\beta,\Phi)
                 &:=& H_N\left(\beta-i\dfrac{5\pi}{2},\Phi\right) - H_N\left(\beta-i\dfrac{\pi}{2},\Phi\right),
    \qquad \beta\in\mathbb{C},
\end{eqnarray}
where $H_N$ is given by (\ref{HN}). We prove that this formula coincides with (14) and (15) in \cite{hew} for $\Phi=2\pi$. Using this formula we determine the values of parameters $\varepsilon$ and $N$ in the space $\mathcal{E}_{\varepsilon,N}$. Moreover, we prove the Limiting Amplitude Principle and find the rate of convergence to the Limiting Amplitude.

%%%%%%%%%%%%%%%%%%%%%%%%%%%%%%%%%%%%%%%%%%%%%%%%%%%%%%%%%%%%%%%%%%%%%%%%%%%%%%%%%%%%%%%%%%%%%%%%%%%%
%%%%%%%%%%%%%%%%%%%%%%%%%   SECCIÓN
%%%%%%%%%%%%%%%%%%%%%%%%%%%%%%%%%%%%%%%%%%%%%%%%%%%%%%%%%%%%%%%%%%%%%%%%%%%%%%%%%%%%%%%%%%%%%%%%%%%%
\section{Solution to the  ``stationary" NN-problem with a parameter}\label{secSolSta}

First we consider the case $\pi<\Phi<2\pi$. For convenience, we omit the variable $\Phi$ in the notation of $H$, $H_N$, $Z$, $Z_N$, $u_d$ and $u_r$.  %In section \ref{SecReduction} we will consider the case $\Phi=2\pi$ and we will compare our results with the results obtained in \cite{hew}.

After the Fourier-Laplace transform $t\rightarrow \omega$ in  problem (\ref{Pus}), we come to the following ``stationary" NN-problem
\begin{equation}\label{SPNN}
\left\{
	   \begin{array}{rcl}
       (\Delta + \omega^2)\ \widehat{u}_s(y,\omega)
            & = & 0,  \hspace{5.7cm} y\in Q   \\ %\\
       \dfrac{\partial\ }{\partial y_2}\widehat{u}_s(y,\omega)
            & = &  -i\omega \widehat{g}(\omega) \sin\alpha\  e^{i\omega y_1 \cos\alpha}, \hspace{1.8cm} y\in Q_1 \\ %\\
       \dfrac{\partial\ }{\partial \mathbf{n_2}}\widehat{u}_s(y,\omega)
            & = & i\omega \widehat{g}(\omega) \sin(\Phi+\alpha)
                  e^{-i\omega y _2 \frac{\cos(\Phi+\alpha)}{\sin\Phi}},\ \quad y\in Q_2
		\end{array}
\right.
\end{equation}
 with a parameter $\omega\in\mathbb{C}^+:=\{z\in\mathbb{C}:\mathrm{Im}\ z> 0\}$.
%and $\widehat{g}(\omega):=\widehat{f}(\omega-\omega_0)$.
This is done similarly to (31) in \cite{la} and (3.7) in \cite{mm}. The solution of this problem admits the following representation in the polar coordinates
	\begin{eqnarray}\label{TFus}
				\widehat{u}_{s}(\rho,\theta,\omega)
						&:=& -  \widehat{g}(\omega) e^{i\rho\omega\cos(\theta-\alpha)} +
                            \displaystyle\frac{i\widehat{g}(\omega)}{4\Phi}% s_1(p_1)}
						    \int\limits_{\mathcal{C}} e^{-\rho \omega \sinh\beta} H_N(\beta+i\theta)\  \mathrm{d}\beta,\quad
                            (\rho,\theta)\in Q,\qquad
				\end{eqnarray}
where $H_N$ is function (\ref{HN}) and  $\mathcal{C}$ is the Sommerfeld-type contour
\begin{eqnarray}\label{C}
\mathcal{C} &=& \mathcal{C}_1 \cup \mathcal{C}_2.
\end{eqnarray}
Here
\begin{equation}\label{C1}
\mathcal{C}_1 =  \left\{\beta_1-i\frac{\pi}{2}:\beta_1\geq 1\right\}\cup
                        \left\{1+i\beta_2:-\frac{5\pi}{2}\leq\beta_2\leq -\frac{\pi}{2}\right\} \cup
                        \left\{\beta_1-i\frac{5\pi}{2}:\beta_1\geq 1\right\}
\end{equation}
and $\mathcal{C}_2:=-\mathcal{C}_1-3i\pi$.
%is the symmetric to $\mathcal{C}_1$ with respect to $-i\frac{3\pi}{2}$.
We choose the  counter-clockwise orientation for the contours $\mathcal{C}_{1,2}$, (see Figure 4). Obviously the integral in (\ref{TFus}) converges absolutely since for any $\beta:=\beta_1+i\beta_2\in\mathcal{C}$ and for any $\omega:=\omega_1+i\omega_2\in\mathbb{C}^+$, $\mathrm{Re}(-\rho\omega\sinh\beta) = -\rho\omega_2\cosh\beta_1$. Representation (\ref{TFus}) may be obtained by the Method of the Complex Characteristics (see \cite{kmz}-\cite{kmr}). This method was used to obtain the similar formulas (150), (148) in \cite{la} and  (10.4), (14.2) in \cite{mm} for the DD and DN-wedges, respectively.  Similarly to Corollary 5.2 in \cite{la} it is proved that $u_s\in C^{\infty}(\overline{Q})$.
We prove directly that (\ref{TFus}) satisfies (\ref{SPNN}) in Appendix A1, Lemma \ref{comprous},
and $u_s\in \mathrm{C}^\infty(\overline{Q})$ in Section \ref{smoothus}.

%%%%%%%%%%%%%%%%%%%%%%%%%%%%%%%%%%%%%%%%%%%%%%%%%%%%%%%%%%
%%%%%%%%%%%%%%%%%%%%%%%%%%%%%%%%%%%%%%%%%%%%%%%%%%%%%%%%%%
%%%%%%%%%%%%%%%%%%%%%%%%%%%%%%%%%%%%%%%%%%%%%%%%%%%%%%%%%%
% AQUÍ EMPIEZAN LAS INSTRUCCIONES PARA LA FIGURA 2 	EL CONTORNO C
%%%%%%%%%%%%%%%%%%%%%%%%%%%%%%%%%%%%%%%%%%%%%%%%%%%%%%%%%%
%%%%%%%%%%%%%%%%%%%%%%%%%%%%%%%%%%%%%%%%%%%%%%%%%%%%%%%%%%
%%%%%%%%%%%%%%%%%%%%%%%%%%%%%%%%%%%%%%%%%%%%%%%%%%%%%%%%%%

%\newpage
\vspace{0.4cm}

{\scriptsize

\setlength{\unitlength}{1mm}
 \begin{picture} (110,100)
 %os' y_1
 \put(30,75) {\vector (1,0){90}} \put(123,74){$\mathrm{Re}\ \beta$}
 \put(40,75) {\line (-1,0){10}}

 %\put(75.1,96) {\circle* {0.7}}\put(78,96){$\pi$}
 \put(75.1,90) {\circle* {0.7}}\put(76.2,89){$\frac{\pi}{2}$}
 \put(75.1,75) {\circle* {0.7}}\put(72,76){$0$}
 \put(75.1,60) {\circle* {0.7}}\put(76.2,59){$-\frac{\pi}{2}$}
 \put(75.1,45) {\circle* {0.7}}\put(76.2,44){$-\pi$}
 \put(75.1,30) {\circle* {0.7}}\put(76.2,29){$-\frac{3\pi}{2}$}
 \put(75.1,15) {\circle* {0.7}}\put(76.2,14){$-2\pi$}
 \put(75.1,0.2) {\circle* {0.7}}\put(76.2,0){$-\frac{5\pi}{2}$}
 %\put(75.1,4) {\circle* {0.7}}\put(76.2,4){$-3\pi$}

 %os' y_2
 \put(75,05){\vector(0,1){90}} \put(73,98){$\mathrm{Im}\ \beta$}
 \put(75,10){\vector(0,-1){12.3}}

% LOS PUNTOS 1 Y -1
% \put(84,75) {\circle* {0.7}}\put(82,69){$1$}
% \put(66,75) {\circle* {0.7}}\put(62,69){$-1$}

 %La figura de vectores en forma de u inversa
 %Izquierdo
 \put(66,0.2){\line(0,1){60}}
 \put(66,25){\vector(0,1){10}}
 \put(66,60){\line (-1,0){35}}
 \put(60,60){\vector (-1,0){10}}\put(50,62){$\mathcal{C}_2$}
 \put(66,0.2){\line (-1,0){35}}
 \put(50,0.2){\vector (1,0){10}}

 %Derecho
 \put(84,0.2){\line(0,1){60}}
 \put(84,50){\vector(0,-1){10}}
 \put(84,60){\line (1,0){35}}
 \put(110,60){\vector (-1,0){10}}\put(95,62){$\mathcal{C}_1$}
 \put(84,0.2){\line (1,0){35}}
 \put(90,0.2){\vector (1,0){10}}

 \end{picture}
}

 \vspace{0.5cm}
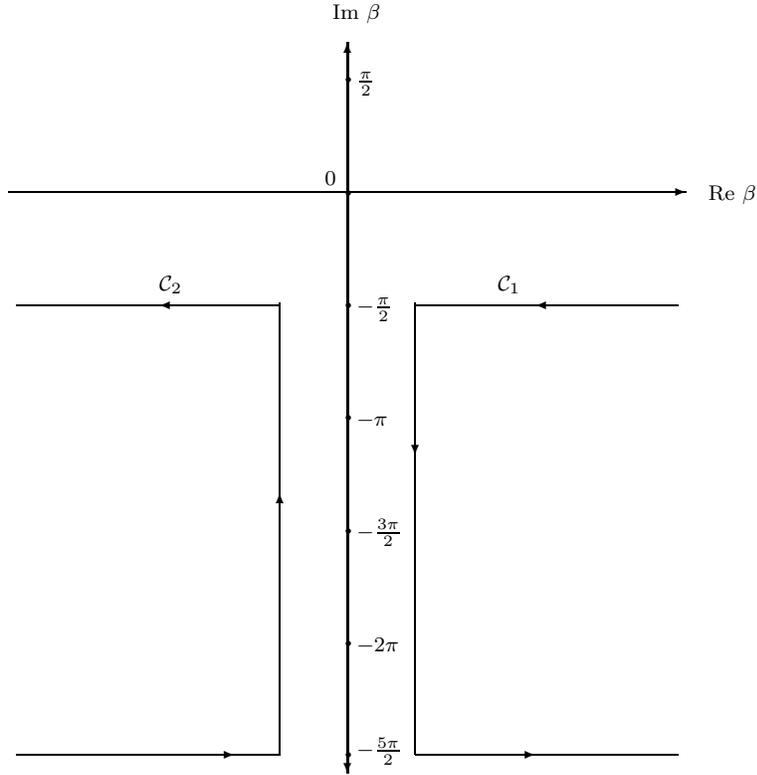
 \centerline{Figure 4. Contour $\mathcal{C}.$ }
\vspace{0.6cm}
Let us decompose the ``stationary" solution $\widehat{u}_s$ into the reflected ``stationary" wave $\widehat{u}_r$ and diffracted ``stationary" wave $\widehat{u}_d$. Changing the contour $\mathcal{C}$ (\ref{C}) by the contour $\mathcal{C}_0$ (\ref{C0}) and calculating residues we obtain %similarly to (45), (46) in \cite{la} and (12.13) in \cite{mm}
that
%\begin{eqnarray}\label{whus}
$\widehat{u}_s = \widehat{u}_r + \widehat{u}_d$,
%\end{eqnarray}
where
\begin{equation*}
\left.
\begin{array}{rcl}
\widehat{u}_r(\rho,\theta,\omega) &=&
                    \left\{
									\begin{array}{cl}
									 \widehat{u}_{r,1}(\rho,\theta,\omega),&\qquad \phi \leq \theta \leq \theta_1\\
									0,& \qquad\theta_1< \theta<\theta_2\\
									 \widehat{u}_{r,2}(\rho,\theta,\omega),&\qquad\theta_2 \leq \theta \leq 2\pi\\
									\end{array}
					\right.
\end{array}
\right|\ (\rho,\theta)\in Q, \ \theta\neq \theta_1, \theta_2,
\end{equation*}
with
%\begin{eqnarray*}
      $\widehat{u}_{r,1}(\rho,\theta,\omega) = \widehat{g}(\omega) \ e^{i\omega\rho\cos(\theta-\theta_1)}$
and   $\widehat{u}_{r,2}(\rho,\theta,\omega) = \widehat{g}(\omega) \ e^{i\omega\rho\cos(\theta-\theta_2)}$.
%\end{eqnarray*}
Here the ``critical" directions $\theta_{1,2}$ are given by (\ref{T1T2}) and $\widehat{u}_d$ is given by (\ref{udc}). It is not difficult to see that the inverse Fourier-Laplace transform of $\widehat{u}_r$ is the function $u_r$ given by (\ref{ur}). Thus, the solution $u_s(y,t)$ of the problem (\ref{Pus}) is decomposed into $u_r(y,t)$ and $u_d(y,t)$ in the following form:
%\begin{eqnarray}\label{splitus}
$u_s(y,t) = u_r(y,t) + u_d(y,t),$
%\end{eqnarray}
 where
%\begin{eqnarray}\label{udorig}
$u_d(y,t) = F^{-1}_{\omega\rightarrow t}\big[\widehat{u}_d(y,\omega)\big]$.
%\end{eqnarray}
It remains only to prove that the function $u_d(y,t)$ admits the representation (\ref{ud2}) and to verify that the function $Z_N$ given in (\ref{ZN}) is an appropriate one for the diffracted wave. To do this we need first the continuation of the function $\widehat{u}_d(\cdot,\cdot,\omega)$ to $\overline{\mathbb{C}^+}$.

%%%%%%%%%%%%%%%%%%%%%%%%%%%%%%%%%%%%%%%%%%%%%%%%%%%%%%%%%%%%%%%%%%%%%%%%%%%%%%%%%%%%%%%%%%%%%%%%%%%%
%%%%%%%%%%%%%%%%%%%%%%%%%   SECCIÓN
%%%%%%%%%%%%%%%%%%%%%%%%%%%%%%%%%%%%%%%%%%%%%%%%%%%%%%%%%%%%%%%%%%%%%%%%%%%%%%%%%%%%%%%%%%%%%%%%%%%%

\section{Extension of diffracted wave density function  }
\begin{defn}\label{defJ} Let
\begin{equation}\label{nt}
    \Theta := [\phi,2\pi]\setminus\{\theta_1,\theta_2\}.
\end{equation}
where $\theta_1,\theta_2$ are defined in (\ref{T1T2}). For any
$\omega\in\mathbb{C}^+$, $(\rho,\theta)\in \mathbb{R}^+\times \Theta$ let us define
\begin{equation}\label{JN}
\mathcal{J}_{d} (\rho,\theta,\omega):=
   \int\limits_{\mathcal{C}_{0}} e^{-\rho \omega \sinh\beta} H_N(\beta+i\theta)\  \mathrm{d}\beta,
\end{equation}
with $H_N$ defined by (\ref{HN}).
\end{defn}
First we extend $\mathcal{J}_d(\cdot,\cdot,t)$ to $\overline{\mathbb{C}^+}$. We will denote by $\mathrm{H}(\Omega)$, the space of the holomorphic functions in the region $\Omega\subset\mathbb{C}$. Everywhere below we will assume that $(\rho, \theta)\in \mathbb{R}^+\times \Theta$, $\omega=\omega_1+i\omega_2$ and $\beta=\beta_1+i\beta_2$.

%%%%%%%%%%%%%%%%%%%%%%%%%%%%%%%%%%%%%%%%%%%%%%%%%%%%%%%%%%%%%%%%%%%%%%%%%%%%%%%%%%%%%%%%%%%%%
%%%%%%%%%%%%%%%%%%%%%%%%%%%%%%%%%%%%%%%%%%%%%%%%%%%%%%%%%%%%%%%%%%%%%%%%%%%%%%%%%%%%%%%%%%%%%
%%%%%%%%%%%%%%%%%%%%%%%%%%%%%%%%%%%%%%%%%%%%%%%%%%%%%%%%%%%%%%%%%%%%%%%%%%%%%%%%%%%%%%%%%%%%%
%%%%%%%%%%%%%%%%%%%%%%%%%%%%%%%%%%%%%%%%%%%%%%%%%%%%%%%%%%%%%%%%%%%%%%%%%%%%%%%%%%%%%%%%%%%%%
%\newpage

%%%%The following Lemma plays a crucial role in the paper.
%%%%
%%%%
%%%%\begin{lem}\label{lexnew}
%%%%If $\omega\in \mathbb{C}^+$ and $\beta\in\mathcal{C}_{0}$, then% the following bound holds
%%%%\begin{eqnarray}
%%%%        |e^{-\rho\omega\sinh\beta}|
%%%%            &=& e^{- \rho\omega_2\cosh\beta_1},
%%%%            \qquad \beta_1\in\mathbb{R}. \label{NBHnew}
%%%%\end{eqnarray}
%%%%\end{lem}
%%%%
%%%%\begin{proof} By (\ref{C0}), $\beta=\beta_1-i\frac{\pi}{2}$ or $\beta=\beta_1-i\frac{5\pi}{2}$.
%%%%In any case we have
%%%%\begin{equation*}%\label{vaexp}
%%%%\left|e^{-\rho \omega \sinh\beta} \right|
%%%%= \left|e^{i\rho (\omega_1 + i \omega_2) \cosh\beta_1}\right|
%%%%= e^{-\rho  \omega_2 \cosh\beta_1}.
%%%%\end{equation*}
%%%%\end{proof}

%\vspace{0.9cm}

\begin{lem}\label{propJ}
\begin{enumerate}[i)]

\item The integral in (\ref{JN}) converges absolutely for $\omega\in\mathbb{C}^+$,
% $(\rho, \theta)\in \mathbb{R}^+\times \mathcal{T}$
 and determines an analytic function in $\omega\in\mathbb{C}^+$:
            \begin{equation}\label{JNH}
            \mathcal{J}_{d}(\rho,\theta,\omega)\in \mathrm{H}(\mathbb{C}^+), \qquad (\rho,\theta)\in\mathbb{R}^+\times\Theta.
%            \qquad \text{for } (\rho,\theta)\in\mathbb{R}^+\times \mathcal{T}.
            \end{equation}

  \item The function $\mathcal{J}_{d}(\cdot,\cdot,\omega)$ is represented as
\begin{eqnarray}\label{repJd}
\mathcal{J}_{d}(\rho,\theta,\omega) &=& \int\limits_\mathbb{R} e^{i\rho\omega\cosh\beta} Z_N(\beta + i\theta)\ \mathrm{d}\beta, \qquad \omega\in\overline{\mathbb{C}^+}
\end{eqnarray}
where $Z_N$ is given by (\ref{ZN}) and it admits a continuous extension to $\overline{\mathbb{C}^+}$.

  \item  $\mathcal{J}_{d}(\cdot,\cdot,\omega)\in C(\mathbb{R})$ and the following estimate holds
\begin{eqnarray}\label{bjd}
        |\mathcal{J}_{d}(\rho,\theta,\omega)| &\leq& C(\theta),   \qquad \omega\in\overline{\mathbb{C}^+}.
\end{eqnarray}

  \item The following limit holds
\begin{eqnarray}\label{limJd}
\mathcal{J}_{d}(\rho,\theta,\omega_1)=\lim\limits_{\omega_2 \rightarrow 0} \mathcal{J}_{d}(\rho,\theta,\omega_1+i\omega_2)
\end{eqnarray}
   in the sense of $\mathcal{S}'(\mathbb{R}_{\omega_1})$.

\end{enumerate}
\end{lem}

\begin{proof} \textbf{i)}
From (\ref{HN}) it follows that for $\theta\in[\phi,2\pi]$ the set $P(\theta)$ of poles of the function $H_N(\beta+i\theta)$ is given by
\begin{eqnarray}\label{poloHNth}
 P(\theta) = P_1(\theta) \cup \Big[-P_1(\theta)+i\pi-2i\theta\Big]
\end{eqnarray}
where
%\begin{eqnarray}\label{Po1}
 $P_1(\theta) = \left\{-i\frac{ \pi}{2} + i\alpha + 2ik\Phi - i\theta : k\in\mathbb{Z}\right\}$.
%\end{eqnarray}
This implies that for $\theta\in[\phi,2\pi]$:
$H_N(\beta+i\theta)  \in \mathrm{H}(\mathbb{C}\setminus P(\theta))$,
$-i\frac{ \pi}{2}\in P(\theta)$  only for $\theta=\theta_2$  and
$-i\frac{5\pi}{2}\in P(\theta)$  only for $\theta=\theta_1$.
%%% LO MISMO QUE AQUÍ PERO CON REFERENCIA
%\begin{equation}\label{etda}
%\begin{array}{c}
%H_N(\beta+i\theta)  \in \mathrm{H}(\mathbb{C}\setminus P(\theta)),  \\
%-i\frac{ \pi}{2}\in P(\theta) \text{ only for } \theta=\theta_2 \text{ and }
%-i\frac{5\pi}{2}\in P(\theta) \text{ only for } \theta=\theta_1.
%\end{array}
%\end{equation}
%(where $\theta_1$, $\theta_2$ are given by (\ref{T1T2}))
Hence $H_N(\beta+i\theta)$ admits the following bound
\begin{equation}\label{Lon1}
|H_N(\beta+i\theta)|\leq C(\theta), \qquad \qquad \beta\in\mathcal{C}_{0}, \ \theta \neq\theta_1, \theta_2.
\end{equation}
%\begin{equation}\label{Lon1}
%|H_N(\beta+i\theta)|\leq C(\theta), \qquad \qquad \beta\in\mathcal{C}_{0,\varepsilon}^{\pm}, \ \theta \neq\theta_1, \theta_2;\ %\varepsilon\in\left[0,\frac{\pi}{4}\right].
%\end{equation}
Therefore %Lemma \ref{lexnew}  yields that
the integral (\ref{JN}) converges absolutely for $\omega\in\mathbb{C}^+$, since
\begin{eqnarray}\label{NBHnew}
|e^{-\rho\omega\sinh\beta}| &=& e^{- \rho\omega_2\cosh\beta_1}, \qquad \beta=\beta_1+i\beta_2\in\mathcal{C}_{0}.
%,\ \beta_1\in\mathbb{R}.
\end{eqnarray}
%\vspace{0.2cm}
%The statement (\ref{JNH}) and (\ref{JNana}).
%By (\ref{C0ma}), (\ref{HE}) and (\ref{Lon1}), t
To prove statement (\ref{JNH}) it suffices to check that for any $\delta>0$ and for any $\theta\neq\theta_1,\theta_2$, the integral
%\begin{equation*}%\label{Nsv}
   $\displaystyle \int\limits_{\mathcal{C}_0}
    \frac{\partial}{\partial\omega} \Big[e^{-\rho \omega \sinh\beta} \Big]\cdot H_N(\beta+i\theta)  \  \mathrm{d}\beta$
%\end{equation*}
converges absolutely and uniformly with respect to $\omega\in\mathbb{C}^+$, $\mathrm{Im}\ \omega\geq\delta>0$. It follows from estimate (\ref{Lon1}) and from (\ref{NBHnew}).
%%%%since
%%%%$$
%%%%\left| \frac{\partial}{\partial\omega} \Big[e^{-\rho \omega \sinh\beta} \Big]  \right|
%%%%= \left| - \rho\sinh\beta\ e^{-\rho \omega \sinh\beta} \right|
%%%%= \rho \left|\sinh\beta\right| e^{-\rho \omega_2 \cosh\beta_1}, \qquad \beta\in\mathcal{C}_0.
%%%%$$

%\vspace{0.8cm}

\noindent \textbf{ii)} Making the change of variable $\beta\mapsto\beta+i\frac{\pi}{2}$ for $\beta\in\gamma_1$ and $\beta\mapsto\beta+i\frac{5\pi}{2}$ for $\beta\in\gamma_2$ we reduce the integral (\ref{JN}) to the integral (\ref{repJd}). This representation (in contrast to (\ref{JN})) admits a  continuation to
$\overline{\mathbb{C}^+}$ since the integrand $Z_N$ (in contrast to $H_N$) admits the estimate
\begin{eqnarray}\label{BZ}
        |Z_N(\beta+i\theta)| &\leq& C(\theta) e^{-2q|\beta|}% = C(\theta) e^{-\frac{\pi}{\Phi}|\beta|}
        ,   \qquad \beta\in\mathbb{R},\ \theta\neq\theta_1,\theta_2,
\end{eqnarray}
by (\ref{ZN}). Hence the integral (\ref{repJd}) remains convergent for $\omega\in\mathbb{R}$. Let us prove that $\mathcal{J}_{d} (\cdot,\cdot,\omega)$ is continuous in $\omega\in\overline{\mathbb{C}^+}$. %Let us note that i
If $\omega_k\rightarrow\overline{\omega}$ when $k\rightarrow +\infty$, $\omega_k\in\overline{\mathbb{C}^+}$, %$\overline{\omega}\in\overline{\mathbb{C}^+}$,
then
$
\displaystyle\lim\limits_{k\rightarrow+\infty} e^{i\rho\omega_k \cosh\beta}\ Z_N(\beta+i\theta)
        = e^{i\rho\overline{\omega} \cosh\beta} Z_N(\beta+i\theta)$, % \qquad
        $\beta\in\mathbb{R}
$
and
$
|e^{i\rho\omega_k \cosh\beta}\ Z_N(\beta+i\theta)| \leq C(\theta) \ e^{-2q|\beta|}$, % \qquad
$k\in\mathbb{N}
$,
because of $\omega_k\in\overline{\mathbb{C}^+}$ and  (\ref{BZ}). Therefore by the Lebesgue's Dominated Convergence Theorem
\begin{equation}\label{isla}
\lim\limits_{k\rightarrow+\infty} \mathcal{J}_{d} (\cdot,\cdot,\omega_k) = \mathcal{J}_{d} (\cdot,\cdot,\overline{\omega}).
\end{equation}
It means that $\mathcal{J}_{d} (\cdot,\cdot,\omega)\in\mathrm{C}\left(\overline{\mathbb{C}^+}\right)$.

\noindent \textbf{iii)} %\underline{\textbf{PENDIENTE, >C\'OMO SE PRUEBA LO QUE FALTA?}}.
Estimate (\ref{bjd}) follows from (\ref{repJd}), (\ref{BZ}), the fact that
$\left|e^{i\rho(\omega_1+i\omega_2) \cosh\beta}\right| = e^{-\rho\omega_2 \cosh\beta}$ for any $\beta\in\mathbb{R}$
and $\omega_2\geq 0$.

%\vspace{0.8cm}

\noindent \textbf{iv)} Statement (\ref{limJd})  follows from (\ref{isla}) and (\ref{bjd}).
\end{proof}

%\vspace{1cm}

We proceed to the extension of $\widehat{u}_d(\cdot,\cdot,\omega)$ %defined by (\ref{udc})
to $\overline{\mathbb{C}^+}$. Let
\begin{eqnarray}\label{gR}
        \widehat{g}(\omega_1) &:=& F_{s\rightarrow\omega_1}\Big[e^{-i\omega_0 s} f(s)\Big],
                                   \qquad \omega_1\in\mathbb{R},
\end{eqnarray}
        where $f$ is given by (\ref{f}), $F_{s\rightarrow\omega_1}[\cdot]$ denotes the Fourier transform in the sense of $\mathcal{S}'$ associated to the classical Fourier transform  (\ref{TF}), and
\begin{eqnarray}\label{g1}
        \widehat{g}_1(\omega_1) &:= & i \ \widehat{f'}(\omega_1-\omega_0),
                                \qquad \omega_1\in\mathbb{R}.
\end{eqnarray}
Obviously
            \begin{equation}\label{rg1}
            g_1(s) = i e^{-i\omega_0 s}f'(s), \qquad s\in\mathbb{R}.
            \end{equation}
%        \begin{eqnarray}\label{ng1}
%        \widehat{g}_1(\omega_1)  = i F_{s\rightarrow\omega_1}\Big[e^{-i\omega_0 s} f'(s)\Big],
%                                \qquad \omega_1\in\mathbb{R}.
%        \end{eqnarray}
The following lemma is proved in Appendix A2.

\begin{lem}\label{propg}
\begin{enumerate}[i)]

  \item $\widehat{g}(\omega_1)$ admits an analytic continuation to $\mathbb{C}^+$ that is there exists the limit
            \begin{equation}\label{extg}
            \widehat{g}(\omega_1) = \lim\limits_{\omega_2\rightarrow 0+} \widehat{g}(\omega_1+i\omega_2)
            %,\qquad            \text{in }\ \mathcal{S}'(\mathbb{R}).
            \end{equation}
            in the sense of       $\mathcal{S}'(\mathbb{R})$. (We denote this analytic continuation also by $\widehat{g}(\omega)$, $\omega\in\overline{\mathbb{C}^+}$ and we say that it is the \emph{Fourier-Laplace transform of } $e^{-i\omega_0 s} f(s)$).
            Moreover
                % The function $\widehat{g}(\omega)\in\mathrm{H}(\mathbb{C}^+)$ and
            there exists $C>0$ such that
            \begin{equation}\label{cotag}
            \Big|\widehat{g}(\omega)\Big|\leq C(\mathrm{Im }\ \omega)^{-1}, \qquad \omega\in\mathbb{C}^+.
            \end{equation}

\item $\widehat{g}_1(\omega_1)$ admits an analytic continuation $\widehat{g}_1(\omega)$ to $\mathbb{C}$, that is
            \begin{equation}\label{g1Ana}
            \widehat{g}_1(\omega)\in \mathrm{H}(\mathbb{C})
            \end{equation}
            and for any $\omega_2\in\mathbb{R}$
            \begin{equation}\label{g1AnaR}
            \widehat{g}_1(\omega_1+i\omega_2)\in \mathcal{S}(\mathbb{R}_{\omega_1}), \qquad
            \left|\frac{\partial^{(k)}\ }{\partial\omega^k} \widehat{g}_1(\omega_1+i\omega_2)\right| \leq C_{k,N}(1+|\omega_1|)^{-N}.%, \quad \omega\in\mathbb{C}.
            \end{equation}
            Moreover,
            \begin{equation}\label{gg1}
            \widehat{g}(\omega)  =   \frac{\widehat{g}_1(\omega)}{\omega-\omega_0}, \ \omega\in\mathbb{C}^+; \qquad
            \widehat{g}(\omega_1)  =   \frac{\widehat{g}_1(\omega_1)}{\omega_1-\omega_0+i0}, \ \omega_1\in\mathbb{R}
            \end{equation}
            and
            \begin{equation}\label{gC}
            \widehat{g}(\omega)\in\mathrm{C}^\infty(\mathbb{R}\setminus\{\omega_0\}).
            \end{equation}

%    \item  The function
%            \begin{equation}\label{dFIg1}
%            g_1(s):= F^{-1}_{\omega\rightarrow s}\Big[\widehat{g}_1(\omega)\Big],
%            \end{equation}
%            admits the following representation
%            \begin{equation}\label{ng1}
%            g_1(s) = i e^{-is\omega_0}f'(s), \qquad s\in\mathbb{R}.
%            \end{equation}
\end{enumerate}
\end{lem}

%\newpage
\vspace{1cm}

%\vspace{1.0cm}
%\newpage
%%%%%%%%%%%%%%%%%%%%%%%%%%%%%%%%%%%%%%%%%%%%%%%%%%%%%%%%%%%%%%%%%

Now we are able to extend $\widehat{u}_{d}(\cdot,\cdot,\omega)$ to $\overline{\mathbb{C}^+}$.

\begin{prop} The function $\widehat{u}_{d}(\rho,\theta,\omega)$ posseses the following properties
\begin{enumerate}[i)]

  \item $\widehat{u}_{d}(\cdot,\cdot,\omega)\in \mathrm{H}(\mathbb{C}^+)$

  \item The  estimate
        \begin{eqnarray}\label{bud}
                |\widehat{u}_{d}(\rho,\theta,\omega)| &\leq& C(\theta) \omega_2^{-1},   \qquad \omega\in\mathbb{C}^+,
        \end{eqnarray}
        holds.

  \item  For any $\omega_1\in\mathbb{R}$, there exists the limit in  $\mathcal{S}'(\mathbb{R}_{\omega_1})$:
  %$\lim\limits_{\omega_2 \rightarrow 0+} \widehat{u}_{d}(\cdot,\cdot,\omega_1+i\omega_2)$
        \begin{eqnarray}\label{limud}
        \widehat{u}_{d}(\cdot,\cdot,\omega_1) &:=&
        \lim\limits_{\omega_2 \rightarrow 0+} \widehat{u}_{d}(\cdot,\cdot,\omega_1+i\omega_2), \qquad \omega_1\in\mathbb{R}.
        \end{eqnarray}

%  \item $\widehat{u}_{d}(\cdot,\cdot,\omega_1)$ admits the following representation
%        \begin{eqnarray}\label{udgr}
%        \widehat{u}_{d}(\cdot,\cdot,\omega_1)
%        &=& \dfrac{i}{4\Phi}\ \widehat{g}(\omega_1) \mathcal{J}_d(\cdot,\cdot,\omega_1) \qquad \omega_1\in\mathbb{R}.
%        \end{eqnarray}

\end{enumerate}
\end{prop}

\begin{proof} From (\ref{udc}), (\ref{JN}) and (\ref{gg1}) we infer that
\begin{eqnarray}\label{udj}
\widehat{u}_{d}(\rho,\theta,\omega)
&=& \dfrac{i}{4\Phi}\cdot \frac{\widehat{g}_1(\omega)}{\omega-\omega_0} \mathcal{J}_d(\rho,\theta,\omega), \qquad \omega\in\mathbb{C}^+.
\end{eqnarray}
Hence, the statement i) follows from  Lemma \ref{propg} i) and (\ref{JNH}), the estimate (\ref{bud}) follows from (\ref{g1AnaR}) and (\ref{bjd}). The existence of the limit (\ref{limud}) follows from (\ref{extg}) and (\ref{limJd}).
\end{proof}

%\vspace{1cm}
%\newpage
%%%%%%%%%%%%%%%%%%%%%%%%%%%%%%%%%%%%%%%%%%%%%%%%%%%%%%%%%%%%%%%%%%%%%%%%%%%%%%%%%%%%%%%%%%%%%%%%%%%%
%%%%%%%%%%%%%%%%%%%%%%%%%   SECCIÓN
%%%%%%%%%%%%%%%%%%%%%%%%%%%%%%%%%%%%%%%%%%%%%%%%%%%%%%%%%%%%%%%%%%%%%%%%%%%%%%%%%%%%%%%%%%%%%%%%%%%%
\section{Diffracted wave representation}
In this section we apply the inverse Fourier-Laplace transform to the function $ \widehat{u}_d(\cdot,\omega)$, $\omega\in\mathbb{C}^+$. First, we make it for the auxiliary function
   \begin{eqnarray}\label{wN}
            \widehat{\mathrm{w}}_{d}(\rho,\theta,\omega) &:=& \widehat{g}_1(\omega) \mathcal{J}_{d}(\rho,\theta,\omega), \qquad  \omega\in\overline{\mathbb{C}^+},
   \end{eqnarray}
where $\mathcal{J}_{d}$ is given by  (\ref{repJd}) and  $\widehat{g}_1(\omega)$
%is an analytic continuation (\ref{g1Ana}) of $\widehat{g}_1(\omega_1)$
is given by (\ref{g1}).

%%%%%%%%%%%%%%%%%%%%%%%%%%%%%%%%%%%%%%%%%%%%%%%%%%%%%%%%%%%%%%%%%%%%%%%%%%%%%%%%%%%%%%%%%%%%%%%%%%%%
%%%%%%%%%%%%%%%%%%%%%%%%%   SECCIÓN
%%%%%%%%%%%%%%%%%%%%%%%%%%%%%%%%%%%%%%%%%%%%%%%%%%%%%%%%%%%%%%%%%%%%%%%%%%%%%%%%%%%%%%%%%%%%%%%%%%%%
\subsection{Inverse Fourier-Laplace transform of function $\widehat{\mathrm{w}}_d(\rho,\theta,t)$}

\begin{prop}\label{LemRw} Let $f$ be a smooth function satisfying (\ref{f}). Then
\begin{enumerate}[i)]
\item There exists the inverse Fourier-Laplace transform of the function  $\widehat{\mathrm{w}}_{d}(\cdot,\cdot,\omega) $,  $\omega\in\overline{\mathbb{C}^+}$,
    $F_{\omega\rightarrow t}^{-1} \big[\widehat{\mathrm{w}}_{d}(\cdot,\cdot,\omega) \big] (t)$ , which is expressed in the following way
            \begin{eqnarray}\label{defwNd}
            \mathrm{w}_{d}(\rho,\theta,t)
%            = F_{\omega_1\rightarrow t}^{-1} \Big[\widehat{\mathrm{w}}_{d}(\rho,\theta,\omega_1) \Big] (t)
            = \dfrac{1}{2\pi} \int\limits_{\mathbb{R}} e^{-i\omega_1 t} \widehat{\mathrm{w}}_{d}(\rho,\theta,\omega_1)  \ \mathrm{d} \omega_1.
            \end{eqnarray}

\item  The function $\mathrm{w}_{d}(\rho,\theta,t)$ admits also the following representation
\begin{eqnarray}\label{repwNd}
 \mathrm{w}_{d}(\rho,\theta,t)
    &=&  i e^{-i\omega_0 t} \int\limits_{-\infty}^{+\infty}  e^{i\rho\omega_0\cosh\beta} Z_N(\beta+i\theta)
         f'(t-\rho\cosh \beta)\ \mathrm{d} \beta
\end{eqnarray}
and
\begin{eqnarray}\label{wdAna}
 \mathrm{w}_{d}(\cdot,\cdot,t)\in\mathrm{C}(\mathbb{R}), \qquad \mathrm{supp}(\mathrm{w}_{d}(\cdot,\cdot,t))\subset\overline{\mathbb{R}^+}.
\end{eqnarray}

\end{enumerate}
\end{prop}

\begin{proof} \textbf{i)} The first statement follows from the Paley-Wiener type Theorem for cones (Theorem I.5.2 in \cite{k}) since $\widehat{\mathrm{w}}_{d}(\cdot,\cdot,\omega)\in \mathrm{H}(\mathbb{C}^+)$ by (\ref{g1Ana}) and it is bounded in $\overline{\mathbb{C}^+}$ by the second inequality in (\ref{g1AnaR}) and (\ref{bjd}). The representation (\ref{defwNd}) follows from the fact that $\widehat{\mathrm{w}}_{d}(\cdot,\cdot,\omega_1) $, $\omega_1\in\mathbb{R}$ is the $\mathcal{S}'$-limit of the function $\widehat{\mathrm{w}}_{d}(\cdot,\cdot,\omega_1+i\omega_2) $, when $\omega_2\rightarrow 0+$ by (\ref{g1Ana}), (\ref{g1AnaR}) and (\ref{bjd}), (\ref{limJd}).

%\vspace{0.8cm}

\noindent \textbf{ii)} Substituting expression (\ref{repJd}) in (\ref{wN}), and then
plugging the result  for $\widehat{\mathrm{w}}_d$ into the integral (\ref{defwNd}), we obtain
%\begin{eqnarray*}
$$\mathrm{w}_d(\rho,\theta,t)
    = \dfrac{1}{2\pi}\displaystyle\int\limits_\mathbb{R} e^{-i\omega_1 t} g_1(\omega_1)
      \left[
      \int\limits_\mathbb{R} e^{i\omega_1 \rho\cosh\beta} Z_N(\beta + i\theta)\ \mathrm{d}\beta
      \right]
      \mathrm{d}\omega_1.$$
%\end{eqnarray*}
By (\ref{g1AnaR}) and (\ref{BZ}) we have
%\begin{eqnarray*}
$\left| e^{i\omega_1 (-t + \rho\cosh\beta) } g_1(\omega_1) Z_N(\beta + i\theta) \right|
\leq C_N(\theta)(1+|\omega_1|)^{-N} e^{-2q |\beta|}$, % \qquad
$(\omega_1,\beta)\in\mathbb{R}^2$.
%\end{eqnarray*}
Hence, by the Fubini Theorem, the definition of the Fourier transform (\ref{TF}) and by the Inverse Fourier Transform  Theorem we obtain
\begin{eqnarray*}
\mathrm{w}_d(\rho,\theta,t)
& = & \int\limits_\mathbb{R}
      \left[
      \dfrac{1}{2\pi} \int\limits_\mathbb{R} e^{-i\omega_1(t - \rho\cosh\beta)} g_1(\omega_1)\ \mathrm{d}\omega_1
      \right]
       Z_N(\beta + i\theta)\ \mathrm{d}\beta \\
& = & \int\limits_\mathbb{R} g_1(t - \rho\cosh\beta) Z_N(\beta + i\theta)\ \mathrm{d}\beta.
\end{eqnarray*}
Now  we obtain (\ref{repwNd}) from (\ref{rg1}). The continuity of $\mathrm{w}_d$ follows from (\ref{repwNd}), (\ref{BZ}), (\ref{f}) and from the Lebesgue's Dominated Convergence Theorem. The inclusion in (\ref{wdAna}) follows directly from (\ref{repwNd}) and (\ref{f}) since $\mathrm{supp} f'\subset [0,s_0]$ and $\rho\geq 0$.
\end{proof}
%%%%%%%%%%%%%%%%%%%%%%%%%%%%%%%%%%%%%%%%%%%%%%%%%%%%%%%%%%%%%%%%%%%%%%%%%%%%%%%%%%%%%%%%%%%%%%%%%%%%

\subsection{Proof of the diffracted wave representation}
We prove the main theorem of this paper.

\begin{teo} The diffracted wave $u_{d}(\rho,\theta,t)$ for the non-stationary  scattering NN-problem (\ref{Pus}) on the wedge $W$ %,  given by (\ref{udorig}),
%with the Neumann boundary conditions  has the
admits the following representation:
\begin{eqnarray}\label{rud}
u_{d}(\rho,\theta,t)
&=& \dfrac{ie^{-i\omega_0 t}}{4\Phi} \int\limits_{-\infty}^{+\infty}
    e^{i\omega_0\rho\cosh\beta} Z_N(\beta+i\theta) f(t-\rho\cosh\beta) \ \mathrm{d}\beta, \qquad t\in\mathbb{R},
\end{eqnarray}
where $Z_N$ and $f$ are defined by (\ref{ZN}) and (\ref{f}), respectively.
\end{teo}

\begin{proof} From (\ref{udj}), (\ref{gg1}) and (\ref{wN}) we infer that
%\begin{equation*}
%\left .
%\begin{array}{lcl}
$ u_{d}(\rho,\theta,t)
%   &=& F^{-1}_{\omega\rightarrow t}\left[\dfrac{i}{4\Phi}\ \widehat{g}(\omega)\mathcal{J}_{d}(\rho,\theta,\omega)\right] \\
%   &=& F^{-1}_{\omega\rightarrow t}\left[\dfrac{i}{4\Phi}\cdot \dfrac{\widehat{g}_1(\omega)}{\omega-\omega_0}\mathcal{J}_{d}(\rho,\theta,\omega)\right]     \\
   = F^{-1}_{\omega\rightarrow t}\left[\dfrac{i}{4\Phi(\omega_1-\omega_0)} \widehat{\mathrm{w}}_{d}(\rho,\theta,\omega)\right]$,
%								\end{array} %\quad
%  \right.$%|  \qquad
$\omega\in\mathbb{C}^+$,
%\end{equation*}
%\begin{eqnarray*}
% u_{d}(\rho,\theta,t)
%   &=& F^{-1}_{\omega\rightarrow t}\left[\dfrac{i}{4\Phi}\ \widehat{g}(\omega)\mathcal{J}_{d}(\rho,\theta,\omega)\right] \\
%   &=& F^{-1}_{\omega\rightarrow t}\left[\dfrac{i}{4\Phi}\cdot \dfrac{\widehat{g}_1(\omega)}{\omega-\omega_0}\mathcal{J}_{d}(\rho,\theta,\omega)\right]     \\
%   &=& F^{-1}_{\omega\rightarrow t}\left[\dfrac{i}{4\Phi(\omega_1-\omega_0)} \ \widehat{\mathrm{w}}_{d}(\rho,\theta,\omega)\right]
%\end{eqnarray*}
where $F^{-1}_{\omega\rightarrow t} $ is the inverse Fourier-Laplace transform in the sense of $\mathcal{S}'(\overline{\mathbb{R}^+})$.
 %This transform exists by ???.
%\begin{eqnarray}%\label{avo10}
%u_{d}(\rho,\theta,t)
%   &=&F^{-1}_{\omega_1\rightarrow t}\left[\dfrac{i}{4\Phi}\ \widehat{g}(\omega_1)\mathcal{J}_{d}(\rho,\theta,\omega_1)\right]\\
%   &=&F^{-1}_{\omega_1\rightarrow t}\left[\dfrac{i}{4\Phi}\cdot \dfrac{\widehat{g}_1(\omega_1)}{\omega_1-\omega_0+i0}\mathcal{J}_{d}(\rho,\theta,\omega_1)\right]\\
%   &=& F^{-1}_{\omega_1\rightarrow t}\left[\dfrac{i}{4\Phi(\omega_1-\omega_0+i0)} \ \widehat{\mathrm{w}}_{d}(\rho,\theta,\omega_1)\right].
%\end{eqnarray}
Since for the Fourier-Laplace transform $F_{t \rightarrow \omega}\Big[f(t)*g(t)\Big]=\widehat{f}(\omega)\widehat{g}(\omega)$, $\omega\in\mathbb{C}^+$, for $f,g\in\mathcal{S}'(\overline{\mathbb{R}^+})$   and
%\newline
$F_{t \rightarrow \omega}\Big[\Theta(t)e^{-i\omega_0 t}\Big]=-\dfrac{1}{i(\omega-\omega_0)}$, $\omega\in\mathbb{C}^+$, we obtain
%\begin{eqnarray*}%\label{avo1}
$u_{d}(\rho,\theta,t)
   =  \dfrac{1}{4\Phi} \Big[ \Theta(t) e^{-i\omega_0 t} * \mathrm{w}_{d}(\rho,\theta,t) \Big]$.
%\end{eqnarray*}
The convolution in the last integral exists in the usual sense by (\ref{wdAna}), hence
%\begin{eqnarray*}%\label{avo2}
$u_{d}(\rho,\theta,t)
   = \dfrac{1}{4\Phi}\displaystyle \int\limits_{0}^{t} e^{-i\omega_0(t-s)}  \mathrm{w}_{d}(\rho,\theta,s) \ \mathrm{d}s.$
%\end{eqnarray*}
Replacing here $\mathrm{w}_d$ by its expression  (\ref{repwNd}) we have
\begin{eqnarray}\label{und2}
u_{d}(\rho,\theta,t)
   = \dfrac{i e^{-i\omega_0 t}}{4\Phi} \int\limits_{0}^{t}  % e^{-i\omega_0(t-s)}
     \left[\
     \int\limits_{-\infty}^{+\infty} e^{i\rho\omega_0\cosh\beta} Z_N(\beta+i\theta) f'(s-\rho\cosh \beta)\ \mathrm{d} \beta
     \right]  \mathrm{d}s.
\end{eqnarray}
%By (\ref{ng1})
%Thus, by (\ref{rg1}) and (\ref{Lon1}) t
The function
$
\Theta(t-s) e^{i\rho\omega_0\cosh\beta}  Z_N(\beta+i\theta) f'(s-\rho\cosh \beta)
$
has compact support in $\mathbb{R}^2$, with respect to $(\beta,s)$ by (\ref{f}). This implies that this function is integrable. Using the Fubini Theorem we change the order of integration in  (\ref{und2}) and obtain
\begin{eqnarray*}
u_{d}(\rho,\theta,t)
  &=& \dfrac{i e^{-i\omega_0 t}}{4\Phi} \int\limits_{-\infty}^{+\infty} e^{i\rho\omega_0\cosh\beta} Z_N(\beta+i\theta)  % e^{-i\omega_0(t-s)}
     \left[\
     \int\limits_{0}^{t} f'(s-\rho\cosh \beta)\  \mathrm{d}s
     \right]  \mathrm{d} \beta\\
  &=& \dfrac{i e^{-i\omega_0 t}}{4\Phi}  \int\limits_{-\infty}^{+\infty}
       e^{i\omega_0\rho\cosh\beta} Z_N(\beta+i\theta) f(t-\rho\cosh \beta) \   \mathrm{d} \beta,
\end{eqnarray*}
by the Newton-Leibnitz Theorem. The theorem is proved.
\end{proof}

%%%%%%%%%%%%%%%%%%%%%%%%%%%%%%%%%%%%%%%%%%%%%%%%%%%%%%%%%%%%%%%%%%%%%%%%%%%%%%%%%%%%%%%%%%%%%%%%%%%%
%%%%%%%%%%%%%%%%%%%%%%%%%   SECCIÓN
%%%%%%%%%%%%%%%%%%%%%%%%%%%%%%%%%%%%%%%%%%%%%%%%%%%%%%%%%%%%%%%%%%%%%%%%%%%%%%%%%%%%%%%%%%%%%%%%%%%%
%\newpage
%\vspace{1cm}

%%%%%%%%%%%%%%%%%%%%%%%%%%%%%%%%%%%%%%%%%%%%%%%%%%%%%%%%%%%%%%%%%%%%%%%%%%%%%%%%%%%%%%%%%%%%%%%%%%%%%
%{\Large\textbf{>AQU\'I VA LO DE LAS P\'AGINAS 10-14?}}

\section{Total field for the NN-scattering problem}\label{smoothus}

In this section we give a complete solution to the problem of
plane periodic wave scattering by a NN-wedge with a smooth profile function. Let us define
\begin{eqnarray}\label{splitu}
u(y,t) = u_{in}(y,t) + u_{r}(y,t) + u_d(y,t) , \quad\qquad y\in Q,\ t\in\mathbb{R},
\end{eqnarray}
where $u_{in}$, $u_{r}$ and $u_d$ are defined by (\ref{uin}), (\ref{ur}) and (\ref{rud}) respectively.

Define $\mathcal{E}_{\varepsilon,N}$ as a space of the functions $u(y,t)\in C(\overline{Q\times\mathbb{R}^+})$ such that $\nabla u(y,t)\in C(\dot{\overline{Q}}\times \overline{\mathbb{R}^+})$ and  the  norm
\begin{equation}\label{normE}
\|u\|_{\varepsilon,N}
:= \sup\limits_{t\geq 0}\left[\sup\limits_{y\in\overline{Q}}|u(y,t)|
 + \sup\limits_ {y\in\dot{Q}} (1+t)^{-N} \left\{y\right\}^\varepsilon \Big|\nabla_y u(y,t)\Big|\right]<\infty, \quad N\geq0,
\end{equation}
is finite. Here
$\left\{y\right\}:=\dfrac{\left|y\right|}{1+\left|y\right|},$ $y\in\mathbb{R}^2$ and $\dot{Q}:=\overline{Q}\setminus \left\{0\right\}$.

\begin{teo}\label{solu}
Let the incident wave profile $f(s)$ be a smooth function satisfying (\ref{f}). Then the function $u$ is a classical solution to system (\ref{NP}), (\ref{ic}), belonging to the space
$C^\infty(  \dot{\overline{Q}}\times \overline{\mathbb{R}^+})
       \cap C(\overline{Q}\times \overline{\mathbb{R}^+})
       \cap \mathcal{E}_{1-2q,1-2q}$
and it is a unique solution in this space.
\end{teo}

\begin{proof} First we prove directly (without using the inverse Fourier transform) that $u$ satisfies the system. Obviously, $u_{in}$ satisfies the D'Alembert equation in $\mathbb{R}^2\times\mathbb{R}$ by (\ref{uin}). The function $u_r$ also satisfies the same equation in the classical sense but only in $\mathbb{R}^+\times\Theta\times\mathbb{R}$ since it has discontinuities on the critical rays $l_{1,2}$. Moreover this function is a sectionally smooth one in $\overline{Q}_k$, $k=I,II,III$ where
$Q_{I}  :=\mathbb{R}^+ \times (\phi,\theta_1)\times\mathbb{R}$,
$Q_{II} :=\mathbb{R}^+ \times (\theta_1,\theta_2)\times\mathbb{R}$,
$Q_{III}:=\mathbb{R}^+ \times (\theta_2,2\pi)\times\mathbb{R}$.
All these statement follow from (\ref{ur}). Finally, the function $u_d$ also satisfies the D'Alembert equation in  $\mathbb{R}^+\times\Theta\times\mathbb{R}$. That can be verified by  direct differentiation under the integral
sign  in (\ref{rud}) using the estimate (\ref{BZ}) and the smoothness of $Z_N$ ($\theta\neq\theta_{1,2}$) and $f$. Moreover it is also sectionally smooth function in $\overline{Q}_k\setminus\{0\}$. In fact, the smoothness of $u_d$ in $Q_I \cup Q_1$, $Q_{II}$ and $Q_{III}\cup Q_2$ (see (\ref{defFi})) follows from the uniform with respect to compact sets in  $\mathbb{R}^+\times\Theta\times\mathbb{R}$ convergence of the integral (\ref{rud}) after differentiation with respect to  $\rho,\theta,t$. Differentiability of $u_d$ up to critical rays $l_k$, that is
the existence of the limits of derivatives when
$\theta\rightarrow \theta_k^\pm$, $k=1,2$ is proved in Appendix A3, Lemma \ref{lemjudr}. In the same lemma it is proved that the jumps of derivatives on the critical rays $l_k$ are opposite to the jumps of $u_r$ on $l_k$, so $u_s=u_d+u_r$ is a smooth function on $l_k$, that is $u\in C^\infty (\dot{\overline{Q}}\times \overline{\mathbb{R}^+})$.

Let us check the Neumann boundary conditions for $u$. From (\ref{uin}) and (\ref{ur}) it follows that
%\begin{eqnarray*}
$\left. \dfrac{\partial\ }{\partial\theta}(u_{in} + u_r)\right|_{Q_{1,2}} = 0$.
%\end{eqnarray*}
From the representation (\ref{rud}) for $u_d$ it also follows that $\left.\dfrac{\partial\ }{\partial\theta}\ u_{d}\right|_{Q_{1,2}}=0$, since the functions $\dfrac{\partial\ }{\partial\theta}\ Z_N(\beta+i\phi)$ and
$\dfrac{\partial\ }{\partial\theta}\ Z_N(\beta+2i\pi)$ are odd.

Finally we prove that $u\in \mathcal{E}_{1-2q,1-2q}$. We have already proved that $u\in C(\overline{Q}\times\mathbb{R}^+)\cap C^\infty(\dot{\overline{Q}}\times \mathbb{R}^+)$, so we need only to check the estimate (\ref{normE}).

Since $u_{in}\in\mathcal{E}_{0,0}$ and $u_r$ satisfies the estimate (\ref{normE}) with $\varepsilon=N=1-2q$ by (\ref{ur}), it suffices only to check that the following estimates hold
%\begin{teo}\label{prop8.3LA} The function $u_d(\rho,\theta,t)$ satisfies the following estimates
\begin{eqnarray}
|u_d(\rho,\theta,t)| &\leq& C, \hspace{4cm} (\rho,\theta,t)\in \overline{Q\times\mathbb{R}^+}, \label{cotaud}\\
|\bigtriangledown u_d(\rho,\theta,t)| &\leq& C_\delta(1+t^{1-2q})(1+\rho^{-(1-2q)}), \qquad 0<\rho<t. \label{buDLA}
\end{eqnarray}
Estimate (\ref{cotaud}) follows from (\ref{rud}), since $Z_N$ satisfies the bound (\ref{BZ}). The proof of the estimate (\ref{buDLA}) coincides with the proof of the estimates (91)and (118) for the  DD-problem  (see Lemma 12.1, Theorem 12.2  and Proposition 14.1 in \cite{la}) since $Z_N$ satisfies the estimate of type (33) in \cite{la} by (\ref{BZ}).

Finally the uniqueness of the solution $u$ in the space $\mathcal{E}_{1-2q,1-2q}$ is proved in the same wave as
the uniqueness for the DD-scattering (see Corollary 8.4 in \cite{kmm}). The theorem is proved.
\end{proof}

%\vspace{1cm}

\begin{obs}\label{obsac}  Let $s_0=0$ in (\ref{f}) i.e.
\begin{equation}\label{fH}
f(s)=\mathcal{H}(s), \qquad s\in\mathbb{R},
\end{equation}
where $\mathcal{H}(s)$ is the Heaviside function. In this case,  formula (\ref{rud}) for the diffracted wave takes the form
%\begin{eqnarray}\label{udH}
$$u_{d}(\rho,\theta,t)
= \dfrac{ie^{-i\omega_0 t}}{4\Phi} \int\limits_{-\mathrm{ac}(\frac{t}{\rho})}^{\mathrm{ac}(\frac{t}{\rho})}
    e^{i\omega_0\rho\cosh\beta} Z_N(\beta+i\theta) \ \mathrm{d}\beta, \qquad t\in\mathbb{R},$$
%\end{eqnarray}
where
\begin{equation}\label{defac}
\mathrm{ac}(x):= \left\{
    \begin{array}{ccc}
      \ln(x+\sqrt{x^2 -1}), &  & x\geq 1, \\% && \\
      0, &  & x<1.
    \end{array}
    \right.
\end{equation}
All the statements of Theorem \ref{solu} remains valid with the exception of the front continuities
of $u_{in}$, $u_r$ and $u_d$, which have jumps generated by the jump of $f$ in $0$.
\end{obs}

%\vspace{0.5cm}

\begin{obs}\label{obsalpha} A solution for the NN-scattering problem (\ref{NP}), (\ref{ic}) is expressed by (\ref{splitu}) with $u_d$ given by (\ref{rud}) not only for the wedge of the magnitude $0<\phi<\pi$ and $\alpha$ satisfying (\ref{alfa}) but also for $\phi=0$ and arbitrary $\alpha$ (in the case of $\phi=\pi$ the diffracted wave $u_d\equiv 0$). It is checked directly by substituting the function $u$ into the system (\ref{NP}). Moreover, in any case the Theorem \ref{solu} holds.
\end{obs}

%%%%%%%%%%%%%%%%%%%%%%%%%%%%%%%%%%%%%%%%%%%%%%%%%%%%%%%%%%%%%%%%%%%%%%%%%%%%%%%%%%%%%%%%%%%%%%%%%%%%
%%%%%%%%%%%%%%%%%%%%%%%%%   SECCIÓN
%%%%%%%%%%%%%%%%%%%%%%%%%%%%%%%%%%%%%%%%%%%%%%%%%%%%%%%%%%%%%%%%%%%%%%%%%%%%%%%%%%%%%%%%%%%%%%%%%%%%
%\vspace{1cm}
%\newpage
\section{Limiting Amplitude Principle and the rate of convergence to the limiting amplitude}\label{secLimAm}

In this section we prove that the amplitude of the solution $u(\rho,\theta,t)$ to the nonsationary NN-scattering problem (\ref{NP}) converges to a limiting amplitude as $t \rightarrow \infty$. This limiting amplitude is a well known solution to the stationary diffracted problem (see for example \cite{babich}). %Namely we prove the following theorem.

\begin{defn} Define the limiting amplitude for the incident, reflected and diffracted waves
\begin{eqnarray}\label{Ainrd}
\left\{
\begin{array}{rll}
A_{in}(\rho,\theta)  &:=& e^{i\omega_0\rho\cos(\theta-\alpha)}\\ \\
A_{r} (\rho,\theta)  &:=&	\left\{
										\begin{array}{cl}
										e^{i\omega_0\rho\cos(\theta-\theta_1)}, & \phi\leq \theta< \theta_1\\
										0, & \theta_1 \leq\theta\leq \theta_2\\
										e^{i\omega_0\rho\cos(\theta-\theta_2)}, & \theta_2<\theta\leq
										2\pi
										\end{array}
							 \right. \\  \\
A_{d}(\rho,\theta)	 &:=&  \dfrac{i}{4\Phi}\displaystyle
				  \int_{-\infty}^{+\infty} e^{i\omega_0\rho\cosh\beta} Z_N(\beta+i\theta)\ \mathrm{d}\beta
\end{array}
\right|\  \rho>0,  \ \theta\in\Theta, %\omega\in\mathbb{C}^+.
\end{eqnarray}
where %$\theta_1$, $\theta_2$ are given by (\ref{T1T2})  and
$Z_N$ is given by (\ref{ZN}) and let
\begin{eqnarray}\label{sumA}
A_\infty(\rho,\theta) := A_{in}(\rho,\theta) + A_r(\rho,\theta) + A_d(\rho,\theta).
\end{eqnarray}
\end{defn}

%\vspace{0.5cm}

Introduce the contour
\begin{eqnarray}\label{Cma}
\mathcal{C}_1^+ &:=& \left[\mathcal{C}_1+i\frac{\pi}{4}\right] \cup \left[-\mathcal{C}_1-i\frac{13\pi}{4}\right],
\end{eqnarray}
where $\mathcal{C}_1$ is contour  (\ref{C1}). The orientation of $\mathcal{C}_1^+$  is showed in Figure 5, (cf (35) in \cite{la}).

%%%%%%%%%%%%%%%%%%%%%%%%%%%%%%%%%%%%%%%%%%%%%%%%%%%%%%%%%%%%%%%%%%%%%%%%%%%%%%%%%%%%%%%%%%%%%
%%%%%%%%%%%%%%%%%%%%%%%%%%%%%%%%%%%%%%%%%%%%%%%%%%%%%%%%%%%%%%%%%%%%%%%%%%%%%%%%%%%%%%%%%%%%%
%%%%%%%%%%%%%%%%%%%%%%%%%%%%%%%%%%%%%%%%%%%%%%%%%%%%%%%%%%%%%%%%%%%%%%%%%%%%%%%%%%%%%%%%%%%%%
%%%%%%%%%%%%%%        FIGURA 5
%\vspace{0.5cm}
%\newpage
{\scriptsize
\setlength{\unitlength}{0.8mm}
\begin {picture} (200,130)

%tochki koordinatnie na vert. osi

\put(80,95) {\circle*{0.7}}
\put(80,80) {\circle*{0.7}}
\put(80,87.5) {\circle*{0.7}}
\put(80,95) {\circle*{0.7}}
\put(80,72.5) {\circle*{0.7}}
\put(80,27.5) {\circle*{0.7}}
\put(80,20) {\circle*{0.7}}
\put(80,65) {\circle*{0.7}}
\put(80,50) {\circle*{0.7}}
\put(80,5) {\circle*{0.7}}
\put(80,-10){\circle*{0.7}}
\put(80,12.5) {\circle*{0.7}}
\put(80,110) {\circle*{1}}

%Oboznacheniya tochek na vert. osi

\put(80.5,95){$0$}
\put(81,110){$\frac{\pi}{2}$}
\put(81,87.5){$-\frac{\pi}{4}$}
\put(81,72.5){$-\frac{3\pi}{4}$}
\put(81,27.5){$-\frac{9\pi}{4}$}
\put(81,12.5){$-\frac{11\pi}{4}$}
\put(81,65){$-\pi$}
\put(81,50){$-\frac{3\pi}{2}$}
\put(81,5){$-3\pi$}
\put(81,-10){$-\frac{7\pi}{2}$}

%Oboznacheiya tochek na vert. osi -5\pi/2 y -\pi/2
 \put(81,18.5){$-\frac{5\pi}{2}$}
 \put(81,79){$-\frac{\pi}{2}$}

%asimptoti krivij
%\put (0,110){\line(1,0){160}}
%\put (0,50){\line(1,0){80}}
 %\put(0,-10){\line(1,0){80}}

%pervaya krivaya \gamma(0)
% \qbezier(80,95)(85,102)(90,104)
%Razriv krivoy dlia nulia
\qbezier(82,98)(85,102)(90,104)
\qbezier(90,104)(100,105.5)(120,106.5)
\qbezier(120,106.5)(140,107)(150,107.3)
\qbezier(80,95)(75,88)(70,86)
 \qbezier(70,86)(60,84.5)(40,83.5)
\qbezier(40,83.5)(20,83)(10,82.7)
%shtrijovka pervoy polosi v pravoy chasti
%razriv pervoy shtrijovki.
\put(82,68){\line(0,1){3}}
\put(82,76){\line(0,1){3}}
\put(82,83){\line(0,1){3}}
%\put(82,91){\line(0,1){7}}
\put(82,91){\line(0,1){3}}
%razriv vtoroy shtrijovki
% \put(84,69){\line(0,1){30}}
 \put(84,70){\line(0,1){2}}
 \put(84,76.5){\line(0,1){2}}

 \put(84,83){\line(0,1){3}}
 \put(84,91){\line(0,1){9}}

%\put(86,71.5){\line(0,1){30}}
%razriv tret'ey shtrijovki.
%\put(86,71.5){\line(0,1){30}}

%\put(86,71.5){\line(0,1){1}}
%\put(86,77){\line(0,1){1}}

\put(86,92){\line(0,1){10}}

 \put(88,73){\line(0,1){30}}
\put(90,74){\line(0,1){30}}
 \put(92,74){\line(0,1){30}}
\put(94,74.5){\line(0,1){30}}
 \put(96,75){\line(0,1){30}}
\put(98,75){\line(0,1){30}}
 \put(100,75){\line(0,1){30}}
\put(102,75.6){\line(0,1){30}}
 \put(104,75.5){\line(0,1){30}}
 \put(106,76){\line(0,1){30}}
 \put(108,76.2){\line(0,1){30}}
 \put(110,76.2){\line(0,1){30}}
 \put(110,76.3){\line(0,1){30}}
 \put(112,76.4){\line(0,1){30}}
  \put(114,76.4){\line(0,1){30}}
   \put(116,76.5){\line(0,1){30}}
    \put(118,76.5){\line(0,1){30}}
     \put(120,76.6){\line(0,1){30}}

 \put(122,76.6){\line(0,1){30}}
  \put(124,76.7){\line(0,1){30}}
   \put(126,76.7){\line(0,1){30}}
    \put(128,76.8){\line(0,1){30}}
     \put(130,76.8){\line(0,1){30}}
 \put(132,76.9){\line(0,1){30}}
  \put(134,76.9){\line(0,1){30}}
   \put(136,77){\line(0,1){30}}
    \put(138,77){\line(0,1){30}}
     \put(140,77.1){\line(0,1){30}}
 \put(142,77.1){\line(0,1){30}}
  \put(144,77.2){\line(0,1){30}}
   \put(146,77.2){\line(0,1){30}}
    \put(148,77.3){\line(0,1){30}}
     \put(150,77.3){\line(0,1){30}}
%shtrijovka pervoy polosi v levoy chasti
\put(78,62){\line(0,1){30}} \put(76,60.5){\line(0,1){30}}
\put(74,58.5){\line(0,1){30}}
 \put(72,57){\line(0,1){30}}
\put(70,56){\line(0,1){30}}
 \put(68,56){\line(0,1){30}}
\put(66,55.5){\line(0,1){30}}
 \put(64,55){\line(0,1){30}}
\put(62,55){\line(0,1){30}}
%razriv 10 shtrijovki dlia kontura y bukvi \gamma_1
\put(60,55){\line(0,1){30}}
% \put(60,80){\line(0,1){5}}
\put(58,54.4){\line(0,1){30}}
 \put(56,54.5){\line(0,1){30}}
 \put(54,54){\line(0,1){30}}
 \put(52,53.8){\line(0,1){30}}
 \put(50,53.8){\line(0,1){30}}
 \put(48,53.8){\line(0,1){30}}
 \put(46,53.6){\line(0,1){30}}
  \put(44,53.6){\line(0,1){30}}
   \put(42,53.5){\line(0,1){30}}
    \put(40,53.5){\line(0,1){30}}
     \put(38,53.4){\line(0,1){30}}
 \put(36,53.4){\line(0,1){30}}
  \put(34,53.3){\line(0,1){30}}
   \put(32,53.3){\line(0,1){30}}
    \put(30,53.2){\line(0,1){30}}
     \put(28,53.2){\line(0,1){30}}
 \put(26,53.1){\line(0,1){30}}
  \put(24,53.1){\line(0,1){30}}
   \put(22,53){\line(0,1){30}}
    \put(20,53){\line(0,1){30}}
     \put(18,52.9){\line(0,1){30}}
 \put(16,52.9){\line(0,1){30}}
  \put(14,52.8){\line(0,1){30}}
   \put(12,52.8){\line(0,1){30}}
    \put(10,52.7){\line(0,1){30}}
%vtoraya liniya

%Razriv okolo -3\pi74
\qbezier(80,65)(85,72)(85,72)
\qbezier(87.5,73)(88,73.5)(90,74)
 \qbezier(90,74)(100,75.5)(120,76.5)
\qbezier(120,76.5)(140,77)(150,77.3)

%%VTORAYA LINIYA, LEVAYA POLOVINA

\qbezier(80,65)(75,58)(70,56) \qbezier(70,56)(60,54.5)(40,53.5)
\qbezier(40,53.5)(20,53)(10,52.7)
%3 liniya
 \qbezier(80,35)(85,42)(90,44)
\qbezier(90,44)(100,45.5)(120,46.5)
\qbezier(120,46.5)(140,47)(150,47.3)
\qbezier(80,35)(75,28)(70,26) \qbezier(70,26)(60,24.5)(40,23.5)
\qbezier(40,23.5)(20,23)(10,22.7)
%shtrijjovka vtoroy polosi, pravaya chast'
%Razriv pervoy shtrijovki.

% \put(82,8){\line(0,1){30}}
 \put(82,8){\line(0,1){4.5}}
 \put(82,15){\line(0,1){2.3}}
 \put(82,24){\line(0,1){3}}
 \put(82,30){\line(0,1){7}}
%Razriv vtoroy shtrijovki

 \put(84,11){\line(0,1){1.5}}
\put(84,14.5){\line(0,1){3}}
 \put(84,24){\line(0,1){3}}
 \put(84,32){\line(0,1){7}}

%Razriv tret'ey shtrijovki

%\put(86,17){\line(0,1){1}}
%\put(86,24){\line(0,1){1}}

\put(86,33){\line(0,1){8}}

%%%%%%%%%%%%%%%%%%%%%%%%%%%%%%%%
%Razriv chervertoy shtrijovki

 \put(88,17){\line(0,1){2}}
 \put(88,24){\line(0,1){3}}
 \put(88,32){\line(0,1){11}}

% \put(88,20){\line(0,1){10}}
% \put(88,32){\line(0,1){5}}

\put(90,14){\line(0,1){30}}
 \put(92,14){\line(0,1){30}}
\put(94,14.5){\line(0,1){30}}
 \put(96,15){\line(0,1){30}}
\put(98,15){\line(0,1){30}}
 \put(100,15){\line(0,1){30}}
\put(102,15.6){\line(0,1){30}}
 \put(104,15.5){\line(0,1){30}}
 \put(106,16){\line(0,1){30}}
 \put(108,16.2){\line(0,1){30}}
 \put(110,16.2){\line(0,1){30}}
 \put(110,16.3){\line(0,1){30}}
 \put(112,16.4){\line(0,1){30}}
  \put(114,16.4){\line(0,1){30}}
   \put(116,16.5){\line(0,1){30}}
    \put(118,16.5){\line(0,1){30}}
     \put(120,16.6){\line(0,1){30}}

 \put(122,16.6){\line(0,1){30}}
  \put(124,16.7){\line(0,1){30}}
   \put(126,16.7){\line(0,1){30}}
    \put(128,16.8){\line(0,1){30}}
     \put(130,16.8){\line(0,1){30}}
 \put(132,16.9){\line(0,1){30}}
  \put(134,16.9){\line(0,1){30}}
   \put(136,17){\line(0,1){30}}
    \put(138,17){\line(0,1){30}}
     \put(140,17.1){\line(0,1){30}}
 \put(142,17.1){\line(0,1){30}}
  \put(144,17.2){\line(0,1){30}}
   \put(146,17.2){\line(0,1){30}}
    \put(148,17.3){\line(0,1){30}}
     \put(150,17.3){\line(0,1){30}}

%shtrijovka vtoroy polosi v levoy chasti
\put(78,2){\line(0,1){30}}
 \put(76,0.5){\line(0,1){30}}
\put(74,-2.1){\line(0,1){30}}
 \put(72,-3){\line(0,1){30}}
\put(70,-4){\line(0,1){30}}
 \put(68,-4){\line(0,1){30}}
\put(66,-4.5){\line(0,1){30}}
 \put(64,-5){\line(0,1){30}}
\put(62,-5){\line(0,1){30}}

\put(60,-5.2){\line(0,1){30}}

\put(58,-5.6){\line(0,1){30}}
 \put(56,-5.5){\line(0,1){30}}
 \put(54,-6){\line(0,1){30}}
 \put(52,-6.2){\line(0,1){30}}
 \put(50,-6.2){\line(0,1){30}}
 \put(48,-6.2){\line(0,1){30}}
 \put(46,-6.4){\line(0,1){30}}
  \put(44,-6.4){\line(0,1){30}}
   \put(42,-6.5){\line(0,1){30}}
    \put(40,-6.5){\line(0,1){30}}
     \put(38,-6.6){\line(0,1){30}}
 \put(36,-6.6){\line(0,1){30}}
  \put(34,-6.7){\line(0,1){30}}
   \put(32,-6.7){\line(0,1){30}}
    \put(30,-6.8){\line(0,1){30}}
     \put(28,-6.8){\line(0,1){30}}
 \put(26,-6.9){\line(0,1){30}}
  \put(24,-6.9){\line(0,1){30}}
   \put(22,-7){\line(0,1){30}}
    \put(20,-7){\line(0,1){30}}
     \put(18,-7.1){\line(0,1){30}}
 \put(16,-7.1){\line(0,1){30}}
  \put(14,-7.2){\line(0,1){30}}
   \put(12,-7.2){\line(0,1){30}}
    \put(10,-7.3){\line(0,1){30}}
%4 liniya

\qbezier(80,65)(85,72)(85,72)
\qbezier(87.5,73)(88,73.5)(90,74)

 \qbezier(90,74)(100,75.5)(120,76.5)
\qbezier(120,76.5)(140,77)(150,77.3)

%razriv okolo -10\pi/4

\qbezier(80,5)(85,12)(85,12)
% \qbezier(89.5,13.7)(90,14)(90,14)

%konec razriva

 \qbezier(90,14)(100,15.5)(120,16.5)
\qbezier(120,16.5)(140,17)(150,17.3)

%levaya polovina

\qbezier(80,5)(75,-2)(70,-4) \qbezier(70,-4)(60,-5.5)(40,-6.5)
\qbezier(40,-6.5)(20,-7)(10,-7.3)

%oboznacheniya konturov {\cal C}, {\cal C}_1+i\pi/4, -{\cal C}_1-13\pi/4
%\put(90,50){${\cal C}_+$}
% \put(120,50){${\cal C}_1+i\pi/4$}
% \put(30,40){$-{\cal C}_{1}-13\pi/4$}

%vektori, ukazivayuschie konturi
%\put(116,50){\vector(-2,-1){15}}
% \put(116,50){\vector(1,-2){11.5}}

% \put(116,50){\vector(1,3){12.5}}
%\put(90,50){\vector(-2,-1){30}}
% \put(90,50){\vector(2,-1){10}}
 %\put(28,40){\vector(1,3){11}}

%\put(28,40){\vector(-1,-2){14}}
 %\put(28,40){\vector(2,-1){32.5}}
%napravleniya kontura

%{\linethickness{0.3mm}\put(60,60){\vector(0,1){10}}}
%\put(120,87.5){\vector(-1,0){10}} \put(60,72.5){\vector(-1,0){10}}
%\put(40,12.5){\vector(1,0){10}}
 %\put(100,27.5){\vector(1,0){10}}
%\put(100,40){\vector(0,-1){10}}

% contur \cal C_0
% kontur \gamma_1 s razrivom dlia -\pi/2
%\linethickness{0.3mm} \put (10,80){\line(1,0){140}}

%%%%%%%%%%%%%%%%%%%%%%%%%%%%%%%%%%%%%%%%%%%%%%%%%%%%%%%%%%%%%%%%%%%%%%%%%%
%%%%%%%%%%%%%%%  LÍNEAS HORIZONTALES SUPERIORES
%%%%%%%%%%%%%%%%%%%%%%%%%%%%%%%%%%%%%%%%%%%%%%%%%%%%%%%%%%%%%%%%%%%%%%%%%%

\linethickness{0.3mm} \put (10,72.5){\line(1,0){58}}
\linethickness{0.3mm} \put (92,87.5){\line(1,0){57.8}}

%\linethickness{0.3mm} \put (100,27.5){\line(1,0){60}}
%\linethickness{0.3mm}\put (60,72.5){\line(-1,0){60}}

%%%%%%%%%%%%%%%%%%%%%%%%%%%%%
%Estoy BORRANDO LAS LINEAS MÁS GRUESAS
%\linethickness{1mm}\put(10,16){\line(1,0){62}}
%\qbezier(72,16)(83,22)(89,25)
%\linethickness{1mm}\put(89,25){\line(1,0){61}}
%%%%%%%%%%%%%%%%%%%%%%%%%%%%

%%%%%%%%%%%%%%%%%%%%%%%%%%%%%%%%%%%%%%%%%%%%%%%%%%%%%%%%%%%%%%%%%%%%%%%%%%
%%%%%%%%%%%%%%%  LÍNEAS HORIZONTALES INFERIORES
%%%%%%%%%%%%%%%%%%%%%%%%%%%%%%%%%%%%%%%%%%%%%%%%%%%%%%%%%%%%%%%%%%%%%%%%%%

\linethickness{0.3mm}\put(10,12.5){\line(1,0){58}}
\linethickness{0.3mm}\put(92,27.5){\line(1,0){57.8}}

%%%%%%%%%%%%%%%%%%%%%%%%%%%%%%%%%%%%%%%%%%%%%%%%%%%%%%%%%%%%%%%%%%%%
%%%%%% ORIENTACIÓN DE $ \mathcal{C}_1 + i\frac{  \pi}{4}$  %%%%%%%%%
%%%%%%%%%%%%%%%%%%%%%%%%%%%%%%%%%%%%%%%%%%%%%%%%%%%%%%%%%%%%%%%%%%%%
{\linethickness{0.3mm}\put(140,87.5){\vector(-1,0){20}}}
{\linethickness{0.3mm}\put(92,77.5){\vector(0,-1){20}}}
{\linethickness{0.3mm}\put(103,27.5){\vector(1,0){20}}}

%%%%%%%%%%%%%%%%%%%%%%%%%%%%%%%%%%%%%%%%%%%%%%%%%%%%%%%%%%%%%%%%%%%%
%%%%%% ORIENTACIÓN DE $ -\mathcal{C}_1 - i\frac{13\pi}{4}$  %%%%%%%%%
%%%%%%%%%%%%%%%%%%%%%%%%%%%%%%%%%%%%%%%%%%%%%%%%%%%%%%%%%%%%%%%%%%%%
{\linethickness{0.3mm}\put(63,72.5){\vector(-1,0){20}}}
{\linethickness{0.3mm}\put(68,22.5){\vector(0,1){20}}}
{\linethickness{0.3mm}\put(23,12.5){\vector(1,0){20}}}

%%%%%%%%%%%%%%%%%%%%%%%%%%%%%%%%%%%%%%%%%%%%%%%%%%%%%%%%%%%%%%%%%%%%%%%%%%
%%%%%%%%%%%%%%%  LÍNEAS VERTICALES
%%%%%%%%%%%%%%%%%%%%%%%%%%%%%%%%%%%%%%%%%%%%%%%%%%%%%%%%%%%%%%%%%%%%%%%%%%

\linethickness{0.3mm} \put(92,27.5){\line(0,1){60}}
\linethickness{0.3mm} \put(68,12.5){\line(0,1){60}}

%%%%%%%%%%%%%%%%%%%%%%%%%%%%%
%Estoy BORRANDO LAS LINEAS MÁS GRUESAS
%\linethickness{1mm}\put(89,85){\vector(1,0){61}}
%\qbezier(72,76)(83,82)(89,85)
%\linethickness{1mm}\put(10,76){\vector(1,0){62}}
%%%%%%%%%%%%%%%%%%%%%%%%%%%%%

%Oboznacheniya \gamma_1 y \gamma_2

\put(94.4,58){$ \mathcal{C}_1 + i\frac{  \pi}{4}$}
\put(47.4,38){$-\mathcal{C}_1 - i\frac{13\pi}{4}$}

%linii vertikal'nie
%\thicklines \linethickness{0.3mm}\put(100,87.5){\line(0,-1){60}}

%\linethickness{0.3mm}\put (60,72.5){\line(0,-1){60}}
% Os' vertikal'naya
%\thinlines \put(80,-12){\vector(0,1){150}}
\thinlines \put(80,-12){\vector(0,1){130}}

%oboznachenie vert. osi
% \put(75,120){$\mathrm{Re}\ \omega > 0$}

 \end{picture}
}

\vspace{0.5cm}
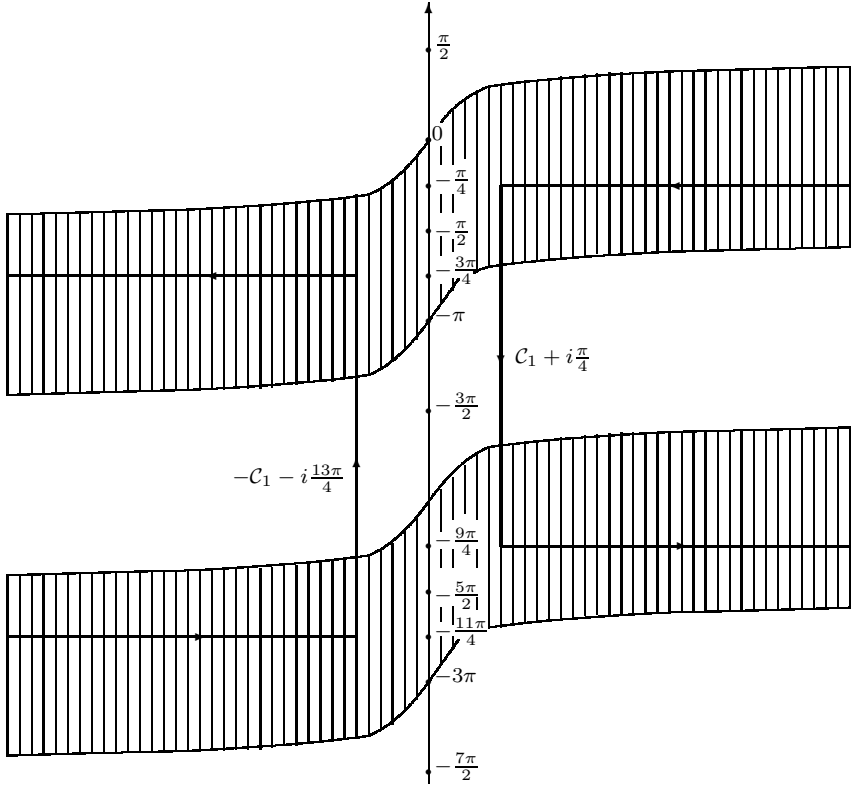
 \centerline{Figure 5. Contour $\mathcal{C}_{1}^+$.}

\vspace{0.2cm}

\begin{teo}\label{TeoRepA} Let $f$ be a function satisfying (\ref{f}) with $s_0 \geq 0$ and $u$ be a solution (\ref{splitu}) of system (\ref{NP}), (\ref{ic}). Then for $\theta\in\Theta$ (see (\ref{nt})), there exists a limit of the amplitude $A(\rho,\theta,t):= u(\rho,\theta,t) e^{i\omega_0 t} $ of the solution $u$ and
\begin{eqnarray}\label{limA}
\lim\limits_{t\rightarrow\infty}\ A(\rho,\theta,t) = A_\infty(\rho,\theta),
\end{eqnarray}
where the limiting amplitude $A_\infty$  admits the following representation
\begin{eqnarray}\label{repA}
A_\infty(\rho,\theta) &:=& \int\limits_{\mathcal{C}_1^+} e^{-\omega_0\rho\sinh\beta} H_N(\beta+i\theta) \ \mathrm{d}\beta.
\end{eqnarray}
Moreover, $A_\infty$ satisfies the  stationary scattering problem
%\begin{equation}\label{pA}
$$\left\{
	   \begin{array}{rcl}
       (\Delta - \omega_0^2)\ A_\infty(\rho,\theta)  & = &  0,  \qquad (\rho,\theta)\in Q   \\ %\\
       \dfrac{\partial\ }{\partial \mathbf{n}} A_\infty(\rho,\theta) &=& 0, \qquad (\rho,\theta)\in \partial Q
		\end{array}
\right.
$$
%\end{equation}
\end{teo}

\begin{proof} From (\ref{uin}), (\ref{ur}) and (\ref{f}) it follows that
\begin{equation}\label{limAinAr}
u_{in}(y,t) e^{i\omega_0 t}   \longrightarrow  A_{in}(\rho,\theta), \qquad
u_{r }(y,t) e^{i\omega_0 t}   \longrightarrow  A_{r }(\rho,\theta), \qquad \  t\rightarrow +\infty,
\end{equation}
uniformly in $\rho\leq\rho_0$, $\theta\in[\phi,2\pi]$. In the following lemma we also prove this convergence for the diffracted wave.
%\end{proof}

%\vspace{1cm}

\begin{lem}\label{teoasyud} \textbf{(Limiting Amplitude Principle for the diffracted wave)}.
Let $f$ be a profile function given by (\ref{f}) with $s_0\geq 0$. Then for any $\rho_0>0$ the following asymptotics hold
%\begin{eqnarray}
$A_d(\rho,\theta, t) - A_d(\rho,\theta) \longrightarrow 0$, when  $t\longrightarrow +\infty$,
%\end{eqnarray}
uniformly in $\rho\in[0,\rho_0]$ and $\theta\in\Theta$. Here
%\begin{eqnarray}\label{Adt}
$A_d(\rho,\theta, t) := u_d(\rho,\theta,t) e^{i\omega_0 t}$
%\end{eqnarray}
and $A_d(\rho,\theta)$ is defined in (\ref{Ainrd}).
\end{lem}

\begin{proof} Representation (\ref{rud}) implies that
\begin{eqnarray}\label{repAd}
A_d(\rho,\theta,t)
%= e^{i\omega_0 t} u_d(\rho,\theta,t)
&=& \dfrac{i}{4\Phi} \int\limits_{-\infty}^{+\infty}
e^{i\omega_0\rho\cosh\beta} f(t-\rho\cosh\beta) Z_N(\beta+i\theta) \ \mathrm{d}\beta.
\end{eqnarray}
It remains only to prove that
\begin{eqnarray}\label{adrtt}
A_d(\rho,\theta,t) \longrightarrow A_d(\rho,\theta), \qquad t\longrightarrow +\infty
\end{eqnarray}
uniformly with respect to $\rho\in[0,\rho_0]$ and $\theta\in\Theta$. By (\ref{Ainrd}) and (\ref{repAd})
\begin{eqnarray}\label{chi}
A_d(\rho,\theta,t) - A_d(\rho,\theta)
= \dfrac{i}{4\Phi} \int\limits_{-\infty}^{+\infty}
  e^{i\omega_0\rho\cosh\beta}\Big[ f(t-\rho\cosh\beta) - 1 \Big]Z_N(\beta+i\theta)\ \mathrm{d}\beta.
\end{eqnarray}
Let us fix $\rho_0>0$, $\theta\in[\phi,2\pi]$ and $\varepsilon>0$. Since the poles of $Z_N(\beta+i\theta)$ can be only in $\beta=0$ by (\ref{ZN}), there exists $C>0$ such that
%\begin{eqnarray*}%\label{2estre}
$|Z_N(\beta+i\theta)| \leq C$, $ \beta\in\mathbb{R}$, $|\beta|\geq 1$, $ \theta\in[\phi,2\pi]$.
%\end{eqnarray*}
Let us choose $\overline{\beta}>1$ such that $\dfrac{8C\Phi e^{-\frac{\pi}{\Phi}|\bar\beta|}}{\pi}<\varepsilon$.
%%\begin{equation*}
%%\dfrac{8C\Phi e^{-\frac{\pi}{\Phi}|\bar\beta|}}{\pi}<\varepsilon.
%%\end{equation*}
Then by (\ref{ZN}), (\ref{BZ}) and (\ref{f})
\begin{eqnarray*}%\label{bintZ}
\int\limits_{|\beta|\geq \bar\beta}
\Big|e^{i\omega_0\rho\cosh\beta}\Big[ f(t-\rho\cosh\beta) - 1 \Big]Z_N(\beta+i\theta)\Big|\mathrm{d}\beta
<    8C \int\limits_{\beta\geq\bar\beta}e^{-\frac{\pi}{\Phi}\beta} \mathrm{d}\beta
%&\leq& 8C \frac{\Phi e^{-\frac{\pi}{\Phi}\bar\beta}}{\pi}
< \varepsilon, \quad t\in\mathbb{R}.
\end{eqnarray*}
It remains only to prove the convergence to zero of the integral (\ref{chi}) over $[-\bar\beta,\bar\beta]$. We have, $\cosh\beta_1(\rho,\theta) = \dfrac{t-s_0 }{\rho_0}\geq\cosh\bar\beta$, for $t\geq s_0 +\rho_0\cosh\bar \beta$, where $\beta_1$ is a  non negative solution to the equation $\cosh\beta_1 = \dfrac{t-s_0}{\rho_0}$.
%%\begin{eqnarray}\label{b3}
%%\cosh\beta_1 &=& \dfrac{t-s_0}{\rho_0}.
%%\end{eqnarray}
This implies that $f(t-\rho\cos\beta)=1$ for $\beta\in[-\bar\beta,\bar\beta]$,
$t\geq s_0 +\rho_0\cosh\bar\beta,\ \rho\leq \rho_0$.
Hence
\begin{equation*}
\int\limits_{-\bar\beta}^{\bar\beta}
\Big|e^{i\omega_0\rho\cosh\beta}\big[ f(t-\rho\cosh\beta) - 1 \big]Z_N(\beta+i\theta)\Big| \mathrm{d}\beta
= 0 <\varepsilon, \quad t\geq s_0 +\rho_0\cosh\bar\beta,\ \rho\leq \rho_0.
\end{equation*}
%Together with (\ref{bintZ}) t
This completes the proof of the lemma.
\end{proof}
%\vspace{1cm}
%%%\begin{center}
%%%\textbf{Proof of Theorem \ref{TeoRepA}. (Continuation)}
%%%\end{center}
Let us continue the proof of Theorem \ref{TeoRepA}. Using Lemma \ref{teoasyud} we infer that (\ref{limA}) follows from (\ref{splitu}) and (\ref{limAinAr}).

Let us prove representation (\ref{repA}) for $A_\infty$. Consider
\begin{eqnarray}\label{flo0}
\overline{A} &:=& \dfrac{i}{4\Phi} \int_{\mathcal{C}^+} e^{-\omega_0\rho\sinh\beta} H_N(\beta+i\theta)\ \mathrm{d}\beta.
\end{eqnarray}
We prove that $\overline{A}=A_\infty$, where $A_\infty$ is defined by (\ref{sumA}). Let us define the contours
%\begin{eqnarray}\label{C0ma}
$\mathcal{C}_0^+ := \gamma_1^+ \cup (\gamma_1^+ - 2i\pi)$
%\end{eqnarray}
(where
%\begin{eqnarray}\label{ling1g2}
$\gamma_1^+ := \left\{\beta_1-i\frac{\pi}{4}: \beta_1\geq 1 \right\}
           \cup \left\{\beta_1+i\left[\frac{\pi}{4}\beta_1 -\frac{\pi}{2}\right]: -1\leq\beta_1\leq 1 \right\}%\\ &&
           \cup \left\{\beta_1-i\frac{3\pi}{4}: \beta_1\leq -1 \right\}$) % \nonumber
%\end{eqnarray}
and
\begin{eqnarray*}\label{C1ma}
\Gamma^+
    &:=&   \left\{\beta_1+i\left[\frac{\pi}{4}\beta_1 -\frac{\pi}{2}\right]: -1\leq\beta_1\leq 1 \right\}
           \cup \left\{1+i\beta_2: -\frac{9\pi}{4}\leq \beta_2 \- \leq -\frac{\pi}{4} \right\} \\ &&
           \cup \left\{\beta_1+i\left[\frac{\pi}{4}\beta_1 -\frac{5\pi}{2}\right]: -1\leq\beta_1\leq 1 \right\}
           \cup \left\{-1+i\beta_2: -\frac{11\pi}{4}\leq \beta_2 \leq -\frac{3\pi}{4} \right\}. %\nonumber
\end{eqnarray*}
%Their orientations are shown in the
(See Figure 6).  Note that by (\ref{Cma})
\begin{eqnarray}\label{flo1}
\int_{\mathcal{C}_1^+} e^{-\omega_0\rho\sinh\beta} H_N(\beta+i\theta)\ \mathrm{d}\beta
&=& \int_{\mathcal{C}_0^+} e^{-\omega_0\rho\sinh\beta} H_N(\beta+i\theta)\ \mathrm{d}\beta
   +\int_{\Gamma^+} e^{-\omega_0\rho\sinh\beta} H_N(\beta+i\theta)\ \mathrm{d}\beta.\nonumber \\
\end{eqnarray}
%where $\mathcal{C}_0^+, \mathcal{C}_1^+$ are the contours given by (\ref{C0ma}), (\ref{C1ma}) respectively.
Also using the Cauchy Residues Theorem we obtain
%\begin{eqnarray*}
$\displaystyle \int_{\Gamma^+} e^{-\omega_0\rho\sinh\beta} H_N(\beta+i\theta)\ \mathrm{d}\beta
=  -2i\pi \displaystyle\sum\limits_{p\in R_1(\theta)}\mathrm{Res }(e^{-\omega_0\rho\sinh\beta} H_N(\beta+i\theta), p)$,
%\end{eqnarray*}
where $R_1(\theta)$  is a set of all the poles of function $e^{-\omega_0\rho\sinh\beta} H_N(\beta+i\theta)$
lying inside  of  $\Gamma^+$. Calculating the residues by means of
(\ref{poloHNth}), (\ref{alfa}) and (\ref{defFi})
we obtain
\begin{eqnarray*}
\int_{\Gamma^+} e^{-\omega_0\rho\sinh\beta} H_N(\beta+i\theta)\ \mathrm{d}\beta
 &=& -4i\Phi\ \left\{
     \begin{array}{ll}
     e^{i\omega_0\rho \cos(\theta - \alpha)} + e^{i\omega_0\rho \cos(\theta - \theta_1)}, &  \theta\in[\phi,\theta_1),\\
     e^{i\omega_0\rho \cos(\theta - \alpha)},                                  & \theta\in[\theta_1, \theta_2],\\
     e^{i\omega_0\rho \cos(\theta - \alpha)} + e^{i\omega_0\rho \cos(\theta - \theta_2)}, &  \theta\in(\theta_2,2\pi].
     \end{array}
     \right.
\end{eqnarray*}
Hence from (\ref{flo0}), (\ref{flo1}), (\ref{adrtt})  we infer that
%\begin{equation*}%\label{sumA}
$\overline{A} := A_{in} + A_r + \overline{A}_d$,
%\end{equation*}
with
%\begin{eqnarray}
$$\overline{A}_d := \dfrac{i}{4\Phi} \displaystyle\int_{\mathcal{C}_0^+} e^{-\omega_0\rho\sinh\beta} H_N(\beta+i\theta)\ \mathrm{d}\beta.$$
%\end{eqnarray}
%Hence and from (\ref{sumA})
Therefore it suffices only to prove that
\begin{eqnarray}\label{AdAdbar}
\overline{A}_d = A_d,
\end{eqnarray}
by  (\ref{sumA}). Making the change of variable $\beta\mapsto\beta+i\frac{\pi}{2}$ and then the change of variable $\beta\mapsto\beta+2i\pi$ in the integral  over $\gamma_1^+-i\frac{3\pi}{2}$ we obtain
$\overline{A}_d
= -\dfrac{i}{4\Phi} \displaystyle \int_{\gamma_1^+ + i\frac{ \pi}{2}}  e^{i\omega_0\rho\cosh\beta}  Z_N\left(\beta+i\theta\right)\ \mathrm{d}\beta$,
%\end{eqnarray}
where $Z_N$ is defined by (\ref{ZN}). Now we are able to transform the contour $\gamma_1^+ + i\frac{\pi}{2}$ to the contour $\mathbb{R}$ using exponential decrease of $Z_N$. Namely, by the Cauchy Residue Theorem, (\ref{BZ}), the fact that $Z_N(\beta+i\theta)$ is analytic in
$\mathbb{C}\setminus P(\theta)$ (see (\ref{poloHNth}))  and (\ref{BZ}), we obtain that
%\begin{eqnarray}\label{iT1R}
$\displaystyle\int_{\gamma_1^+ + i\frac{ \pi}{2}}  e^{i\omega_0\rho\cosh\beta}  Z_N\left(\beta+i\theta\right)\ \mathrm{d}\beta
= - \int_{-\infty}^{+\infty}  e^{i\omega_0\rho\cosh\beta}  Z_N\left(\beta+i\theta\right)\ \mathrm{d}\beta.$
%\end{eqnarray}
Hence (\ref{AdAdbar}) follows.
 %from (\ref{iT1}) and (\ref{iT1R}).
Theorem \ref{TeoRepA} is proved.
\end{proof}

%%%%%%%%%%%%%%%%%%%%%%%%%%%%%%%%%%%%%%%%%%%%%%%%%%%%%%%%%%%%%%%%%%%%%%%%%%%%%%%%%%%%%%%%%%%%%
%%%%%%%%%%%%%%%%%%%%%%%%%%%%%%%%%%%%%%%%%%%%%%%%%%%%%%%%%%%%%%%%%%%%%%%%%%%%%%%%%%%%%%%%%%%%%
%%%%%%%%%%%%%%%%%%%%%%%%%%%%%%%%%%%%%%%%%%%%%%%%%%%%%%%%%%%%%%%%%%%%%%%%%%%%%%%%%%%%%%%%%%%%%
%%%%%%%%%%%%%%        FIGURA 6

%\newpage
{\scriptsize
\setlength{\unitlength}{0.8mm}
\begin {picture} (200,130)

%tochki koordinatnie na vert. osi

\put(80,95) {\circle*{0.7}}
\put(80,80) {\circle*{0.7}}
\put(80,87.5) {\circle*{0.7}}
\put(80,95) {\circle*{0.7}}
\put(80,72.5) {\circle*{0.7}}
\put(80,27.5) {\circle*{0.7}}
\put(80,20) {\circle*{0.7}}
\put(80,65) {\circle*{0.7}}
\put(80,50) {\circle*{0.7}}
\put(80,5) {\circle*{0.7}}
\put(80,-10){\circle*{0.7}}
\put(80,12.5) {\circle*{0.7}}
\put(80,110) {\circle*{1}}

%Oboznacheniya tochek na vert. osi

\put(80.5,95){$0$}
\put(81,110){$\frac{\pi}{2}$}
\put(81,87.5){$-\frac{\pi}{4}$}
\put(81,72.5){$-\frac{3\pi}{4}$}
\put(81,27.5){$-\frac{9\pi}{4}$}
\put(81,12.5){$-\frac{11\pi}{4}$}
\put(81,65){$-\pi$}
\put(81,50){$-\frac{3\pi}{2}$}
\put(81,5){$-3\pi$}
\put(81,-10){$-\frac{7\pi}{2}$}

%Oboznacheiya tochek na vert. osi -5\pi/2 y -\pi/2
 \put(81,18.5){$-\frac{5\pi}{2}$}
 \put(81,79){$-\frac{\pi}{2}$}

%asimptoti krivij
%\put (0,110){\line(1,0){160}}
%\put (0,50){\line(1,0){80}}
 %\put(0,-10){\line(1,0){80}}

%pervaya krivaya \gamma(0)
% \qbezier(80,95)(85,102)(90,104)
%Razriv krivoy dlia nulia
\qbezier(82,98)(85,102)(90,104)
\qbezier(90,104)(100,105.5)(120,106.5)
\qbezier(120,106.5)(140,107)(150,107.3)
\qbezier(80,95)(75,88)(70,86)
 \qbezier(70,86)(60,84.5)(40,83.5)
\qbezier(40,83.5)(20,83)(10,82.7)
%shtrijovka pervoy polosi v pravoy chasti
%razriv pervoy shtrijovki.
\put(82,68){\line(0,1){3}}
\put(82,76){\line(0,1){3}}
\put(82,83){\line(0,1){3}}
%\put(82,91){\line(0,1){7}}
\put(82,91){\line(0,1){3}}
%razriv vtoroy shtrijovki
% \put(84,69){\line(0,1){30}}
 \put(84,70){\line(0,1){2}}
 \put(84,76.5){\line(0,1){2}}

 \put(84,83){\line(0,1){3}}
 \put(84,91){\line(0,1){9}}

%\put(86,71.5){\line(0,1){30}}
%razriv tret'ey shtrijovki.
%\put(86,71.5){\line(0,1){30}}

%\put(86,71.5){\line(0,1){1}}
%\put(86,77){\line(0,1){1}}

\put(86,92){\line(0,1){10}}

 \put(88,73){\line(0,1){30}}
\put(90,74){\line(0,1){30}}
 \put(92,74){\line(0,1){30}}
\put(94,74.5){\line(0,1){30}}
 \put(96,75){\line(0,1){30}}
\put(98,75){\line(0,1){30}}
 \put(100,75){\line(0,1){30}}
\put(102,75.6){\line(0,1){30}}
 \put(104,75.5){\line(0,1){30}}
 \put(106,76){\line(0,1){30}}
 \put(108,76.2){\line(0,1){30}}
 \put(110,76.2){\line(0,1){30}}
 \put(110,76.3){\line(0,1){30}}
 \put(112,76.4){\line(0,1){30}}
  \put(114,76.4){\line(0,1){30}}
   \put(116,76.5){\line(0,1){30}}
    \put(118,76.5){\line(0,1){30}}
     \put(120,76.6){\line(0,1){30}}

 \put(122,76.6){\line(0,1){30}}
  \put(124,76.7){\line(0,1){30}}
   \put(126,76.7){\line(0,1){30}}
    \put(128,76.8){\line(0,1){30}}
     \put(130,76.8){\line(0,1){30}}
 \put(132,76.9){\line(0,1){30}}
  \put(134,76.9){\line(0,1){30}}
   \put(136,77){\line(0,1){30}}
    \put(138,77){\line(0,1){30}}
     \put(140,77.1){\line(0,1){30}}
 \put(142,77.1){\line(0,1){30}}
  \put(144,77.2){\line(0,1){30}}
   \put(146,77.2){\line(0,1){30}}
    \put(148,77.3){\line(0,1){30}}
     \put(150,77.3){\line(0,1){30}}
%shtrijovka pervoy polosi v levoy chasti
\put(78,62){\line(0,1){30}} \put(76,60.5){\line(0,1){30}}
\put(74,58.5){\line(0,1){30}}
 \put(72,57){\line(0,1){30}}
\put(70,56){\line(0,1){30}}
 \put(68,56){\line(0,1){30}}
\put(66,55.5){\line(0,1){30}}
 \put(64,55){\line(0,1){30}}
\put(62,55){\line(0,1){30}}
%razriv 10 shtrijovki dlia kontura y bukvi \gamma_1
\put(60,55){\line(0,1){30}}
% \put(60,80){\line(0,1){5}}
\put(58,54.4){\line(0,1){30}}
 \put(56,54.5){\line(0,1){30}}
 \put(54,54){\line(0,1){30}}
 \put(52,53.8){\line(0,1){30}}
 \put(50,53.8){\line(0,1){30}}
 \put(48,53.8){\line(0,1){30}}
 \put(46,53.6){\line(0,1){30}}
  \put(44,53.6){\line(0,1){30}}
   \put(42,53.5){\line(0,1){30}}
    \put(40,53.5){\line(0,1){30}}
     \put(38,53.4){\line(0,1){30}}
 \put(36,53.4){\line(0,1){30}}
  \put(34,53.3){\line(0,1){30}}
   \put(32,53.3){\line(0,1){30}}
    \put(30,53.2){\line(0,1){30}}
     \put(28,53.2){\line(0,1){30}}
 \put(26,53.1){\line(0,1){30}}
  \put(24,53.1){\line(0,1){30}}
   \put(22,53){\line(0,1){30}}
    \put(20,53){\line(0,1){30}}
     \put(18,52.9){\line(0,1){30}}
 \put(16,52.9){\line(0,1){30}}
  \put(14,52.8){\line(0,1){30}}
   \put(12,52.8){\line(0,1){30}}
    \put(10,52.7){\line(0,1){30}}
%vtoraya liniya

%Razriv okolo -3\pi74
\qbezier(80,65)(85,72)(85,72)
\qbezier(87.5,73)(88,73.5)(90,74)
 \qbezier(90,74)(100,75.5)(120,76.5)
\qbezier(120,76.5)(140,77)(150,77.3)

%%VTORAYA LINIYA, LEVAYA POLOVINA

\qbezier(80,65)(75,58)(70,56) \qbezier(70,56)(60,54.5)(40,53.5)
\qbezier(40,53.5)(20,53)(10,52.7)
%3 liniya
 \qbezier(80,35)(85,42)(90,44)
\qbezier(90,44)(100,45.5)(120,46.5)
\qbezier(120,46.5)(140,47)(150,47.3)
\qbezier(80,35)(75,28)(70,26) \qbezier(70,26)(60,24.5)(40,23.5)
\qbezier(40,23.5)(20,23)(10,22.7)
%shtrijjovka vtoroy polosi, pravaya chast'
%Razriv pervoy shtrijovki.

% \put(82,8){\line(0,1){30}}
 \put(82,8){\line(0,1){4.5}}
 \put(82,15){\line(0,1){2.3}}
 \put(82,24){\line(0,1){3}}
 \put(82,30){\line(0,1){7}}
%Razriv vtoroy shtrijovki

 \put(84,11){\line(0,1){1.5}}
\put(84,14.5){\line(0,1){3}}
 \put(84,24){\line(0,1){3}}
 \put(84,32){\line(0,1){7}}

%Razriv tret'ey shtrijovki

%\put(86,17){\line(0,1){1}}
%\put(86,24){\line(0,1){1}}

\put(86,33){\line(0,1){8}}

%%%%%%%%%%%%%%%%%%%%%%%%%%%%%%%%
%Razriv chervertoy shtrijovki

 \put(88,17){\line(0,1){2}}
 \put(88,24){\line(0,1){3}}
 \put(88,32){\line(0,1){11}}

% \put(88,20){\line(0,1){10}}
% \put(88,32){\line(0,1){5}}

\put(90,14){\line(0,1){30}}
 \put(92,14){\line(0,1){30}}
\put(94,14.5){\line(0,1){30}}
 \put(96,15){\line(0,1){30}}
\put(98,15){\line(0,1){30}}
 \put(100,15){\line(0,1){30}}
\put(102,15.6){\line(0,1){30}}
 \put(104,15.5){\line(0,1){30}}
 \put(106,16){\line(0,1){30}}
 \put(108,16.2){\line(0,1){30}}
 \put(110,16.2){\line(0,1){30}}
 \put(110,16.3){\line(0,1){30}}
 \put(112,16.4){\line(0,1){30}}
  \put(114,16.4){\line(0,1){30}}
   \put(116,16.5){\line(0,1){30}}
    \put(118,16.5){\line(0,1){30}}
     \put(120,16.6){\line(0,1){30}}

 \put(122,16.6){\line(0,1){30}}
  \put(124,16.7){\line(0,1){30}}
   \put(126,16.7){\line(0,1){30}}
    \put(128,16.8){\line(0,1){30}}
     \put(130,16.8){\line(0,1){30}}
 \put(132,16.9){\line(0,1){30}}
  \put(134,16.9){\line(0,1){30}}
   \put(136,17){\line(0,1){30}}
    \put(138,17){\line(0,1){30}}
     \put(140,17.1){\line(0,1){30}}
 \put(142,17.1){\line(0,1){30}}
  \put(144,17.2){\line(0,1){30}}
   \put(146,17.2){\line(0,1){30}}
    \put(148,17.3){\line(0,1){30}}
     \put(150,17.3){\line(0,1){30}}

%shtrijovka vtoroy polosi v levoy chasti
\put(78,2){\line(0,1){30}}
 \put(76,0.5){\line(0,1){30}}
\put(74,-2.1){\line(0,1){30}}
 \put(72,-3){\line(0,1){30}}
\put(70,-4){\line(0,1){30}}
 \put(68,-4){\line(0,1){30}}
\put(66,-4.5){\line(0,1){30}}
 \put(64,-5){\line(0,1){30}}
\put(62,-5){\line(0,1){30}}

\put(60,-5.2){\line(0,1){30}}

\put(58,-5.6){\line(0,1){30}}
 \put(56,-5.5){\line(0,1){30}}
 \put(54,-6){\line(0,1){30}}
 \put(52,-6.2){\line(0,1){30}}
 \put(50,-6.2){\line(0,1){30}}
 \put(48,-6.2){\line(0,1){30}}
 \put(46,-6.4){\line(0,1){30}}
  \put(44,-6.4){\line(0,1){30}}
   \put(42,-6.5){\line(0,1){30}}
    \put(40,-6.5){\line(0,1){30}}
     \put(38,-6.6){\line(0,1){30}}
 \put(36,-6.6){\line(0,1){30}}
  \put(34,-6.7){\line(0,1){30}}
   \put(32,-6.7){\line(0,1){30}}
    \put(30,-6.8){\line(0,1){30}}
     \put(28,-6.8){\line(0,1){30}}
 \put(26,-6.9){\line(0,1){30}}
  \put(24,-6.9){\line(0,1){30}}
   \put(22,-7){\line(0,1){30}}
    \put(20,-7){\line(0,1){30}}
     \put(18,-7.1){\line(0,1){30}}
 \put(16,-7.1){\line(0,1){30}}
  \put(14,-7.2){\line(0,1){30}}
   \put(12,-7.2){\line(0,1){30}}
    \put(10,-7.3){\line(0,1){30}}
%4 liniya

\qbezier(80,65)(85,72)(85,72)
\qbezier(87.5,73)(88,73.5)(90,74)

 \qbezier(90,74)(100,75.5)(120,76.5)
\qbezier(120,76.5)(140,77)(150,77.3)

%razriv okolo -10\pi/4

\qbezier(80,5)(85,12)(85,12)
% \qbezier(89.5,13.7)(90,14)(90,14)

%konec razriva

 \qbezier(90,14)(100,15.5)(120,16.5)
\qbezier(120,16.5)(140,17)(150,17.3)

%levaya polovina

\qbezier(80,5)(75,-2)(70,-4) \qbezier(70,-4)(60,-5.5)(40,-6.5)
\qbezier(40,-6.5)(20,-7)(10,-7.3)

%oboznacheniya konturov {\cal C}, {\cal C}_1+i\pi/4, -{\cal C}_1-13\pi/4
%\put(90,50){${\cal C}_+$}
% \put(120,50){${\cal C}_1+i\pi/4$}
% \put(30,40){$-{\cal C}_{1}-13\pi/4$}

%vektori, ukazivayuschie konturi
%\put(116,50){\vector(-2,-1){15}}
% \put(116,50){\vector(1,-2){11.5}}

% \put(116,50){\vector(1,3){12.5}}
%\put(90,50){\vector(-2,-1){30}}
% \put(90,50){\vector(2,-1){10}}
 %\put(28,40){\vector(1,3){11}}

%\put(28,40){\vector(-1,-2){14}}
 %\put(28,40){\vector(2,-1){32.5}}
%napravleniya kontura

%{\linethickness{0.3mm}\put(60,60){\vector(0,1){10}}}
%\put(120,87.5){\vector(-1,0){10}} \put(60,72.5){\vector(-1,0){10}}
%\put(40,12.5){\vector(1,0){10}}
 %\put(100,27.5){\vector(1,0){10}}
%\put(100,40){\vector(0,-1){10}}

% contur \cal C_0
% kontur \gamma_1 s razrivom dlia -\pi/2
%\linethickness{0.3mm} \put (10,80){\line(1,0){140}}

%%%%%%%%%%%%%%%%%%%%%%%%%%%%%%%%%%%%%%%%%%%%%%%%%%%%%%%%%%%%%%%%%%%%%%%%%%
%%%%%%%%%%%%%%%  LÍNEAS HORIZONTALES SUPERIORES
%%%%%%%%%%%%%%%%%%%%%%%%%%%%%%%%%%%%%%%%%%%%%%%%%%%%%%%%%%%%%%%%%%%%%%%%%%

\linethickness{0.3mm} \put (10,72.5){\line(1,0){58}}
\linethickness{0.3mm} \put (92,87.5){\line(1,0){57.8}}

%\linethickness{0.3mm} \put (100,27.5){\line(1,0){60}}
%\linethickness{0.3mm}\put (60,72.5){\line(-1,0){60}}

%%%%%%%%%%%%%%%%%%%%%%%%%%%%%
%Estoy BORRANDO LAS LINEAS MÁS GRUESAS
%\linethickness{1mm}\put(10,16){\line(1,0){62}}
%\qbezier(72,16)(83,22)(89,25)
%\linethickness{1mm}\put(89,25){\line(1,0){61}}
%%%%%%%%%%%%%%%%%%%%%%%%%%%%

%%%%%%%%%%%%%%%%%%%%%%%%%%%%%%%%%%%%%%%%%%%%%%%%%%%%%%%%%%%%%%%%%%%%%%%%%%
%%%%%%%%%%%%%%%  LÍNEAS HORIZONTALES INFERIORES
%%%%%%%%%%%%%%%%%%%%%%%%%%%%%%%%%%%%%%%%%%%%%%%%%%%%%%%%%%%%%%%%%%%%%%%%%%

\linethickness{0.3mm}\put(10,12.5){\line(1,0){58}}
\linethickness{0.3mm}\put(92,27.5){\line(1,0){57.8}}

%%%%%%%%%%%%%%%%%%%%%%%%%%%%%%%%%%%%%%%%%%%%%%%%%%%%%%%%%%%%%%%%%%%%
%%%%%% ORIENTACIÓN DE $ \mathcal{C}_1 + i\frac{  \pi}{4}$  %%%%%%%%%
%%%%%%%%%%%%%%%%%%%%%%%%%%%%%%%%%%%%%%%%%%%%%%%%%%%%%%%%%%%%%%%%%%%%
{\linethickness{0.3mm}\put(140,87.5){\vector(-1,0){20}}}
{\linethickness{0.3mm}\put(92,77.5){\vector(0,-1){20}}}
{\linethickness{0.3mm}\put(103,27.5){\vector(1,0){20}}}

%%%%%%%%%%%%%%%%%%%%%%%%%%%%%%%%%%%%%%%%%%%%%%%%%%%%%%%%%%%%%%%%%%%%
%%%%%% ORIENTACIÓN DE $ -\mathcal{C}_1 - i\frac{13\pi}{4}$  %%%%%%%%%
%%%%%%%%%%%%%%%%%%%%%%%%%%%%%%%%%%%%%%%%%%%%%%%%%%%%%%%%%%%%%%%%%%%%
{\linethickness{0.3mm}\put(63,72.5){\vector(-1,0){20}}}
{\linethickness{0.3mm}\put(68,22.5){\vector(0,1){20}}}
{\linethickness{0.3mm}\put(23,12.5){\vector(1,0){20}}}

%%%%%%%%%%%%%%%%%%%%%%%%%%%%%%%%%%%%%%%%%%%%%%%%%%%%%%%%%%%%%%%%%%%%%%%%%%
%%%%%%%%%%%%%%%  LÍNEAS VERTICALES
%%%%%%%%%%%%%%%%%%%%%%%%%%%%%%%%%%%%%%%%%%%%%%%%%%%%%%%%%%%%%%%%%%%%%%%%%%

%\linethickness{0.3mm} \put(92,27.5){\line(0,1){60}}
%\linethickness{0.3mm} \put(68,12.5){\line(0,1){60}}

\linethickness{0.3mm} \put(92,28){\line(0,1){58}}
\linethickness{0.3mm} \put(68,13.5){\line(0,1){58}}

%%%%%%%%%%%%%%%%%%%%%%%%%%%%%%%%%%%%%%%%%%%%%%%%%%%%%%%%%%%%%%%%%%%%%%%%%%
%%%%%%%%%%%%%%%  LÍNEAS DIAGONALES
%%%%%%%%%%%%%%%%%%%%%%%%%%%%%%%%%%%%%%%%%%%%%%%%%%%%%%%%%%%%%%%%%%%%%%%%%%

\linethickness{0.4mm} \put(68,72.5){\line(5,3){25}}
\linethickness{0.4mm} \put(68,12.5){\line(5,3){25}}

\linethickness{0.6mm} \put(68,71.5){\line(5,3){24}}
\linethickness{0.6mm} \put(68,13.5){\line(5,3){24}}

%%%%%%%%%%%%%%%%%%%%%%%%%%%%%%%%%%%%%%%%%%%%%%%%%%%%%%%%%%%%%%%%%%%%%%%%%%
%%%%%%%%%%%%%%%  ORIENTACIÓN DE LAS LÍNEAS DIAGONALES
%%%%%%%%%%%%%%%%%%%%%%%%%%%%%%%%%%%%%%%%%%%%%%%%%%%%%%%%%%%%%%%%%%%%%%%%%%

\linethickness{0.3mm} \put(72,15){\vector(2,1){3}}
\linethickness{0.3mm} \put(87,83.7){\vector(-2,-1){3}}

\linethickness{0.3mm} \put(72,73.8){\vector(2,1){3}}
\linethickness{0.3mm} \put(87,24.7){\vector(-2,-1){3}}

%(8,5)

%%%%%%%%%%%%%%%%%%%%%%%%%%%%%
%Estoy BORRANDO LAS LINEAS MÁS GRUESAS
%\linethickness{1mm}\put(89,85){\vector(1,0){61}}
%\qbezier(72,76)(83,82)(89,85)
%\linethickness{1mm}\put(10,76){\vector(1,0){62}}
%%%%%%%%%%%%%%%%%%%%%%%%%%%%%

%Oboznacheniya \gamma_1 y \gamma_2

\put(127.4,90.5){$\gamma_1^+$}
\put(124.4,30.5){$\gamma_1^+-2i\pi$}

\put(61,40.5){$\Gamma^+$}

%linii vertikal'nie
%\thicklines \linethickness{0.3mm}\put(100,87.5){\line(0,-1){60}}

%\linethickness{0.3mm}\put (60,72.5){\line(0,-1){60}}
% Os' vertikal'naya
%\thinlines \put(80,-12){\vector(0,1){150}}
\thinlines \put(80,-12){\vector(0,1){130}}

%oboznachenie vert. osi
% \put(75,120){$\mathrm{Re}\ \omega > 0$}

 \end{picture}
}

\vspace{0.5cm}

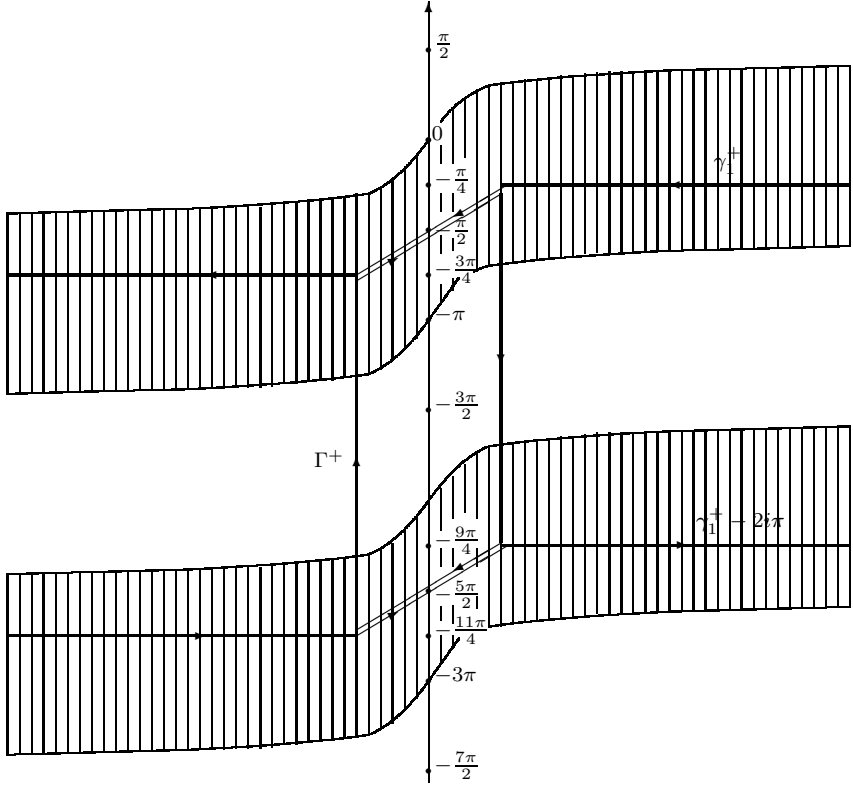
\centerline{Figure 6. The  contours $\mathcal{C}_{0}^+$ and $\Gamma^+$}

\vspace{0.5cm}

\section{The rate of convergence to the limiting amplitude}

In this section we analyze in more detail the difference between the amplitude of the nonstationary diffracted wave and its limiting amplitude in the case when the profile function $f$ is the Heaviside function (see (\ref{fH})). Let
\begin{eqnarray}\label{Rd}
R_d(\rho,\theta,t):= A_d(\rho,\theta) - A_d(\rho,\theta,t), \qquad \rho\geq 0, \ \theta\neq\theta_{1,2}, \ t\geq0,
\end{eqnarray}
where $A_d(\rho,\theta)$, $A_d(\rho,\theta,t)$ are defined in (\ref{Ainrd}) and (\ref{repAd}), respectively.
Since $f$ is the Heaviside function, %given by (\ref{f}) with $s_0=0$,
\begin{eqnarray}\label{1chR}
R_d(\rho,\theta,t)
%    &=& \dfrac{i}{4\Phi} \int\limits_{\mathbb{R}} e^{i\omega_0 \rho \cosh\beta} Z_N(\beta+i\theta)\ \mathrm{d}\beta \\ &&-
%        \dfrac{i}{4\Phi} \int\limits_{\mathbb{R}} e^{i\omega_0 \rho \cosh\beta} Z_N(\beta+i\theta) f(t-\rho\cosh\beta)
%        \ \mathrm{d}\beta\\
%    &=& \dfrac{i}{4\Phi} \int\limits_{\mathbb{R}} e^{i\omega_0 \rho \cosh\beta} Z_N(\beta+i\theta)
%       \Big[1- f(t-\rho\cosh\beta)\Big]  \ \mathrm{d}\beta
    &=& \dfrac{i}{4\Phi} \int\limits_{|\beta|\geq \mathrm{ac}(\frac{t}{\rho})} e^{i\omega_0 \rho \cosh\beta} Z_N(\beta+i\theta)
       %\Big[1- f(t-\rho\cosh\beta)\Big]
        \ \mathrm{d}\beta
\end{eqnarray}
where $Z_N$ is defined by (\ref{ZN}) and $\mathrm{ac}(\cdot)$ is defined in (\ref{defac}). First we need an expansion of $Z_N(\beta+i\theta)$.

\begin{lem} The function $Z_N$ admits the following representation
\begin{eqnarray}\label{ZNexpan}
Z_N(\beta+i\theta)
&=& ib_1 \left[
      \displaystyle\sum\limits_{k=1}^6 z_k^\pm e^{\mp 2kq\beta} + r_1^\pm(\beta) e^{\mp 14 q\beta}, \qquad \beta\in\mathbb{R}.
      \right]
\end{eqnarray}
where
\begin{equation}\label{bz1z2r}
\begin{array}{rcl}
b_1 &:=& 4\sin\left(\frac{\pi^2}{\Phi}\right)%\in\mathbb{R}
; \qquad \\
z_1^\pm &=& z_1^\pm (\theta,\alpha)
    := e^{\pm 2iq(\pi - \theta + \alpha)} \big[1 + e^{\pm 4iq(\pi - \alpha)}\big], \\
z_2^\pm &=& z_2^\pm (\theta,\alpha)
    := e^{\pm 2iq(\pi - 2\theta + 2\alpha)} \big[1 + e^{\pm 8iq(\pi - \alpha)}\big]
                                              \big[ 1 + e^{\pm 4iq\pi}\big],  \\
z_m &:=& z_m^+ + z_m^- \in \mathbb{R}, \ m=\overline{1,6};
\qquad |z_m^\pm(\theta,\alpha)| \leq C, \ \theta\in\mathbb{R},\ m=\overline{2,6}; \\
z_1 &=& z_1(\theta,\alpha)
    = 4 \cos\left[\dfrac{\pi}{\Phi}( 2\pi - \theta)\right] \cos\left[\dfrac{\pi}{\Phi}(  \pi - \alpha)\right];
\\ &&
|r_1^\pm\left(\beta,\theta,\alpha\right)| \leq C, \ \pm \beta \geq \frac{\ln 2}{q}, \ \theta\in\mathbb{R}.
\end{array}
\end{equation}
\end{lem}

\begin{proof} From (\ref{ZN}) and (\ref{HN}) it follows that
%\begin{eqnarray}\label{ZNsplit}
$Z_N(\beta) =  H_1(\beta) + H_2(\beta)$, %\qquad
$\beta\in\mathbb{C}$,
%\end{eqnarray}
where
%\begin{equation}\label{H1}
%\begin{array}{rcl}
%%$$H_k(\beta) =  \sinh\left(2iq\pi\right)
%%                      \left[ \sinh\left[ q \left(\beta + i (-1)^k \alpha - 2i(k-1)\pi \right) \right]
%%                        \sinh\left[ q \left(\beta + i (-1)^k \alpha - 2ik\pi \right) \right]\right]^{-1},$$
$$H_k(\beta) =  \dfrac{ \sinh\left(2iq\pi\right) }
                      { \sinh\left[ q \left(\beta + i (-1)^k \alpha - 2i(k-1)\pi \right) \right]
                        \sinh\left[ q \left(\beta + i (-1)^k \alpha - 2ik\pi \right) \right]},$$% \quad k=1,2  \\
%%H_1(\beta) &=&  \dfrac{ \sinh\left(2iq\pi\right) }
%%                      { \sinh\left[ q \left(\beta  - i\alpha \right) \right]
%%                        \sinh\left[ q \left(\beta  - i\alpha -2i\pi \right) \right]}, \\
%%H_2(\beta) &=&  \dfrac{ \sinh\left(2iq\pi\right) }
%%                      { \sinh\left[ q \left(\beta + i\alpha - 2i\pi \right) \right]
%%                        \sinh\left[ q \left(\beta + i\alpha - 4i\pi \right) \right]}.% \label{H2}
%\end{array}
%\end{equation}
for $k=1,2$. Let $\beta:=\beta_1+i\beta_2$ and
%\begin{eqnarray}\label{sb1}
$s=s(\beta_1):=e^{-q\beta_1}$.
%\end{eqnarray}
Then, for $k=1,2$, the functions $H_k$ admit the %following
expansions
%\begin{eqnarray}
$H_k(\beta)
= -ib_1
\left[\displaystyle\sum\limits_{j=1}^{6} h_{k,j}^{\pm }  s^{\pm 2j}
        +\ r_1(s^{\pm 1},e^{\pm ic_{k,1}(\beta_2)},e^{\pm ic_{k,2}(\beta_2)}) s^{\pm 14} \right]$, %\nonumber\\
%H_1(\beta)
%&=& 4\sinh\left(2iq\pi\right)
%\Big[e^{\mp 2iq( \beta_2 -  \alpha -  \pi)} s^{\pm 2}
%        + e^{\mp 2iq(2\beta_2 - 2\alpha -  \pi)} \big[ 1 + e^{\pm 4iq\pi}\big] s^{\pm 4} \label{repH1}\\
%        &&\hspace{2.7cm}
%        + \ e^{\mp 2iq(3\beta_2 - 3\alpha - \pi)} \left[ 1 + e^{\pm 4iq\pi} + e^{\pm 8iq\pi}\right]s^{\pm 6}  \nonumber \\&& \hspace{2.7cm}
%        +\ r_1(s^{\pm 1},e^{\pm ic_1(\beta_2)},e^{\pm ic_2(\beta_2)}) s^{\pm 8} \Big].\nonumber\\
%H_2(\beta)
%&=& 4\sinh\left(2iq\pi\right)
%  \Big[e^{\mp 2iq( \beta_2 +  \alpha - 3\pi)} s^{\pm2}
%        + e^{\mp 2iq(2\beta_2 + 2\alpha - 5\pi)} \big[ 1 +  e^{\pm 4iq\pi}\big] s^{\pm 4}\label{repH2} \\
%        &&\hspace{2.7cm}
%        + \ e^{\mp 2iq(3\beta_2 + 3\alpha - 7\pi)} \left[ 1 + e^{\pm 4iq\pi} + e^{\pm 8iq\pi}\right]s^{\pm 6}  \nonumber \\&& \hspace{2.7cm}
%        +  r_1(s^{\pm 1},e^{\pm i c_3(\beta_2)},e^{\pm i c_4(\beta_2)}) s^{\pm 8}\Big]. \nonumber
%\end{eqnarray}
where
$c_{1,1}\left( \beta_2 \right) := q(\beta_2 - \alpha)$,
$c_{1,2}\left( \beta_2 \right) := q(\beta_2 - \alpha - 2\pi)$,
$c_{2,1}\left( \beta_2 \right) := q(\beta_2 + \alpha - 2\pi)$,
$c_{2,2}\left( \beta_2 \right) := q(\beta_2 + \alpha -4\pi)$.
%%\begin{equation}\label{c12}
%%\begin{array}{rclcrcl}
%%c_{1,1}\left( \beta_2 \right) &:=& q(\beta_2 - \alpha),  &\quad&  c_{1,2}\left( \beta_2 \right) &:=& q(\beta_2 - \alpha - 2\pi), \\
%%c_{2,1}\left( \beta_2 \right) &:=& q(\beta_2 + \alpha - 2\pi), &\quad&
%%c_{2,2}\left( \beta_2 \right) &:=&q(\beta_2 + \alpha -4\pi).
%%%\label{c34}
%%\end{array}
%%\end{equation}
In both cases the function $s\mapsto r_1(s,a_1,a_2)$ is analytic in $B_1(0):=\{z\in\mathbb{C}: |z|<1\}$ for all $a_1,a_2\in\mathbb{T}:=\{z\in\mathbb{C}: |z|=1\}$ and  admits the estimate
\begin{eqnarray}\label{1br1}
|r_1(s,a_1,a_2)| &\leq& C_1, \qquad  |s|\leq 2^{-1}, \ \ a_1,a_2\in\mathbb{T}.
\end{eqnarray}
Moreover,
$h_{1,1}^\pm = e^{\mp 2iq( \beta_2 - \alpha -  \pi)}$,
$h_{2,1}^\pm = e^{\mp 2iq( \beta_2 + \alpha - 3\pi)}$ and
$h_{k,j}^+ + h_{k,j}^- \in \mathbb{R}$, for  $k=1,2$;  $ j=\overline{1,6}$.
%%\begin{eqnarray}
%%h_{1,1}^\pm = e^{\mp 2iq( \beta_2 - \alpha -  \pi)}, \quad
%%h_{2,1}^\pm = e^{\mp 2iq( \beta_2 + \alpha - 3\pi)}, \quad
%%h_{k,j}^+ + h_{k,j}^- \in \mathbb{R},\ k=1,2; \ j=\overline{1,6}.
%%\end{eqnarray}
These statements are proved similarly to Proposition 2.1 (iii) of \cite{ocoa}. Hence, taking $\beta_2=0$ we infer (\ref{ZNexpan}),  (\ref{bz1z2r}).
\end{proof}

\begin{lem} The function $R_d(\rho,\theta,t)$  admits the following representation
\begin{eqnarray}\label{REp}
R_d(\rho,\theta,t)
&=& -\dfrac{b_1}{4\Phi}
    \left[ \displaystyle\sum\limits_{m=1}^{6} z_m E_{2qm} + R_\infty(\rho,\theta,t)
%       z_1 E_{2q}  + z_2 E_{4q} + z_3 E_{6q}
%     +  \int\limits_{\mathrm{ac}  \left(\frac{t}{\rho}\right) }^{+\infty} r_1(\beta)
%                                   e^{i\omega_0 \rho \cosh\beta - 8q\beta}   \ \mathrm{d}\beta
        \right],
\end{eqnarray}
where $b_1$, $z_m$ are given by (\ref{bz1z2r}) for $m=\overline{1,6}$,
\begin{equation}\label{zmEp}
\begin{array}{rcl}
%r_1(\beta)&:=&r_1(\beta,\theta,\alpha) = r_1^+(\beta,\theta,\alpha) +  r_1^-(-\beta,\theta,\alpha), \\
E_p &:=& E_p(\rho,t)
      =\displaystyle \int\limits_{\mathrm{ac}\left(\frac{t}{\rho}\right)}^{+\infty} e^{i\omega_0 \rho \cosh\beta - p\beta}\ \mathrm{d}\beta
\end{array}
\end{equation}
and
%\begin{eqnarray}
%|z_m(\theta,\alpha)| &\leq C&, \qquad \theta\in\mathbb{R},\ m=2,3,   \label{bzm}\\
%|r_1(\beta,\theta,\alpha)| &\leq C&, \qquad \beta\geq\dfrac{\ln 2}{q}, \ \theta\in\mathbb{R}.\label{br1}
%\end{eqnarray}
\begin{eqnarray}\label{bR}
|R_\infty(t,\rho,\theta)| &\leq C(\rho_0) t^{-14 q}&, \qquad %\beta\geq\dfrac{\ln 2}{q}
t\geq \rho_0 2^{\frac{1}{q}}, \ \theta\in\mathbb{R}.
\end{eqnarray}
\end{lem}

\begin{proof} It follows from (\ref{1chR}), (\ref{ZNexpan}) and the evenness of the integrand in (\ref{zmEp}).
The estimate (\ref{bR})  follows from %(\ref{zmEp}) and the
estimates (\ref{bz1z2r}), (\ref{1br1}).
\end{proof}
%
%\vspace{0.5cm}

The main result of this section is  the description of  the rate of the convergence to the limiting amplitude. Following  \cite{hew} we find this rate for the real and imaginary part separately.

\begin{teo}\label{teoasyR}
Let $R_d$ be defined by (\ref{Rd}). The real and imaginary parts of the function $e^{-i\omega_0 t} R_d$ %(\rho,\theta,t)$,
admit the following asymptotic behavior when $t\rightarrow +\infty$
\begin{eqnarray}
\mathrm{Re}\big[ e^{-i\omega_0 t} R_d(\rho,\theta,t)\big] \label{asyReR}
&=& -  \dfrac{b_1 z_1 (2q+1)}{4\Phi\omega_0^2} \left(\dfrac{\rho}{2}\right)^{2q} \ t^{-(2q+2)}
        +  O\left(t^{-(4q+2)}\right), \\
\mathrm{Im}\big[ e^{-i\omega_0 t} R_d(\rho,\theta,t) \big]\label{asyImR}
&=& -  \dfrac{b_1 z_1}{4\Phi \omega_0}   \left(\dfrac{\rho}{2}\right)^{2q}    \ t^{-(2q+1)}
    +  O\left(t^{-(4q+1)}\right),
\end{eqnarray}
where $0<\rho\leq \rho_0$ %, $t\geq \rho_0 2^{1/q}+s_0$
and the symbol $O(\cdot)$ depends only on $\rho_0$.
% $b$, $z_1$ are given by (\ref{bz1z2r}), (\ref{zmrep}), respectively.
\end{teo}

%
%%%%%%%%%%%%%%%%%%%%%%%%%%%%%%%%%%%%%%%%%%%%%%%%%%%%%%
%
%% The following result is proved in \cite{ocoa}, see Lemma 3.2.
%
%%\begin{lem}
%%Let $\rho>0$. For $t>\rho_0\geq 0$, $E_p(\rho,\theta)$ has the following asymptotic behavior
%%\begin{eqnarray}
%%E_p(\rho,t) = \dfrac{\rho^p}{t^{p+1}}\ \overline{E}_p(\rho,t) + O(t^{-(p+2)}),
%%\end{eqnarray}
%%where
%%\begin{eqnarray}\label{EpEpbar}
%%\overline{E}_p(\rho,t) = -\dfrac{2}{i\omega_0}\ e^{i\omega_0 t} \left[\dfrac{t}{t+\sqrt{t^2-\rho^2}}\right]^{p+1}.
%%\end{eqnarray}
%%\end{lem}
%%%%%%%%%%%%%%%%%%%%%%%%%%%%%%%%%%%%%%%%%%%%%%%%%%%%%%%%%%%

\begin{proof}
After the change of variable
%\begin{eqnarray}
$u =  \dfrac{\rho}{t}\ e^\beta$,
%\end{eqnarray}
the function $E_p$, given by (\ref{zmEp}), takes the form
%\begin{eqnarray}\label{epbp}
$E_p(\rho,t) = \left(\dfrac{\rho}{t}\right)^p B_p(\rho,t)$,
%\end{eqnarray}
where
%\begin{eqnarray}\label{bp}
$B_p(\rho,t)
= \displaystyle\int\limits_{1+\sqrt{1-(\frac{\rho}{t})^2}}^{+\infty}
 e^{i\omega_0 \frac{u^2 t^2 + \rho^2}{2ut}}\ u^{-p-1}\ \mathrm{d}u.$
%\end{eqnarray}
Similarly to the proof of Lemma 3.2 in \cite{ocoa} we obtain the following expansions
%Let $p>0$. For $t> \rho_0\geq \rho\geq 0$, the following asymptotics  holds:
\begin{eqnarray}\label{b1}
E_p(\rho,t) &=& \frac{\rho^p}{t^{p+1}}\overline{E}_p(\rho,t)
    +   \frac{2(p+1)}{i\omega_0} \cdot \frac{\rho^p}{t^{p+2}} \overline{E}_{p+1}(\rho,t)
    +   O\left(\frac{1}{t^{p+3}}\right)
\end{eqnarray}
where
%\begin{eqnarray}\label{EpEpbar}
$\overline{E}_p(\rho,t) = -\dfrac{2}{i\omega_0}\ e^{i\omega_0 t} \left[\dfrac{t}{t+\sqrt{t^2-\rho^2}}\right]^{p+1}$.
%\end{eqnarray}
Expanding $\overline{E}_p$ as in (\ref{asyB}) (Appendix A4) and substituting it into (\ref{b1}) we get
$E_p(\rho,t)
= - e^{i\omega_0 t} \dfrac{\rho^{p}}{2^{p}(i\omega_0)}        \cdot \dfrac{1}{t^{p+1}}
    - e^{i\omega_0 t} \dfrac{\rho^{p}(p+1)}{2^{p}(i\omega_0)^2} \cdot \dfrac{1}{t^{p+2}}
    +  O\left(\frac{1}{t^{p+3}}\right)$.
%%\begin{eqnarray}\label{asyEpWe}
%%E_p(\rho,t)
%%&=& - e^{i\omega_0 t} \dfrac{\rho^{p}}{2^{p}(i\omega_0)}        \cdot \dfrac{1}{t^{p+1}}
%%    - e^{i\omega_0 t} \dfrac{\rho^{p}(p+1)}{2^{p}(i\omega_0)^2} \cdot \dfrac{1}{t^{p+2}}
%%    +  O\left(\frac{1}{t^{p+3}}\right).
%%\end{eqnarray}
Finally, substituting these expressions into (\ref{REp}) and using (\ref{bR}) we obtain
\begin{equation}\label{lagata}
\begin{array}{rcl}%\label{REp}
e^{-i\omega_0 t} R_d(\rho,\theta,t)
&=& -\dfrac{b_1 z_1}{4\Phi}
    \left[   \left(\dfrac{\rho}{2}\right)^{2q} \dfrac{2q+1}{ \omega_0^2} \ t^{-(2q+2)}
           - \left(\dfrac{\rho}{2}\right)^{2q} \dfrac{  1 }{i\omega_0  } \ t^{-(2q+1)}   \right]\\ &&
    -\dfrac{b_1 z_2}{4\Phi}
    \left[   \left(\dfrac{\rho}{2}\right)^{4q} \dfrac{4q+1}{ \omega_0^2} \ t^{-(4q+2)}
           - \left(\dfrac{\rho}{2}\right)^{4q} \dfrac{ 1  }{i\omega_0  } \ t^{-(4q+1)}   \right]\\ &&
    +  \sum\limits_{j=3}^6  m_j t^{-(2qj + 2)}
    +i \sum\limits_{j=3}^6  n_j t^{-(2qj + 1)} \\&&
    +  \sum\limits_{j=1}^6   O( t^{-(2qj + 3)})
    +  r_\infty(\rho,\theta,t) t^{-14q},\qquad t\rightarrow\infty,
\end{array}
\end{equation}
where $m_j,n_j\in\mathbb{R}$ by (\ref{bz1z2r}) and
%\begin{eqnarray*}
 $ |r_\infty(\rho,\theta,t)| \leq C(\rho_0)$. %$ t^{-14q}$.
%\end{eqnarray*}
Noting that $q>\frac{1}{4}$ by (\ref{defFi}), we infer from this  (\ref{asyReR}) and (\ref{asyImR}).
\end{proof}

\section{The case of  half plane}

In this section we consider the case of $\Phi=2\pi$ (see Remark \ref{obsalpha}) and we compare our results  with the results of \cite{hew}.

%%%%%%%%%%%%%%%%%%%%%%%%%%%%%%%%%%%%%%%%%%%%%%%%%%%%%%%%%%%%%%%%%%%%%%%%%%%%%%%%%%%%%%%%%%%%%%%
%\subsection{Comparison of solution to the NN-scattering for the arbitrary angle of the incidence}
%%%%%%%%%%%%%%%%%%%%%%%%%%%%%%%%%%%%%%%%%%%%%%%%%%%%%%%%%%%%%%%%%%%%%%%%%%%%%%%%%%%%%%%%%%%%%%%
%In this section, we prove that in the case $\Phi=2\pi$ our formula (\ref{rud}) for the diffracted wave is equivalent to the corresponding formula (43) from \cite{hew} for the arbitrary angle of the incidence.

\begin{prop} Let $\Phi=2\pi$ and $f$ be the Heaviside function.  Then representation (\ref{rud}) for the diffracted wave can be rewritten as
        \begin{eqnarray}\label{ud2T}
        u_{d}(\rho,\theta,t)
        = \frac{ie^{-i\omega_0 t}}{2\pi} \int\limits_{-\mathrm{ac}(\frac{t}{\rho})}^{\mathrm{ac}(\frac{t}{\rho})}
        e^{i\omega_0\rho\cosh\beta}
        A(\beta)\ \mathrm{d}\beta, \qquad t\geq 0,
        \end{eqnarray}
        where
        \begin{eqnarray}\label{defA}
        A(\beta)
        :=  \frac{-\sinh\frac{i\alpha}{2}}{2}
        \left[
                \frac{  \cosh\frac{\beta+i\theta}{2} }
                     {  \sinh\frac{\beta+i\theta+i\alpha}{2} \sinh\frac{\beta+i\theta-i\alpha}{2} }
                +\frac{ \cosh\frac{\beta-i\theta}{2}}
                      { \sinh\frac{\beta-i\theta-i\alpha}{2} \sinh\frac{\beta-i\theta+i\alpha}{2} }
        \right].
        \end{eqnarray}
Moreover,
        \begin{eqnarray}\label{rudpart2}
        u_{d}(\rho,\theta,t)
        %&=& \dfrac{e^{-i\omega_0 t}}{2\pi}\cdot\dfrac{1}{2} \int\limits_{-ac(\frac{\rho}{t})}^{ac(\frac{\rho}{t})}
        %         e^{i\omega_0 \rho\cosh\beta}
        %         \left[
        %         \dfrac{4\cosh\frac{\beta}{2}  \cosh\frac{i\theta}{2} }{ \cosh\beta + \cosh( i\theta) }
        %         \right]\mathrm{d}\beta\\
        &=& \dfrac{e^{-i\omega_0 t}}{\pi} \int\limits_{-\mathrm{ac}(\frac{\rho}{t})}^{\mathrm{ac}(\frac{\rho}{t})}
                 e^{i\omega_0 \rho\cosh\beta}
                 \left[
                 \dfrac{ \cosh\frac{\beta}{2}  \cosh\frac{i\theta}{2} }{ \cosh\beta + \cosh( i\theta) }
                 \right]\mathrm{d}\beta.
        \end{eqnarray}
for $\alpha=\pi$.
%\end{enumerate}
\end{prop}

\begin{proof} Representation (\ref{ud2T}) follows from (\ref{ZN}), (\ref{HN}) and (\ref{rud}) when $f=\mathcal{H}$ and $\Phi=2\pi$. Representation (\ref{rudpart2}) follows from (\ref{ud2T}) when $\alpha=\pi$.
\end{proof}

\begin{obs} In the case of  half plane ($\Phi=2\pi$) and the Heaviside function $f$ representation (\ref{ud2T})
%and (\ref{rudpart2})
for the diffracted wave coincides with the representation of the diffracted wave $\Phi_d$ given by (43) and modified according to Section 3.3 in  \cite{hew} for $\theta_0\neq 0$. For $\theta_0=0$
\begin{equation}\label{dif}
u_d(\rho,\theta,t) = 2\Phi_d(\rho,\theta,t)
\end{equation}
\end{obs}

\begin{proof}
%Remind that functions $Z_N$ and $u_d$ depend on $\Phi$, by (\ref{ZN}) and (\ref{ud2}). From (\ref{ZN}), (\ref{HN}) and (\ref{rud}), for $f=\mathcal{H}$ and $\Phi=2\pi$ we obtain
% %and the identities $\coth \left(A - i\frac{\pi}{2}\right) = \tanh A$, $\coth \left(A - i\pi \right) = \coth A$ and $\coth A - \tanh A = \frac{2}{\sinh(2A)}$ it follows that
%we have
%\begin{eqnarray}\label{limZ}
%Z_N(\beta+i\theta,2\pi)
%&=& 4 \left[ \frac{\sinh(-i\frac{\alpha}{2}) \cosh\frac{\beta+i\theta}{2}}
%               {\sinh \frac{\beta+i\theta+i\alpha}{2}  \sinh \frac{\beta+i\theta-i\alpha}{2} }\right].
%\end{eqnarray}
%Hence for $f=\mathcal{H}$ we obtain from  (\ref{rud})
%\begin{eqnarray}\label{udT}
%u_{d}(\rho,\theta,t,2\pi)
%= \frac{-ie^{-i\omega_0 t}}{2\pi} \int\limits_{-ac(\frac{t}{\rho})}^{ac(\frac{t}{\rho})}
%    e^{i\omega_0\rho\cosh\beta}
%    \left[\frac{\sinh\frac{i\alpha}{2}\cosh\frac{\beta+i\theta}{2}}
%               {\sinh \frac{\beta+i\theta+i\alpha }{2}  \sinh \frac{\beta+i\theta-i\alpha}{2} }\right]
%    \mathrm{d}\beta, \ \ t\geq 0.
%\end{eqnarray}
%The function $Z_N$ is not even by (\ref{limZ}), so we represent $u_d$ in the form where the integrand is even. Therefore we have
Let us consider
%Here% $\theta_0$ %is the angle of the impinging wave %. We reduce this formula to (\ref{ud2T}), when
\begin{eqnarray}\label{t0}
\theta_0 =\alpha-\pi
\end{eqnarray}
the angle which corresponds to the orientation of the impinging wave in our problem. When $\theta_0\neq 0$, the diffracted wave obtained in \cite{hew}   by means of formula (43) and modified according to Section 3.3 (we denote it by $\Phi_d$) is expressed as  follows for $c_0=1$ (we omit the part which corresponds to the incident wave (see page 210 in \cite{hew}))
%%It is easy to check that our formula (\ref{ud2T}) holds for any $\alpha$. According to (43) from \cite{hew}, page 210, we have
%in the construction of \cite{hew}, page 210 and >(43)?
\begin{eqnarray}\label{obj}
\Phi_d(\rho,\theta,t)
&=& \dfrac{e^{-i\omega_0 t}}{2\pi}\left[ -\mathrm{sgn}(\pi-(\theta-\theta_0))\sqrt{\rho(1+\cos(\theta-\theta_0))}
        \int\limits_{\rho}^t  \frac{e^{i\omega_0s}}{\sqrt{s-\rho}(s+\rho\cos(\theta-\theta_0))}\ \mathrm{d}s\right.\nonumber\\
& &\hspace{1.4cm}
\left.
\displaystyle
-\mathrm{sgn}(\pi-(\theta+\theta_0)) \sqrt{\rho(1+\cos(\theta+\theta_0))}
\int\limits_{\rho}^t  \dfrac{e^{i\omega_0s}}{\sqrt{s-\rho}(s+\rho\cos(\theta+\theta_0))}\ \mathrm{d}s
\right ].\nonumber\\ &&
\end{eqnarray}
Making the change of variable
$
s:=\rho\cosh\beta,
$
and using the evenness of the integrands obtained after the change of variable we rewrite $\Phi_d$ as
\begin{eqnarray*} %\label{fid}
\Phi_d(\rho,\theta,t)
&=&
\dfrac{e^{-i\omega_0 t}}{2\pi}
\int\limits_{-\mathrm{ac}(\frac{t}{\rho})}^{\mathrm{ac}(\frac{t}{\rho})}
e^{i\omega_0\rho\cosh\beta}
\left[    -\mathrm{sgn}(\pi-(\theta-\theta_0))
           \frac{\cosh\frac{i(\theta-\theta_0)}{2} \cosh\frac{\beta}{2}}
                 {\cosh\beta +\cosh (i(\theta-\theta_0))}
\right.
\\ && \hspace{4cm}
\displaystyle
\left.   -\mathrm{sgn}(\pi-(\theta+\theta_0))
          \frac{\cosh\frac{i(\theta+\theta_0)}{2} \cosh\frac{\beta}{2}}{\cosh\beta +\cosh (i(\theta+\theta_0))}
\right] \mathrm{d}\beta.\nonumber
\end{eqnarray*}
Let us consider the case $\theta > \theta_0 + \pi$, $\frac{\pi}{2}<\theta_0 <\pi$. The other cases are analyzed similarly. From (\ref{t0}) we obtain
\begin{eqnarray}\label{fid1}
\Phi_d(\rho,\theta,t)
&=&
%i\dfrac{e^{-i\omega_0t}}{2\pi}\left[ \int\limits_{-ac(t/\rho)}^{ac(t/\rho)}
%e^{i\omega_0\rho\cosh\beta}\left (
% \frac{\sinh \frac{i(\theta-\alpha)}{2}}{\cosh\beta -\cosh (i(\theta-\alpha))}
%-\frac{\sinh \frac{i(\theta+\alpha)}{2}}{\cosh\beta -\cosh (i(\theta+\alpha))}
%\right)
%\cosh\frac{\beta}{2}\ \mathrm{d}\beta\right]
\dfrac{ie^{-i\omega_0 t}}{2\pi}
\int\limits_{-\mathrm{ac}(\frac{t}{\rho})}^{\mathrm{ac}(\frac{t}{\rho})}
e^{i\omega_0\rho\cosh\beta}B\left (\beta\right)\ \mathrm{d}\beta,
\end{eqnarray}
where
\begin{eqnarray}\label{defB}
B(\beta)
:=\left[    \dfrac{\sinh \frac{i(\theta - \alpha)}{2}}{\cosh\beta -\cosh (i(\theta - \alpha))}
            - \dfrac{\sinh \frac{i(\theta + \alpha)}{2}}{\cosh\beta -\cosh (i(\theta + \alpha))}
     \right] \cosh\frac{\beta}{2}.
\end{eqnarray}
The poles and residues of $A(\beta)$ and $B(\beta)$ coincide and both functions are periodic with period $2i\pi$. %It remains only to estimate the grow of $A(\beta)$ and $B(\beta)$ at infinity.
Moreover, from  (\ref{defA}) and (\ref{defB}),
$
|A(\beta)|\leq C \ e^{\frac{\mathrm{Re} |\beta|}{2}}
$
and
$
|B(\beta)|\leq C \ e^{\frac{|\mathrm{Re}\ \beta|}{2}}
$, respectively. Hence $A\equiv B$ by the Liouville Theorem. This proves that (\ref{ud2T}) and (\ref{fid1}) are identically equal functions. \newline
\noindent When $\theta_0=0$, formula (43) in \cite{hew} implies that
\begin{eqnarray}\label{2obj}
\Phi_d(\rho,\theta,t)
%%&=& \dfrac{e^{-i\omega_0 t}}{2\pi}\left[ -\mathrm{sgn}(\pi-\theta)\sqrt{\rho(1+\cos\theta)}
%%        \int\limits_{\rho}^t  \frac{e^{i\omega_0 s}}{\sqrt{s-\rho}(s+\rho\cos\theta)}\ \mathrm{d}s\right.\nonumber\\
%%& &\hspace{1.4cm}
%%\left.
%%\displaystyle
%%-\mathrm{sgn}(\pi-\theta) \sqrt{\rho(1+\cos\theta)}
%%\int\limits_{\rho}^t  \dfrac{e^{i\omega_0 s}}{\sqrt{s-\rho}(s+\rho\cos\theta)}\ \mathrm{d}s
%%\right ].\nonumber\\
&=& -\dfrac{e^{-i\omega_0 t}}{2\pi}\ \mathrm{sgn}(\pi-\theta)\sqrt{\rho(1+\cos\theta)}
        \int\limits_{\rho}^t  \frac{e^{i\omega_0 s}}{\sqrt{s-\rho}(s+\rho\cos\theta)}\ \mathrm{d}s.
\end{eqnarray}
Making the change of variable
$
s=\rho\cosh\beta,
$
and using the evenness of the integrand in the obtained integral, we infer (\ref{dif}).
%that (\ref{rudpart2}) and (\ref{2obj}) are the same expression.
\end{proof}

\begin{obs}
The difference between the cases $\theta_0>0$ and $\theta_0=0$ is explained in the following way. When $\theta_0\rightarrow 0+$,  expression (\ref{obj}) does not converge
to expression (\ref{2obj}), instead it turns into the doubled value of $\Phi_d$ from (\ref{2obj}).
  At the same time expression (\ref{ud2T}) for $u_d$  in our representation turns into (\ref{rudpart2}) when $\alpha\rightarrow \pi$. This relates to the difference in scattering problem formulation for $\theta_0=0$.

  Unlike approach in \cite{hew},
  We take into account the ``reflected" wave which equals to $u_{in}$. Thus our diffracted wave $u_d$ compensates for $2u_{in}$ on the line $\theta=\pi$, and the diffracted wave $\Phi_d$ compensates for only  $u_{in}$. \newline
This fact leads to the difference in the principal terms of the amplitude asymptotic behavior as $t\rightarrow \infty$. In fact,  in the case of $\Phi=2\pi$, $\alpha=\pi$ we have for any $0<\rho\leq \rho_0$, $t\geq \rho_0 2^4$ the following asymptotic  behaviours when $t\rightarrow \infty$
\begin{eqnarray*}
\mathrm{Re}\big[ e^{-i\omega_0 t}R_d(\rho,\theta,t) \big]  \label{asyReR2pi}
&=& -\dfrac{3\sqrt{\rho} }{ \sqrt{2} \ \pi \omega_0^2 }\ \cos\frac{\theta}{2}\   t^{-\frac{5}{2}} + O(t^{-\frac{7}{2}}),\\
\mathrm{Im }\big[e^{-i\omega_0 t}\ R_d(\rho,\theta,t)\big] \label{asyImR2pi}
&=&  - \dfrac{\sqrt{2\rho}}{\pi\omega_0} \cos\frac{\theta}{2}\ t^{-\frac{3}{2}}  +  O(t^{-\frac{5}{2}}).
\end{eqnarray*}
%\begin{proof} (\ref{asyReR2pi}) and (\ref{asyImR2pi})
This follows from (\ref{lagata}) and (\ref{bz1z2r}) % when $\Phi=2\pi$ and $\alpha=\pi$,
since in this case, $z_2=0$. On the other hand from (61) of \cite{hew} it follows that
%\end{proof}
$
\mathrm{Re}\ \Phi_d \sim -\mathrm{sgn}(\pi-\theta)
\dfrac{ \sqrt{\eta} (3\xi + \eta)} { 4 \pi \xi^{\frac{3}{2}} (\xi + \eta)^2} \left[1+O\left(\dfrac{1}{\xi^2}\right)\right]$,
 when %\qquad
 $\xi\rightarrow\infty$,
where $\xi=\omega_0(t-\rho)$ and $\eta=\omega_0\rho(1+\cos\theta)$. Hence
%, we have
%\begin{eqnarray*}
%\dfrac{ \sqrt{\eta} (3\xi + \eta)} { 4 \pi \xi^{\frac{3}{2}} (\xi + \eta)^2}
%%%&=& \dfrac{  \sqrt{\omega_0\rho(1+\cos\theta)}  \big[3\omega_0(t-\rho) + \omega_0\rho(1+\cos\theta)\big] }
%%%          {  4\pi \big[\omega_0(t-\rho)\big]^{\frac{3}{2}} \big[\omega_0(t-\rho) + \omega_0\rho(1+\cos\theta)\big]^2} \\
%%%&=& \dfrac{  \omega_0^{\frac{1}{2}} \rho^{\frac{1}{2}} (2\cos^2\frac{\theta}{2})^{\frac{1}{2}}
%%%             \omega_0 \big[3t - 2\rho + \rho\cos\theta\big] }
%%%          {  4\pi \omega_0^{\frac{3}{2}}(t-\rho)^{\frac{3}{2}} \omega_0^2 \big[t +\rho\cos\theta\big]^2} \\
%%%&=& \dfrac{  \omega_0^{\frac{3}{2}} \rho^{\frac{1}{2}} 2^\frac{1}{2}\  \sqrt{\cos^2\frac{\theta}{2}}\
%%%             \big[3t - 2\rho + \rho\cos\theta\big] }
%%%          {  2^2 \pi \omega_0^{\frac{7}{2}} (t-\rho)^{\frac{3}{2}}  \big[t +\rho\cos\theta\big]^2} \\
%&=& \dfrac{  \sqrt{\rho}  }{ 2\sqrt{2}\ \pi \omega_0^2 } \left|\cos\frac{\theta}{2}\right|\cdot
%           \dfrac{  3t - 2\rho + \rho\cos\theta  } {  (t-\rho)^{\frac{3}{2}}  \big[t +\rho\cos\theta\big]^2}
%\end{eqnarray*}
\begin{eqnarray*}\label{asyReFi}
\mathrm{Re}\ \Phi_d
&=& - \dfrac{ 3\sqrt{\rho} }{ 2\sqrt{2}\ \pi \omega_0^2 } \cos\frac{\theta}{2}\  t^{-\frac{5}{2}} + O(t^{-\frac{7}{2}}).
\end{eqnarray*}
%It follows from
Similarly, from (62) of \cite{hew} %that
%%$$
%%\mathrm{Im}\ \Phi_d \sim -\mathrm{sgn}(\pi-\theta) \dfrac{\sqrt{\eta}}{2\pi \sqrt{\xi}\ (\xi + \eta)}\left[1+O\left(\dfrac{1}{\xi^2}\right)\right], \qquad \xi\rightarrow\infty
%%$$
%%and similarly to (\ref{asyReFi})
we obtain
\begin{eqnarray*}\label{asyImFi}
\mathrm{Im}\ \Phi_d &=&
- \dfrac{\sqrt{2\rho} }{ 2 \pi \omega_0 }\ \cos\frac{\theta}{2} \  e^{i\omega_0 t} \  t^{-\frac{3}{2}} + O(t^{-\frac{5}{2}}).
\end{eqnarray*}
\end{obs}

\section{Conclusion}

We have completely solved the problem of plane periodic wave scattering by a NN-wedge. Namely, we obtained
 an explicit formula for the cylindrical wave diffracted by the edge of the wedge, we proved the Limiting Amplitude Principle and we found the rate of  convergence to the limiting amplitude. Moreover, this formula is convenient for
 studying solution behavior near the wavefront and for creating theory of nonperiodic wave scattering by wedges. We will explore the result in future publications.

\section{Appendix}

\begin{center}
\textbf{A1}
\end{center}

\begin{lem}\label{propHN}
The function $H_N(\beta)$, defined by (\ref{HN}),  has the following properties
\begin{enumerate}[i)]
			\item $H_N(-\beta+i\pi)=-H_N(\beta)$, for any $\beta\in\mathbb{C}$.
			\item $H_N(\beta+2i\Phi)=H_N(\beta)$, for any $\beta\in\mathbb{C}$.
\end{enumerate}
\end{lem}

\begin{proof} %\textbf{i)} and \textbf{ii)}
It follows directly from (\ref{HN}).
\end{proof}
%%%%%%%%%%%%%%%%%%%%%%%%%%%%%%%%%%%%%%%%%%%%%%%%%%%%%%%%%%%%%%%%%%%%%%%%%%%%

\begin{lem}\label{comprous}
The function $\widehat{u}_{s}$ given in (\ref{TFus}) satisfies ``stationary" NN-problem (\ref{SPNN}).
\end{lem}
%%%% DEMOSTRACIÓN COMPLETA EN LA VERSIÓN 28/NOV/2012
\begin{proof} It is easy to check that $f_1(\rho,\theta,\omega) := - \widehat{g}(\omega) e^{i\rho\omega\cos(\theta-\alpha)}$ satisfies system (\ref{SPNN}). Therefore by (\ref{TFus}), to prove (\ref{SPNN}) for $\widehat{u}_{s}$ it suffices to prove that
\begin{eqnarray}\label{f2}
  f_2(\rho,\theta,\omega) &:=& \int\limits_{\mathcal{C}} e^{-\rho \omega \sinh\beta} H_N(\beta+i\theta)\  \mathrm{d}\beta
\end{eqnarray}
satisfies for any $\rho>0$
\begin{equation}\label{SPNN2}
\left\{
	   \begin{array}{rcl}
       (\Delta + \omega^2)\ f_2(\rho,\theta,\omega) & = & 0, \ \qquad \theta\in[\phi,2\pi]   \\ %\\
       \dfrac{\partial\ }{\partial y_2} f_2(\rho,\theta,\omega) & = &  0, \ \qquad \theta = 2\pi \\ %\\
       \dfrac{\partial\ }{\partial \mathbf{n_2}}f_2(\rho,\theta,\omega) & = & 0,\ \qquad \theta = \phi
		\end{array}
\right.
\end{equation}
%%%\begin{center}
%%%\textbf{$f_2$ satisfies the first equality of (\ref{SPNN2})}
%%%\end{center}
The Helmholtz equation in  (\ref{SPNN2}) follows by differentiation of the integral (\ref{f2})
 after the change of variable  $\beta\mapsto\beta'-i\theta$, since $(\Delta + \omega^2) e^{-\rho\omega\sinh(\beta-i\theta)}=0$. Moreover, the integral in  (\ref{f2}) converges absolutely after the differentiation for any $\omega\in\mathbb{C}^+$  by (\ref{C}), the condition $\omega\in\mathbb{C}^+$ and (\ref{HN}).

%% Let us prove that also $f_2(y,\omega)$ satisfies  the boundary conditions given in (\ref{SPNN2}).
Let us prove that $f_2$ satisfies the second equality of (\ref{SPNN2}).
%%We have to prove that
%%$$
%%\dfrac{\partial\ }{\partial  y_2}\ f_2(y,\omega) \Big|_{y\in Q_1} = 0
%%$$
Since
$\dfrac{\partial}{\partial y_2 }\Big|_{\theta=2\pi}=\dfrac{1}{\rho}\cdot \dfrac{\partial}{\partial \theta }\Big|_{\theta=2\pi}$
it suffices to prove that
\begin{equation}\label{sin}
\dfrac{\partial}{\partial \theta }\ f_2(\rho,\theta,\omega) \Big|_{\theta=2\pi} = 0.
\end{equation}
Since
$\dfrac{\partial}{\partial \theta } = i\ \dfrac{\mathrm{d}\ }{\mathrm{d}\beta}$ and
the integral
$\displaystyle\int\limits_{\mathcal{C}} e^{-\rho\omega\sinh\beta}\dfrac{\partial}{\partial \theta} H_N(\beta+i\theta) \ \mathrm{d} \beta$ converges uniformly with respect to $\theta$ (by  (\ref{C}), $\omega\in\mathbb{C}^+$ and (\ref{HN}))
%% (\ref{sin}) is equivalent to
%%\begin{equation}\label{ns3}
%%\int\limits_{\mathcal{C}} e^{-\rho\omega\sinh\beta}\dfrac{\partial}{\partial \theta} H_N(\beta+i\theta) \ \mathrm{d} \beta = 0, \qquad \theta = 2\pi.
%%\end{equation}
%%Since
to prove (\ref{sin}) it suffices to prove that
\begin{equation}\label{Hcomp2eq}
\int\limits_{\mathcal{C}}e^{-\rho\omega\sinh\beta}\dfrac{\mathrm{d}\ }{\mathrm{d}\beta} H_N(\beta+2i\pi)\ \mathrm{d}\beta=0.
\end{equation}
The function $e^{-\rho\omega\sinh\beta}\dfrac{\mathrm{d}\ }{\mathrm{d}\beta} H_N(\beta+2i\pi)$ is invariant with respect to $\mathcal{S}(\beta)=-\beta-3i\pi$,  for any $\beta\in\mathcal{C}$. It follows from
Lemma \ref{propHN}, i) and the fact that $\dfrac{\mathrm{d}\ }{\mathrm{d}\beta}\ H_N(\beta+2i\pi)$ is invariant with respect to $\mathcal{S}(\beta)$ for any $\beta\in\mathcal{C}$. Hence (\ref{Hcomp2eq}) holds by the symmetry of $\mathcal{C}$ given in (\ref{C}) with respect to $-i\frac{3\pi}{2}$, see Figure 4.

Let us prove that $f_2$ satisfies the third equality of (\ref{SPNN2}).
%%We have to prove that
%%\begin{equation}\label{sint}
%%\dfrac{\partial\ }{\partial  \mathbf{n_2}} \ f_2(y,\omega) \Big|_{y\in Q_2} = 0
%%\end{equation}
Since (\ref{f2}), $\dfrac{\partial}{\partial \mathbf{n_2} }\Big|_{\theta=\phi}  = -\dfrac{1}{\rho}\cdot \dfrac{\partial}{\partial \theta }\Big|_{\theta=\phi}$,
$\dfrac{\partial}{\partial \theta } = i\ \dfrac{\mathrm{d}\ }{\mathrm{d}\beta}$,
and the fact that the integral
$\displaystyle\int\limits_{\mathcal{C}} e^{-\rho\omega\sinh\beta}\dfrac{\partial}{\partial \theta} H_N(\beta+i\theta) \ \mathrm{d} \beta$ converges uniformly with respect to $\theta$ (by (\ref{C}), $\omega\in\mathbb{C}^+$ and (\ref{HN})), %to prove (\ref{sint})
it suffices to prove that
%%%\begin{equation}\label{ns4}
%%%\int\limits_{\mathcal{C}} e^{-\rho\omega\sinh\beta} \dfrac{\partial}{\partial \theta}  H_N(\beta+i\theta) \ \mathrm{d} \beta = 0, \qquad \theta = \phi.
%%%\end{equation}
%%%Since
%%% to prove (\ref{ns4}) it suffices to prove that
%\begin{equation}\label{Hcomp3eq}
$\displaystyle\int\limits_{\mathcal{C}} e^{-\rho\omega\sinh\beta} \dfrac{\mathrm{d}\ }{\mathrm{d}\beta} H_N(\beta+i\phi) \ \mathrm{d} \beta = 0$.
%\end{equation}
Integrating by parts %the integral in (\ref{Hcomp3eq})
and  using that  for any $\beta\in\mathcal{C}$, $H_N(\beta+i\phi)$ is a bounded function and $e^{-\rho \omega \sinh\beta}\longrightarrow 0$  when $\left|\mathrm{Re}\ \beta\right|\rightarrow+\infty$ we obtain
%%%%%\\
%%%%%
%%%%%\noindent
%%%%%$\displaystyle\int\limits_{\mathcal{C}} e^{-\rho\omega\sinh\beta} \dfrac{\mathrm{d}\ }{\mathrm{d}\beta} H_N(\beta+i\phi) \ \mathrm{d} \beta $
$\displaystyle\int\limits_{\mathcal{C}} e^{-\rho\omega\sinh\beta} \dfrac{\mathrm{d}\ }{\mathrm{d}\beta} H_N(\beta+i\phi) \ \mathrm{d} \beta
= \rho\omega \int\limits_{\mathcal{C}}  e^{-\rho \omega \sinh\beta} \cosh\beta\ H_N(\beta+i\phi)\  \mathrm{d}\beta$.
%%\begin{eqnarray*}
%%\displaystyle\int\limits_{\mathcal{C}} e^{-\rho\omega\sinh\beta} \dfrac{\mathrm{d}\ }{\mathrm{d}\beta} H_N(\beta+i\phi) \ \mathrm{d} \beta
%%%		&=&  e^{-\rho\omega \sinh\beta} H_N(\beta+i\phi) \Big|_\mathcal{C}
%%%			   + \rho\omega \int\limits_{\mathcal{C}}  e^{-\rho \omega \sinh\beta} \cosh\beta\ H_N(\beta+i\phi)  \ \mathrm{d}\beta \\
%%    &=& \rho\omega \int\limits_{\mathcal{C}}  e^{-\rho \omega \sinh\beta} \cosh\beta\ H_N(\beta+i\phi)\  \mathrm{d}\beta.
%%\end{eqnarray*}
The last integral is equal to 0, because of the invariance of the integrand with respect to to $-\beta-3i\pi$,  for any $\beta\in\mathcal{C}$ and by the symmetry of $\mathcal{C}$ with respect to  $-i\frac{3\pi}{2}$.
%Since $e^{-\rho\omega\sinh\beta}\cosh\beta \ H_N(\beta+i\phi)$ is invariant with respect to $\mathcal{S}(\beta)=-\beta-3i\pi$,  for any $\beta\in\mathcal{C}$.
% It follows from
%the fact that $\cosh\mu\ H_N(\mu+i\phi)$ is invariant with respect to $\mathcal{S}(\beta)$ for any $\beta\in\mathcal{C}$.
%Hence
\end{proof}

%\vspace{1cm}

%%%%%%%%%%%%%%%%%%%%%%%%%%%%%%%%%%%%%%%%%%%%%%%%%%%%%%%%%%%%%%%%%%%%%%%%%%%%%%%%%%%%%%%%%%%%%%%%
%\subsection{A2}
%%%%%%%%%%%%%%%%%%%%%%%%%%%%%%%%%%%%%%%%%%%%%%%%%%%%%%%%%%%%%%%%%%%%%%%%%%%%%%%%%%%%%%%%%%%%%%%%
%Now we prove Lemma \ref{propg}.
\begin{center}
\textbf{A2}
\end{center}
\begin{proof}[\textbf{Proof of Lemma \ref{propg}}] \textbf{i)} The analytic continuation of $\widehat{g}(\omega_1)$ to $\mathbb{C}^+$ and  (\ref{extg}) follow from the Paley-Wiener type Theorem for convex cones (Theorem I.5.2 in \cite{k}) since $\mathrm{supp} f \subset [0,\infty)$ by (\ref{f}).
The estimate (\ref{cotag}) follows from (\ref{f}) since $\omega_0\in\mathbb{R}$.
%\vspace{0.4cm}

\noindent \textbf{ii)} $f'\in C_0^{\infty}(\mathbb{R})$ since  $\mathrm{supp}(f')$ is a compact set by (\ref{f}) and
$f\in C^{\infty}(\mathbb{R})$. Hence existence of the analytic continuation of $\widehat{g}_1(\omega_1)$ to $\mathbb{C}$ and (\ref{g1AnaR}) follow from the Classic Paley-Wiener Theorem  \cite{Rud}. It is easy to check that $\widehat{g}_1(\omega)=(\omega-\omega_0)\widehat{g}(\omega)$, for any $\omega\in\mathbb{C}^+$ by (\ref{gR}), (\ref{g1}) and the Analytic Continuation Principle. This implies the first identity in (\ref{gg1}). Hence, the second identity in (\ref{gg1}) follows from (\ref{extg}) and (\ref{g1AnaR}). The statement (\ref{gC}) follows from  the first identity in (\ref{gg1}) and (\ref{g1Ana}).
%
%\vspace{0.4cm}
%\noindent \textbf{iii)} The representation (\ref{ng1}) follows from (\ref{ng1}).
%
\end{proof}

%%%%%%%%%%%%%%%%%%%%%%%%%%%%%%%%%%%%%%%%%%%%%%%%%%%%%%%%%%%%%%%%%%%%%%%%%%%%%%%%%%%%%%%%%%%%%%%%
%%%%%%%%%%%%%%%%%%%%%%%%%%%%%%%%%%%%%%%%%%%%%%%%%%%%%%%%%%%%%%%%%%%%%%%%%%%%%%%%%%%%%%%%%%%%%%%%

%%%%%%%%%%%%%%%%%%%%%%%%%%%%%%%%%%%%%%%%%%%%%%%%%%%%%%%%%%%%%%%%%%%%%%%%%%%%%%%%%%%%%%%%%%%%%%%%
%\subsection{A1}
%%%%%%%%%%%%%%%%%%%%%%%%%%%%%%%%%%%%%%%%%%%%%%%%%%%%%%%%%%%%%%%%%%%%%%%%%%%%%%%%%%%%%%%%%%%%%%%%
\begin{center}
\textbf{A3}
\end{center}
\begin{defn}
For a function  $h(s)$, we denote the jump of $h(s)$,  at a point $s=s^*\in\mathbb{R}$ as % $    jump(h,s=s^*)$, that is
%\begin{equation}\label{jump}
%    jump(h,s=s^*)
 $   \mathfrak{J}(h,s^*)
    := \lim\limits_{\varepsilon\rightarrow 0+}   h(s^* + i \varepsilon) -
       \lim\limits_{\varepsilon\rightarrow 0+}   h(s^* - i \varepsilon)$,
%\end{equation}
if the limits exist.
\end{defn}

\begin{lem}\label{lemdiflim} Let $f\in \mathrm{C}_0(\mathbb{R})$ and for $ |\varepsilon|\leq 1$ let
%\begin{eqnarray}
$F(\varepsilon) := \dfrac{1}{2i\pi}\displaystyle\int\limits_{-\infty}^{\infty} f(s) \coth(qs+i\varepsilon)\ \mathrm{d}s$.
%\end{eqnarray}
Then there exist the limits $\lim\limits_{\varepsilon\rightarrow 0+} F(\varepsilon)$, $\lim\limits_{\varepsilon\rightarrow 0-} F(\varepsilon)$ and
%\begin{eqnarray}
%\lim\limits_{\varepsilon\rightarrow 0+} F(\varepsilon) - \lim\limits_{\varepsilon\rightarrow 0-} F(\varepsilon)
$\mathfrak{J}(F,0) = -\dfrac{1}{q}\ f(0)$.
%\end{eqnarray}
\end{lem}

\begin{proof}
%%% Note that
%%%\begin{eqnarray}\label{F12}
%%%F(\varepsilon) &=& F_1(\varepsilon) + F_2(\varepsilon),
%%%\end{eqnarray}
%%%where
%%%\begin{eqnarray}
%%%F_1(\varepsilon) &:=& \dfrac{1}{2i\pi}\int\limits_{-\infty}^{\infty}  \frac{f(s)}{qs+i\varepsilon}\ \mathrm{d}s, \\
%%%F_2(\varepsilon) &:=& \dfrac{1}{2i\pi}\int\limits_{-\infty}^{\infty}
%%%                        f(s)\left[\coth(qs+i\varepsilon)-\dfrac{1}{qs+i\varepsilon}\right]\ \mathrm{d}s.
%%%\end{eqnarray}
%%%Then
%%%\begin{eqnarray*}
%%%\lim\limits_{\varepsilon\rightarrow 0+} F(\varepsilon) - \lim\limits_{\varepsilon\rightarrow 0-} F(\varepsilon)
%%%    &=& \left[\lim\limits_{\varepsilon\rightarrow 0+} F_1(\varepsilon)
%%%            - \lim\limits_{\varepsilon\rightarrow 0-} F_1(\varepsilon)\right]   +
%%%        \left[\lim\limits_{\varepsilon\rightarrow 0+} F_2(\varepsilon)
%%%            - \lim\limits_{\varepsilon\rightarrow 0-} F_2(\varepsilon)\right].   % \\
%%%%    &=& - f(0) +  \lim\limits_{\varepsilon\rightarrow 0+} F_2(\varepsilon)
%%%%               -  \lim\limits_{\varepsilon\rightarrow 0-} F_2(\varepsilon).
%%%\end{eqnarray*}
%%%Note that
%%%\begin{eqnarray*}
%%%\lim\limits_{\varepsilon\rightarrow 0+} F_1(\varepsilon) - \lim\limits_{\varepsilon\rightarrow 0-} F_1(\varepsilon)
%%%&=& -\frac{f(0)}{q}
%%%\end{eqnarray*}
It follows by the Sokhotsky-Plemelj Theorem.
%%%It remains only to prove that
%%%\begin{equation*}\label{limF2}
%%% \lim\limits_{\varepsilon\rightarrow 0+} F_2(\varepsilon) - \lim\limits_{\varepsilon\rightarrow 0-} F_2(\varepsilon)= 0.
%%%\end{equation*}
%%%It follows from the continuity of the function $\varphi(s,\varepsilon) := \coth(qs+i\varepsilon)-\dfrac{1}{qs+i\varepsilon}$.
\end{proof}

\begin{lem}\label{lemjudr} Let $u_r$, $u_d$ be the functions given by (\ref{ur}) and (\ref{ud2}), respectively. Then
\begin{eqnarray}\label{jumpurud}
      \mathfrak{J} \left(  \dfrac{\partial^{(k)}u_r}{\partial \theta^k}, \theta_l\right)
  = - \mathfrak{J} \left(  \dfrac{\partial^{(k)}u_d}{\partial \theta^k}, \theta_l\right),
      \qquad k\in\mathbb{N}_0, \ l=1,2.
\end{eqnarray}
%\begin{eqnarray}\label{jumpderud}
%  jump \left(  \dfrac{\partial^{(k)}u_d}{\partial \theta^k}, \theta=\theta_1\right)
%    &=& (-i)^k \ \dfrac{\partial^{(k)}\ }{\partial \beta^k} A(0,\rho,t),\qquad k\in\mathbb{N}_0.
%\end{eqnarray}
\end{lem}

\begin{proof} We consider the case $\theta=\theta_1$. The case $\theta=\theta_2$ is analyzed similarly.

%\vspace{0.8cm}

\noindent First we find $  \mathfrak{J}\left(\dfrac{\partial^{(k)}u_r}{\partial \theta^k}, \theta_1\right)$.
%\newline
From (\ref{ur}) it follows that
\begin{eqnarray}\label{jumpdur}
  \mathfrak{J}\left(\dfrac{\partial^{(k)}u_r}{\partial \theta^k}, \theta_1\right)
  &=& -\dfrac{\partial^{(k)}\ }{\partial \theta^k} u_{r,1}(\rho,\theta_1,t), \qquad k\in\mathbb{N}_0.
\end{eqnarray}
Using  polar coordinates $y=(\rho\cos\theta,\rho\sin\theta)$  in $u_{r,1}$ and making the change of variable $\beta=-i(\theta-\theta_1)$ we obtain
%\begin{eqnarray*}%\label{jumpderur}
$u_{r,1}(\rho,\theta_1,t)
= e^{-i\omega_0(t-\rho\cos(\theta-\theta_1))}  f(t-\rho\cos(\theta-\theta_1)) %\\
%%&=& e^{-i\omega_0(t-\rho\cosh\beta)}  f(t-\rho\cosh\beta)  \\
=   A(\beta,\rho,t)$,
%\end{eqnarray*}
where
\begin{eqnarray}\label{Anew}
  A(\beta,\rho,t) &:=& e^{-i\omega_0(t-\rho\cosh\beta)}  f(t-\rho\cosh\beta).
\end{eqnarray}
Then
%\begin{eqnarray*}%\label{jumpderur}
$\dfrac{\partial^{(k)}\ }{\partial \theta^k} u_{r,1}(\rho,\theta,t),
= (-i)^k \dfrac{\partial^{(k)}\ }{\partial \beta^k} A(\beta,\rho,t)$, %\qquad
$k\in\mathbb{N}_0$.
%\end{eqnarray*}
Hence by (\ref{jumpdur}) we obtain
\begin{eqnarray}\label{jumpderur}
  \mathfrak{J}\left(\dfrac{\partial^{(k)}u_r}{\partial \theta^k}, \theta_1\right)
%  &=& -\dfrac{\partial^{(k)}\ }{\partial \theta^k} u_{r,1}(\rho,\theta_1,t), k\in\mathbb{N}_0 \\
  &=& -(-i)^k \dfrac{\partial^{(k)}\ }{\partial \beta^k} A(0,\rho,t),\qquad  k\in\mathbb{N}_0.
\end{eqnarray}

%\vspace{0.8cm}

\noindent Now we find $\mathfrak{J}\left(\dfrac{\partial^{(k)} u_d}{\partial\theta^k} ,\theta_1\right)$. From (\ref{ZN}) and (\ref{HN}) it follows that
%\begin{eqnarray*}
$Z_N(\rho,\theta,t)
    = \coth \Big[ q(\beta + i\theta - i\theta_1 ) \Big]
+       \coth \Big[ q(\beta + i\theta-i\theta_2) \Big]
-       \coth \Big[ q(\beta + i\theta-2i\pi-i\alpha) \Big]
-       \coth \Big[ q(\beta + i\theta-i\alpha) \Big]$.
%\end{eqnarray*}
Since function $\coth \Big[ q(\beta + i\theta - i\theta_1 ) \Big]$ is discontinuous in $\theta=\theta_1$, then $Z_N(\rho,\theta,t)$ is also discontinuous in $\theta=\theta_1$. Hence from (\ref{rud}) we have that
\begin{eqnarray}\label{judju1}
\mathfrak{J}\left(\frac{\partial^{(k)} u_d}{\partial\theta^k} ,\theta_1\right) &=&
\mathfrak{J}\left(\frac{\partial^{(k)} u_1}{\partial\theta^k} ,\theta_1\right),
\end{eqnarray}
where
%\begin{eqnarray*}
$u_1(\rho,\theta,t)
    = -\dfrac{q}{2i\pi} \displaystyle\int\limits_{-\infty}^{+\infty}  A(\beta,\rho,t)
        \coth\Big[ q(\beta + i\theta - i\theta_1 ) \Big]  \ \mathrm{d}\beta$
%\end{eqnarray*}
and $A(\beta,\rho)$ is given by (\ref{Anew}).
Using
$  \dfrac{\partial\ }{\partial\theta} \coth\Big[ q(\beta + i\theta - i\theta_1 ) \Big]
= i\dfrac{\partial\ }{\partial\beta } \coth\Big[ q(\beta + i\theta - i\theta_1 ) \Big]$,
integrating by parts $k$ times (for $k=0$ we do not integrate $\frac{\partial^{(k)} u_1}{\partial\theta^k}$), and using that by (\ref{f}) the integration is realized on a compact interval, we obtain
%\begin{eqnarray*}
$\dfrac{\partial^{(k)} }{\partial\theta^k} u_1(\rho,\theta,t)
    =  (-1)^{k-1}\dfrac{i^k q}{2i\pi} \displaystyle\int\limits_{-\infty}^{+\infty}
         \dfrac{\partial^{(k)} }{\partial\beta^k} A(\beta,\rho,t) \cdot
         \coth [ q(\beta + i\theta - i\theta_1) ] \ \mathrm{d}\beta$.
%\end{eqnarray*}
Applying  Lemma \ref{lemdiflim} with $f(\beta)=(-1)^{k-1} i^k q \dfrac{\partial^{(k)} }{\partial\beta^k} A(\beta,\rho,t)$ we obtain
\begin{eqnarray}\label{ju1}
\mathfrak{J}\left(\frac{\partial^{(k)} u_1}{\partial\theta^k} , \theta_1\right)
&=& (-i)^k\ \frac{\partial^{(k)} }{\partial\theta^k} A(0,\rho,t).
\end{eqnarray}
Therefore (\ref{jumpurud}) follows from (\ref{judju1}), (\ref{ju1}) and (\ref{jumpderur}).
\end{proof}

%%%%%%%%%%%%%%%%%%%%%%%%%%%%%%%%%%%%%%%%%%%%%%%%%%%%%%%%%%%%%%%%%%%%%%%%%%%%%%%%%%%%%%%%%%%%%%%%%%
%%%%%%%%%%%%%%%%%%%%%%%%%%%%%%%%%%%%%%%%%%%%%%%%%%%%%%%%%%%%%%%%%%%%%%%%%%%%%%%%%%%%%%%%%%%%%%%%%%

\begin{center}
\textbf{A4}
\end{center}

\begin{lem} The function
$\left[\dfrac{t}{t+\sqrt{t^2 - \rho^2}}\right]^m$
%$A^m(t)$, where $A(t)$ is given by (\ref{danzA})
admits the following asymptotic behavior
\begin{eqnarray}\label{asyB}
%A^m(t)
\left[\dfrac{t}{t+\sqrt{t^2 - \rho^2}}\right]^m
&=& \left(\dfrac{1}{2}\right)^m
      + m\left(\dfrac{1}{2}\right)^{m-1} \dfrac{1}{8}\cdot \dfrac{\rho^2}{t^2}
      + O \left( \dfrac{1}{t^4} \right),  \qquad  t\rightarrow\infty.
\end{eqnarray}
\end{lem}

\end{document}